\documentclass[11pt]{amsart}
\usepackage{amssymb,latexsym,bm}
\usepackage{amsmath,color}
\usepackage{amsthm}
\usepackage{epsfig}
\usepackage{graphicx}
\usepackage{hyperref}
\usepackage{titletoc}
\usepackage{psfrag}
\usepackage{fullpage}
\usepackage{palatino,mathpazo}
\numberwithin{equation}{section}
\newtheorem{theorem}{Theorem}[section]

\newtheorem{lemma}[theorem]{Lemma}
\newtheorem{proposition}[theorem]{Proposition}

\newtheorem{definition}[theorem]{Definition}

\newtheorem{rmk}{Remark}
\allowdisplaybreaks[4]
\usepackage{tikz}
\usepackage{amscd}
\usepackage{makecell}
\usetikzlibrary{matrix}
\usepackage{multirow}
\usepackage{longtable}
\usepackage{cancel}
\usepackage{multirow}
\usepackage[misc]{ifsym}

\def\XXint#1#2#3{{\setbox0=\hbox{$#1{#2#3}{\int}$}
		\vcenter{\hbox{$#2#3$}}\kern-.5\wd0}}

\def\XXint#1#2#3{{\setbox0=\hbox{$#1{#2#3}{\int}$ }
		\vcenter{\hbox{$#2#3$ }}\kern-.6\wd0}}

\usepackage{enumerate}
\usepackage{bm}
\begin{document}
	\title[Blow-up analysis: Affine Toda System]{Affine Toda system of $\mathbf{A}$ and $\mathbf{C}^t$ type: compactness and affine Weyl group}
	
	\author[L.L. Cui]{Leilei Cui}
	\address{Leilei Cui, School of Mathematics and Statistics, Central China Normal University, Wuhan, 430079, P. R. China}
	\email{leileicui@mails.ccnu.edu.cn}
	
	\author[Z.H. Nie]{Zhaohu Nie}
	\address{Zhaohu Nie, Department of Mathematics and Statistics, Utah State University, Logan, UT 84322-3900, USA}
	\email{zhaohu.nie@usu.edu}
	
	\author[W. Yang]{Wen Yang}
	\address{Wen Yang, Wuhan Institute of Physics and Mathematics, Innovation Academy for Precision Measurement Science and Technology, Chinese Academy of Sciences, Wuhan 430071, P. R. China}
	\email{wyang@wipm.ac.cn}
	
	\thanks{The third author is partially supported by National Key R\&D Program of China 2022YFA1006800, NSFC No. 12171456 and NSFC No. 12271369.}
	
	\begin{abstract}
		The local mass is a fundamental quantized information that characterizes the blow-up solution to the Toda system and has a profound relationship with its underlying algebraic structure. In \cite{Lin-Yang-Zhong-2020}, it was observed that the associated Weyl group can be employed to represent this information for the $\mathbf{A}_n$, $\mathbf{B}_n$, $\mathbf{C}_n$ and $\mathbf{G}_2$ type Toda systems. The relationship between the local mass of blow-up solution and the corresponding affine Weyl group is further explored for some affine $\mathbf{B}$ type Toda systems in \cite{Cui-Wei-Yang-Zhang-2022}, where the possible local masses are explicitly expressed in terms of $8$ types. The current work presents a comprehensive study of the general affine $\mathbf{A}$ and $\mathbf{C}^t$ type Toda systems with arbitrary rank. At each stage of the blow-up process (via scaling), we can employ certain elements (known as "set chains") in the corresponding affine Weyl group to measure the variation of local mass. Consequently, we obtain the a priori estimate of the affine $\mathbf{A}$ and $\mathbf{C}^t$ type Toda systems with arbitrary number of singularities.
	\end{abstract}
	\maketitle
	
	\vspace{0.25cm}
	
	
	\section{Introduction and main results}\label{Affine-Toda-Section-1}
	\setcounter{equation}{0}
	The Toda theory has been a significant subject of study in both mathematics and physics, particularly in the fields of field theory and partial differential equations, see \cite{Leznov-Saveliev-1992} for a survey of related works. One particular area of interest is the affine Toda system
	\begin{equation}\label{Affine-Toda-Section-1-Eq-4}
		\Delta u_i+\sum\limits_{j\in I}k_{ij}e^{u_j}=0, \ \text{for} \ i\in I,
	\end{equation}
	where the index set $I=\{1,2,\cdots,n+1\}$ with $n\geq 2$ and $(k_{ij})_{(n+1)\times(n+1)}$ stands for the generalised Cartan matrix associated with an affine Lie algebra. For example, the generalised Cartan matrix of the {\em affine Weyl group} of type $\mathbf{A}_{n}^{(1)}$ is given by (with a little abuse of notations)
	\begin{equation}\label{Affine-A_n^1-Cartan-matrix}
		\mathbf{A}_{n}^{(1)}:=(k_{ij})_{(n+1)\times(n+1)}=
		\begin{pmatrix}
			2      &  -1    &    0   &  \cdots  &  0     &    0    &    -1       \\
			-1     &  2     &    -1  &  \cdots  &  0     &    0    &    0        \\
			0      &  -1    &    2   &  \ddots  &  0     &    0    &    0        \\
			\vdots & \vdots & \ddots &  \ddots  & \ddots & \vdots  &  \vdots     \\
			0      &   0    &    0   &  \ddots  &    2   &   -1    &    0         \\
			0      &   0    &    0   &  \cdots  &    -1  &   2     &    -1        \\
			-1     &   0    &    0   &  \cdots  &    0   &   -1    &    2         \\
		\end{pmatrix}_{(n+1)\times(n+1)}
		.
	\end{equation}
	Equation \eqref{Affine-Toda-Section-1-Eq-4} is a classical integrable system that arises in various branches of mathematics and physics, including two-dimensional statistical mechanics, conformal field theory and mathematical theory. The equation of this type was first introduced by Toda \cite{Toda-1967-1,Toda-1967-2} more than $50$ years ago in the context of lattice vibrations, and has since attracted considerable interests from both pure and applied mathematicians. For a detailed account of the original context, we recommend referring to \cite{Adler-Moerbeke-1990,Braden-Corrigan-Dorey-Sasaki-1990,Corrigan-1999,Fordy-1990,konstant-1979,Toda-2012-3}.
	
	We can express the affine Toda equations for a compact simple Lie group $\mathbf{G}$ of rank $n$ in the following form, see \cite{Bolton-Woodward-1994,konstant-1979},
	\begin{equation}\label{Affine-Toda-Section-1-Eq-1}
		2\frac{\partial^2 U}{\partial z\partial \bar z}+\sum_{j=0}^nc_je^{2\alpha_j(U)}\alpha_j^\sharp=0,
	\end{equation}
	where $U(z,\bar z)$ is a function defined on an open subset of the complex numbers taking values in $i{\bf g}$, (${\bf g}$ stands for the Cartan subalgebra of $\mathbf{G}$), $\{\alpha_j \mid \alpha_{j}\in i{\bf g}^*, 1\leq j\leq n\}$ is a set of positive simple roots of $\mathbf{G}$, $c_0=1$ and $-\alpha_0=\sum_{i=1}^nc_i\alpha_i$ is the highest root. Additionally, for each $j\in\{0,1,\ldots,n\}$, $\alpha_j^\sharp \in i\bf{g}$ represents the dual of $\alpha_j$ under the Killing form $\langle\cdot,\cdot\rangle$, which is defined as follows:
	\begin{equation*}
		\langle \alpha_j^\sharp, v\rangle=\alpha_j(v), \ v\in\mathbf{g}.
	\end{equation*}
	In the special case $\mathbf{G}=\mathrm{SU}(n+1)$ and we write $U=(w_1,\cdots,w_{n+1})$ with $(w_1,\cdots,w_{n+1})\in\mathbb{R}^{n+1}$, equation \eqref{Affine-Toda-Section-1-Eq-1} can be rewritten as
	\begin{equation}\label{Affine-Toda-Section-1-Eq-2}
		2\frac{\partial^2w_i}{\partial z\partial \bar z}+e^{2(w_i-w_{i-1})}-e^{2(w_{i+1}-w_i)}=0, \ \forall \ i\in I, \footnote{Here we set $w_{i}=w_{n+1+i}$ for any $i\in\mathbb{N}\cup\{0\}$.}
	\end{equation}
	where
	\begin{equation}\label{Affine-Toda-Section-1-Eq-5}
		w_{n+1+i}=w_i, \ i\in I, \ \text{satisfying that} \ \sum_{i\in I}w_i=0.
	\end{equation}
	Naturally, we may write \eqref{Affine-Toda-Section-1-Eq-2} as
	\begin{equation}\label{Affine-Toda-Section-1-Eq-6}
		2\frac{\partial^2\alpha_i}{\partial z\partial \bar z}-e^{2\alpha_{i-1}}+2e^{2\alpha_i}-e^{2\alpha_{i+1}}=0, \ \forall \ i=0,1,\cdots,n,
	\end{equation}
	where $\alpha_i(U)=w_{i+1}-w_{i}$. Therefore, \eqref{Affine-Toda-Section-1-Eq-5} holds whenever $\alpha_1,\cdots,\alpha_n$ are the positive simple roots of $\mathrm{SU}(n+1)$ and $\alpha_0=-\sum_{i=1}^n\alpha_i$ is the highest root, then we derive the form as in \eqref{Affine-Toda-Section-1-Eq-4} up to some constant term ($u_i=2\alpha_{i-1}+\log4$, $i\in I$). In fact, equation \eqref{Affine-Toda-Section-1-Eq-4} can be also observed from the theory of super-conformal harmonic maps, see \cite{Bolton-Woodward-1992,Wood-1994}. Indeed, let $\phi: \Omega\rightarrow \mathbb{CP}^n$ be a harmonic map. Then the {\em harmonic sequence} $\phi_n$ associated to $\phi$ consists of harmonic maps $\phi_i: \Omega\rightarrow \mathbb{CP}^n$ with $\phi_0=\phi$, where $\phi_i=[f_i]$ with $f_i$ a locally defined $\mathbb{C}^{n+1}$-valued function verifying that
	\begin{equation}\label{Affine-Toda-Section-1-Eq-7}
		\frac{\partial f_i}{\partial z}=f_{i+1}+\frac{\partial }{\partial z}\log|f_i|^2f_i, \ \frac{\partial f_i}{\partial \bar z}=-\frac{|f_i|^2}{|f_{i-1}|^2}f_{i-1}, \ \langle f_{i+1},f_i \rangle=0.
	\end{equation}
	For a given local coordinate $z$ on $S$, the sequence $\{f_i\}$ can be characterized up to multiplication by a meromorphic function using \eqref{Affine-Toda-Section-1-Eq-7} and $\phi=[f_0]$. The maps $\phi_i$ can be used to determine a complex line sub-bundle $L_i$ of the trivial bundle $S\times \mathbb{C}^{n+1}$, where $L_i$ is defined as the set
	\begin{equation*}
		L_i=\{(x,v)\in \Omega\times\mathbb{C}^{n+1}\mid v\in\phi_i(x)\}.
	\end{equation*}
	Each $L_i$ inherits a holomorphic vector bundle structure from its inclusion in $\Omega\times\mathbb{C}^{n+1}$, and $f_i$ is a local meromorphic section of $L_i$. The integrability condition together with \eqref{Affine-Toda-Section-1-Eq-7} then imply that
	\begin{equation}\label{Affine-Toda-Section-1-Eq-8}
		\frac{\partial^2}{\partial z\partial\bar z}\log|f_i|^2=\frac{|f_{i+1}|^2}{|f_i|^2}-\frac{|f_i|^2}{|f_{i-1}|^2}.
	\end{equation}
	Writing $v_i=\log\frac{|f_i|^2}{|f_{i-1}|^2}$, we see that \eqref{Affine-Toda-Section-1-Eq-8} is equivalent to \eqref{Affine-Toda-Section-1-Eq-6} or \eqref{Affine-Toda-Section-1-Eq-4}. Particularly, when $n=1$ equation \eqref{Affine-Toda-Section-1-Eq-4} can be reduced to the classical sinh-Gordon equation, which arises geometrically in the study of surfaces of constant mean curvature, see \cite{Chern-1981,Wente-1980,Wente-1986}.
	
	The main area of our concern is the singularity case. We may write \eqref{Affine-Toda-Section-1-Eq-4} in a local form as
	\begin{equation}\label{Affine-Toda-system-A_n^1-form}
		\begin{cases}
			\Delta u_i+\sum\limits_{j\in I}k_{ij}e^{u_j}=4\pi\alpha_i\delta_0 \ \text{in} \ B_1(0)\subseteq\mathbb{R}^2, \ \ \forall \ i\in I,\\
			\\
			\sum\limits_{i\in I}u_i\equiv C, \ C \ \text{is a fixed constant},
		\end{cases}
	\end{equation}
	where the index set $I=\{1,2,\cdots,n+1\}$ with $n\geq 2$, $\alpha_i>-1$ for $i\in I$ and $\delta_0$ is the Dirac measure at $0$. The counterpart equation of \eqref{Affine-Toda-system-A_n^1-form} in the classical case reads as
	\begin{equation}\label{Affine-Toda-system-Classical-A_n-form}
		\Delta u_i+\sum\limits_{j=1}^{n}a_{ij}e^{u_j}=4\pi\alpha_i\delta_0 \ \text{in} \ B_1(0)\subseteq\mathbb{R}^2, \ \ \forall \ i=1,\cdots,n,
	\end{equation}
	where $\alpha_{i}>-1$, $i=1,\cdots,n$ and the coefficient matrix $\mathbf{A}_n:=(a_{ij})$ stands for the Cartan matrix associated with the simple Lie algebra $\mathbf{A}_{n}$, i.e.,
	\begin{equation}\label{A_n-Cartan-matrix}
		(a_{ij})_{n\times n}=
		\begin{pmatrix}
			2      &  -1    &    0   &  \cdots  &  0     &    0    &    0        \\
			-1     &  2     &    -1  &  \cdots  &  0     &    0    &    0        \\
			0      &  -1    &    2   &  \ddots  &  0     &    0    &    0        \\
			\vdots & \vdots & \ddots &  \ddots  & \ddots & \vdots  &  \vdots     \\
			0      &   0    &    0   &  \ddots  &    2   &   -1    &    0         \\
			0      &   0    &    0   &  \cdots  &    -1  &   2     &    -1        \\
			0      &   0    &    0   &  \cdots  &    0   &   -1    &    2         \\
		\end{pmatrix}_{n\times n}.
	\end{equation}
	For convenience, we always denote $\mathbf{A}_{\ell}:=(a_{ij})$ by the Cartan matrix of the simple Lie algebra $\mathbf{A}_{\ell}$ with $\ell\geq 1$. \footnote{We call \eqref{Affine-Toda-system-Classical-A_n-form} with $(a_{ij})=\mathbf{A}_{\ell}$ the $\mathbf{A}_{\ell}$ Toda system. Similarly, $\mathbf{C}_{\ell}$ Toda system refers to \eqref{Affine-Toda-system-Classical-A_n-form} with $(a_{ij})=\mathbf{C}_{\ell}$, $\ell\geq 2$, where $\mathbf{C}_{\ell}$ stands for the $\mathbf{C}$ type Cartan matrix of rank $\ell$.}
	
	The Toda system given by equation \eqref{Affine-Toda-system-Classical-A_n-form} is commonly referred to as the $\mathbf{A}_n$ ($\mathrm{SU}(n+1)$) Toda system. The structures of equations \eqref{Affine-Toda-system-A_n^1-form} and \eqref{Affine-Toda-system-Classical-A_n-form} are quite similar, while the structure of the sinh-Gordon equation is closely related to the single Liouville equation. It is well-known that these types of equations exhibit a loss of compactness, which makes studying them a fascinating and natural question. There have been numerous works concerning the blow-up phenomenon in classical single Liouville equation and Toda systems. The pioneering work of Brezis-Merle \cite{Brezis-Merle-1991}, Li-Shafrir \cite{Li-Shafrir-1994}, Li \cite{Li-1999}, Bartolucci-Tarantello \cite{Bartolucci-Tarantello-2002}, Kuo-Lin \cite{Kuo-Lin-2016} and Wei-Zhang \cite{Wei-Zhang-2021,Wei-Zhang-2022,Wei-Zhang-2022-1} has fully understood the blow-up phenomena for the single Liouville equation with singularity. However, when dealing with a system of rank two, i.e., a two-component system, the problem becomes more challenging. In \cite{Jost-Lin-Wang-2006}, the compactness issue was settled when equation \eqref{Affine-Toda-system-Classical-A_n-form} has no singular source and it has been widely studied through a series of works by Lin-Wei-Zhang \cite{Lin-Wei-Zhang-2015}, Lin-Zhang \cite{Lin-Zhang-2016}, Lin-Wei-Yang-Zhang \cite{Lin-Wei-Yang-Zhang-2018} and Lin-Yang \cite{Lin-Yang-2021}. In the aforementioned works, the most essential step is to derive the so-called local mass of the blow-up solution $\mathbf{u}^k=(u^k_1,u^k_2)$. The local mass for $\mathbf{u}^k$ is defined as follows:
	\begin{equation}\label{Affine-Toda-Section-1-Eq-9}
		\sigma_i=\frac{1}{2\pi}\lim_{r\rightarrow 0}\lim_{k\rightarrow+\infty}\int_{B_r(p)}e^{u_i^k(x)}\mathrm{d}x, \ \text{for} \ i=1,2, \ \text{where} \ p \ \text{is the blow-up point}.
	\end{equation}
	For the case with rank $n\geq 3$, the structure of the associated Lie algebra plays a crucial role in determining the quantized information of the blow-up solution. Specifically, in \cite{Lin-Yang-Zhong-2020}, the authors demonstrated that the local mass has a deep connection with the Weyl group of the corresponding Lie algebra and it can be represented implicitly by the permutation map:
	\begin{equation*}
		\begin{small}
			\begin{aligned}
				\sigma_i=2\sum_{\ell=0}^{i-1}\Big(\sum_{j=1}^{f(\ell)}\mu_j-\sum_{j=1}^\ell \mu_j\Big)+2m_i, \ \text{for some} \ m_i\in\mathbb{Z}, \ i=1,\cdots,n,
			\end{aligned}
		\end{small}
	\end{equation*}
	where $f$ is a permutation map from $\{0,1,\cdots,n\}$ to itself. It is known that the group of such permutation maps is isomorphic to the associated {\em Weyl group} of the $\mathbf{A}_n$ type Lie algebra. The essential point of obtaining  the above formula is that for the Lie algebras of types $\mathbf{A}_n$, $\mathbf{B}_n$, $\mathbf{C}_n$ and $\mathbf{G}_2$, the solution is connected to a complex ODE whose coefficients are the $\mathcal{W}$-invariants of the Toda system, as described in \cite{Lin-Wei-Ye-2012}. This naturally raises the question of whether a similar conclusion can be drawn for the affine case. In a recent study, the authors of \cite{Cui-Wei-Yang-Zhang-2022} addressed this question and achieved this objective for the $\mathbf{B}_2^{(1)}$ Toda system
	\begin{equation}\label{Affine-Toda-Section-1-Eq-10}
		\left\{
		\begin{aligned}
			\Delta u_i+\sum\limits_{j=1}^3k_{ij}e^{u_j}&=4\pi\alpha_i\delta_0 \ \text{in} \ B_1(0)\subseteq\mathbb{R}^2, \ \ \forall \ i=1,2,3,\\
			u_1+u_2+2u_3&\equiv0,
		\end{aligned}
		\right.
	\end{equation}
	where $\alpha_i>-1$, $i=1,2,3$ and the coefficient matrix is given by
	\begin{equation}\label{Affine-Toda-Section-1-Eq-11}
		(k_{ij})_{3\times 3}=\left(\begin{matrix}
			1 & 0 &-1\\
			0 & 1 &-1\\
			-\frac12 & -\frac12 & 1
		\end{matrix}\right).
	\end{equation}
	Rather than using the {\em Weyl group}, the corresponding algebraic structure is the {\em affine Weyl group}. Prior to this work, the computation of the local mass of the blow-up solution to \eqref{Affine-Toda-Section-1-Eq-10} without singularity had been carried out in \cite{Liu-Wang-2021}.
	
	The aim of this article is to pursue the same goal for the affine Toda system \eqref{Affine-Toda-system-A_n^1-form} but with arbitrary rank. To state our result, we first introduce the {\em affine Weyl group} for affine Lie algebra $\mathbf{A}_n^{(1)}$, see \cite[Chapter 4]{Shi-1986}. Let $\mathbf{G}_{\mathbf{A}}$ be a group having the following representation
	\begin{equation}\label{Affine-A_n^1-Group-Definition}
		\begin{aligned}
			\mathbf{G}_{\mathbf{A}}=\big\langle \ \mathfrak{R}_1,\mathfrak{R}_2,\cdots,\mathfrak{R}_{n+1} \mid \ & \mathfrak{R}_i^2=e, \ \forall \ i\in I,\\
			& (\mathfrak{R}_i\mathfrak{R}_{j})^{3}=e, \ \text{for} \ |i-j|=1 \ \text{or} \ n, \ \forall \ i,j\in I,\\
			& (\mathfrak{R}_i\mathfrak{R}_{j})^{2}=e, \ \text{for} \ 1<|i-j|<n, \ \forall \ i,j\in I,\\
			& \mathfrak{R}_i\mathfrak{R}_{j}\mathfrak{R}_i=\mathfrak{R}_{j}\mathfrak{R}_{i}\mathfrak{R}_{j}, \ \text{for} \ |i-j|=1 \ \text{or} \ n, \ \forall \ i,j\in I
			\big\rangle.
		\end{aligned}
	\end{equation}
	Here $e$ stands for the identity and $|\mathfrak{R}|$ represents the order of the element $\mathfrak{R}$, i.e., $\mathfrak{R}^{|\mathfrak{R}|}=e$. The group $\mathbf{G}_{\mathbf{A}}$ is one of the infinite irreducible Coxeter groups, called the {\em affine Weyl group} of type $\mathbf{A}_{n}^{(1)}$ (or type $\mathbf{A}$).
	
	Let $\mathbf{u}^k=(u_1^k,\cdots,u_{n+1}^k)$ be a sequence of blow-up solutions to system \eqref{Affine-Toda-system-A_n^1-form} satisfying the conditions
	\begin{equation}\label{Affine-Toda-3Conditions}
		\begin{cases}
			(i): &0 \ \text{is the only blow-up point of }{\bf u}^k \ \text{in} \ B_1(0), \ \text{i.e.},\\ &\max\limits_{i\in I}\sup\limits_{x\in B_1(0)}\{u_i^k(x)-2\alpha_i\log|x|\}\rightarrow+\infty \ \text{and} \ \max\limits_{i\in I}\sup\limits_{K\subseteq B_1(0)\setminus\{0\}, \ K \ \text{is compact}} u_i^k\leq C(K),\\
			(ii): &|u_i^k(x)-u_i^k(y)|\leq C, \ \forall \ x,y \ \text{on} \ \partial B_1(0), \ i\in I,\\
			(iii): &\int_{B_1(0)}e^{u_i^k(x)}\mathrm{d}x\leq C, \ \forall \ i\in I.
		\end{cases}
	\end{equation}
	For this sequence of blow-up solutions, we define the local mass by $\bm{\sigma}=(\sigma_1,\cdots,\sigma_{n+1})$, where
	\begin{equation}\label{Affine-Toda-sigma-Definition}
		\sigma_i=\frac{1}{2\pi}\lim_{r\rightarrow0}\lim_{k\rightarrow+\infty}\int_{B_r(0)}e^{u_i^k(x)}\mathrm{d}x, \ \text{for} \ i\in I.
	\end{equation}
	
	We define the set $\Gamma_{\mathbf{A}}(\bm{\mu})$ for a given $\bm{\mu}=(\mu_1,\cdots,\mu_{n+1})$ and the possible local mass quantity $\bm{\sigma}$ as follows:
	\begin{enumerate}[(1)]
		\item ${\bf0}\in\Gamma_{\mathbf{A}}(\bm{\mu})$.
		
		\item If $\bm{\sigma}=(\sigma_1,\cdots,\sigma_{n+1})\in\Gamma_{\mathbf{A}}(\bm{\mu})$, then $\mathfrak{R}_i\bm{\sigma}\in\Gamma_{\mathbf{A}}(\bm{\mu})$ for any $\mathfrak{R}_i\in \mathbf{G}_{\mathbf{A}}$ satisfying \eqref{Affine-A_n^1-Group-Definition}, where each generator $\mathfrak{R}_i$, $i\in I$ sends $\bm{\sigma}$ to $\mathfrak{R}_i\bm{\sigma}$ with
		\begin{equation*}
			(\mathfrak{R}_i\bm{\sigma})_j
			=\begin{cases}
				2\mu_i-\sum\limits_{t\in I}k_{it}\sigma_t+\sigma_i,\quad &\text{if} \ j=i,\\	
				\sigma_j,\quad &\text{if} \ j\neq i.	
			\end{cases}
		\end{equation*}
	\end{enumerate}
	Here $(k_{ij})$ is defined as in \eqref{Affine-A_n^1-Cartan-matrix}. For any $\bm{\sigma}=(\sigma_1,\cdots,\sigma_{n+1})\in\Gamma_{\mathbf{A}}(\bm{\mu})$, $\sigma_i$ is a degree one polynomial of $\mu_i$ for $i\in I$. Next, we present the first result of this article.
	
	\begin{theorem}\label{Affine-Toda-Section-1-Theorem-1}
		Suppose that $\bm{\sigma}=(\sigma_1,\cdots,\sigma_{n+1})$, defined as in \eqref{Affine-Toda-sigma-Definition}, is the local mass of a sequence of blow-up solutions $\mathbf{u}^{k}=(u^{k}_{1},\cdots,u^{k}_{n+1})$ of system \eqref{Affine-Toda-system-A_n^1-form} satisfying conditions \eqref{Affine-Toda-3Conditions}. Then there exists $\hat{\bm{\sigma}}=(\hat{\sigma}_1,\cdots,\hat{\sigma}_{n+1})\in\Gamma_{\mathbf{A}}(\bm{\mu})$ such that
		\begin{equation*}
			{\sigma}_i=\hat{\sigma}_i+2m_{i}, \ m_i\in\mathbb{Z}, \ \forall \ i\in I.
		\end{equation*}
	\end{theorem}
	
	\begin{rmk}
		The constant appearing in the $(n+2)$-th equation of \eqref{Affine-Toda-system-A_n^1-form} can take any real value without affecting the conclusion of Theorem \ref{Affine-Toda-Section-1-Theorem-1}. Therefore, we may assume, without loss of generality, that the constant $C$ is equal to zero. Using this equation, we observe that the coefficients $\alpha_i$, $i\in I$ must satisfy the relation $\sum_{i\in I}\alpha_i=0$.
	\end{rmk}
	
	As previously mentioned, the group structure of the corresponding Lie algebra is known to play a crucial role in determining the local mass of blow-up solutions. Our findings in the current work and the result in \cite{Cui-Wei-Yang-Zhang-2022} reveal that a similar conclusion can be drawn for affine Toda systems, such as the sinh-Gordon equation, the $\mathbf{B}_2^{(1)}$ Toda system and affine $\mathbf{A}$ (and $\mathbf{C}^t$) type Toda system with arbitrary rank. However, there are significant differences between the classical Toda system and the affine case. In the affine case, the counterpart to the Weyl group is the {\em affine Weyl group}, which contains an infinite number of elements, making it impossible to express the local mass in a concrete way. Instead, we must rely on an abstract representation through the use of generators. Additionally, the process of obtaining a clear transformation in each blow-up combination differs between the classical and affine Toda systems. In the classical case, the permutation map is used, while exhaustively exploring all possibilities is sufficient for the affine case with low rank. Nevertheless, for the affine case with arbitrary rank, neither of these methods seems to work. Instead, we must represent each transformation of the local mass by the composition of generators. The issue lies in expressing the $\bm{\sigma}$ (where the $i$-th element $\sigma_i=2\sum_{\ell=0}^{i-1}\left(\sum_{j=1}^{f(\ell)}\mu_j-\sum_{j=1}^\ell \mu_j\right)$ and $f$ denotes the reverse map $f(i)=n-i$, $0\leq i\leq n$) in an equivalent form through the composition of generators acting on the starting point ${\bf0}$. We should remark that there may be multiple ways to express $\bm{\sigma}$ using the composition of generators, obtaining a specific solution is not a straightforward task. To tackle this issue, we propose an indirect expression using induction, which is further explained in Section \ref{Affine-Toda-Section-2+} and Section \ref{Affine-Toda-Section-4}.
	
	It is well known that all the simple Lie algebras are $\mathbf{A}_n$, $\mathbf{B}_n$, $\mathbf{C}_n$, $\mathbf{D}_n$, $\mathbf{F}_4$, $\mathbf{E}_6$, $\mathbf{E}_7$, $\mathbf{E}_8$ and $\mathbf{G}_2$. Furthermore, each of these algebras has some corresponding affine Lie algebras, see \cite{Carter-2005}. The second author of the current article in \cite{Nie-2014,Nie-2016} has demonstrated that the Toda systems of $\mathbf{B}$ and $\mathbf{C}$ type can be inserted into $\mathbf{A}$ type with a higher rank under certain symmetries. The corresponding Cartan matrices for these two types are represented respectively by
	\begin{equation}\label{cartan-m}
		\begin{aligned}
			\mathbf{B}_n=\left(\begin{matrix}
				2 & -1 & 0 & \cdots & 0 \\
				-1 & 2 & -1 & \cdots & 0 \\
				\vdots & \vdots &\vdots  & \ddots & \vdots\\
				0 & \cdots & -1 & 2 & -2 \\
				0 & \cdots &  0 & -1 & 2
			\end{matrix}\right), \ \
			\mathbf{C}_n=\left(\begin{matrix}
				2 & -1 & 0 & \cdots & 0 \\
				-1 & 2 & -1 & \cdots & 0 \\
				\vdots & \vdots &\vdots  & \ddots & \vdots\\
				0 & \cdots & -1 & 2 & -1 \\
				0 & \cdots &  0 & -2 & 2
			\end{matrix}\right).
		\end{aligned}
	\end{equation}
	Similar to the classical case, it can be observed that the affine $\mathbf{C}^t$ type Toda system can be obtained from the affine $\mathbf{A}$ type. Specifically, the affine $\mathbf{C}_n^{(1),t}$ (or $\mathbf{C}^t$) type Toda system is given as follows:
	\begin{equation}\label{Affine-Toda-system-C_n^1-form}
		\begin{cases}
			\Delta u_i+\sum\limits_{j\in I}k_{ij}^{\mathbf{c}^t}e^{u_j}=4\pi\alpha_i\delta_0 \ \text{in} \ B_1(0)\subseteq\mathbb{R}^2, \ \ \forall \ i\in I,\\
			\\
			u_1+2\sum\limits_{i\in I\setminus\{1,n+1\}}u_i+u_{n+1}\equiv0,
		\end{cases}
	\end{equation}
	where $\alpha_i>-1$ for $i\in I$ and $(k_{ij}^{\mathbf{c}^{t}})_{(n+1)\times(n+1)}$ refers to the Cartan matrix of the affine $\mathbf{C}^t$ type Lie algebra, i.e.,
	\begin{equation}\label{Affine-C_n^1-Cartan-matrix}
		(k_{ij}^{\mathbf{c}^{t}})_{(n+1)\times(n+1)}:=
		\begin{pmatrix}
			2      &  -2    &    0   &  \cdots  &  0     &    0    &    0        \\
			-1     &  2     &    -1  &  \cdots  &  0     &    0    &    0        \\
			0      &  -1    &    2   &  \ddots  &  0     &    0    &    0        \\
			\vdots & \vdots & \ddots &  \ddots  & \ddots & \vdots  &  \vdots     \\
			0      &   0    &    0   &  \ddots  &    2   &   -1    &    0         \\
			0      &   0    &    0   &  \cdots  &    -1  &   2     &    -1        \\
			0      &   0    &    0   &  \cdots  &    0   &   -2    &    2         \\
		\end{pmatrix}_{(n+1)\times(n+1)}.
	\end{equation}
	To ensure the compatibility of system \eqref{Affine-Toda-system-C_n^1-form}, it must hold that
	\begin{equation*}
		\alpha_1+2\sum\limits_{i\in I\setminus\{1,n+1\}}\alpha_i+\alpha_{n+1}=0.
	\end{equation*}
	We can deduce \eqref{Affine-Toda-system-C_n^1-form} from an $\mathbf{A}_{2n-1}^{(1)}$ Toda system by imposing the following symmetry
	\begin{equation*}
		\tilde{u}_{i}=\tilde{u}_{2n+2-i}=u_{i}, \ \tilde{\alpha}_{i}=\tilde{\alpha}_{2n+2-i}=\alpha_i, \ \text{for} \ i\in I.
	\end{equation*}
	
	Now we shall define the {\em affine Weyl group} for the $\mathbf{C}^{t}$ type affine Lie algebra. It is noteworthy that this group is isomorphic to the one associated with the affine Lie algebra of $\mathbf{C}$ type. This is due to the fact that the corresponding Cartan matrices are transpose of each other, see \cite{Carter-2005}. Let group $\mathbf{G}_{\mathbf{C}^t}$ be defined by (see \cite{Albar-Al-Hamed-2000})
	\begin{equation}\label{Affine-C_n^1-Group-Definition}
		\begin{aligned}
			\mathbf{G}_{\mathbf{C}^t}=\big\langle \ \mathfrak{R}_1,\mathfrak{R}_2,\cdots,\mathfrak{R}_{n+1} \mid \ &\mathfrak{R}_i^2=e, \ \forall \ i\in I,\\
			& \mathfrak{R}_i\mathfrak{R}_{j}=\mathfrak{R}_{j}\mathfrak{R}_{i}, \ \text{for} \ 1<|i-j|\leq n, \ \forall \ i,j\in I,\\
			& \mathfrak{R}_i\mathfrak{R}_{j}\mathfrak{R}_i=\mathfrak{R}_{j}\mathfrak{R}_{i}\mathfrak{R}_{j}, \ \text{for} \ |i-j|=1, \ \forall \ 2\leq i,j\leq n,\\
			&(\mathfrak{R}_{2}\mathfrak{R}_{1})^{4}=(\mathfrak{R}_{n}\mathfrak{R}_{n+1})^{4}=e
			\big\rangle.
		\end{aligned}
	\end{equation}
	Group $\mathbf{G}_{\mathbf{C}^t}$ is called the {\em affine Weyl group} of type $\mathbf{C}_{n}^{(1),t}$ (or type $\mathbf{C}^t$). Define the set $\Gamma_{\mathbf{C}^t}(\bm{\mu})$ for a given $\bm{\mu}=(\mu_1,\cdots,\mu_{n+1})$ and the possible local mass quantity $\bm{\sigma}$ as follows:
	\begin{enumerate}[(1)]
		\item ${\bf0}\in\Gamma_{\mathbf{C}^t}(\bm{\mu})$.
		
		\item If $\bm{\sigma}=(\sigma_1,\cdots,\sigma_{n+1})\in\Gamma_{\mathbf{C}^t}(\bm{\mu})$, then $\mathfrak{R}_i\bm{\sigma}\in\Gamma_{\mathbf{C}^t}(\bm{\mu})$ for any $\mathfrak{R}_{i}\in\mathbf{G}_{\mathbf{C}^t}$ satisfying \eqref{Affine-C_n^1-Group-Definition}, where each generator $\mathfrak{R}_i$, $i\in I$ sends $\bm{\sigma}$ to $\mathfrak{R}_i\bm{\sigma}$ with
		\begin{equation*}
			(\mathfrak{R}_i\bm{\sigma})_j
			=\begin{cases}
				2\mu_i-\sum\limits_{t\in I}k_{it}\sigma_t+\sigma_i,\quad &\text{if} \ j=i,\\	
				\sigma_j,\quad &\text{if} \ j\neq i.	
			\end{cases}
		\end{equation*}
	\end{enumerate}
	Here $(k_{ij}^{\mathbf{c}^{t}})_{(n+1)\times(n+1)}$ is defined as in \eqref{Affine-C_n^1-Cartan-matrix}. For any $\bm{\sigma}=(\sigma_1,\cdots,\sigma_{n+1})\in\Gamma_{\mathbf{C}^t}(\bm{\mu})$, $\sigma_i$ is a degree one polynomial of $\mu_i$ for $i\in I$.
\medskip

	The counterpart result for system \eqref{Affine-Toda-system-C_n^1-form} is presented as follows.
	
	\begin{theorem}\label{Affine-Toda-Section-1-Theorem-2}
		Suppose that $\bm{\sigma}=(\sigma_1,\cdots,\sigma_{n+1})$, defined as in \eqref{Affine-Toda-sigma-Definition}, is the local mass of a sequence of blow-up solutions $\mathbf{u}^{k}=(u^{k}_{1},\cdots,u^{k}_{n+1})$ of system \eqref{Affine-Toda-system-C_n^1-form} satisfying conditions \eqref{Affine-Toda-3Conditions}. Then there exists $\hat{\bm{\sigma}}=(\hat{\sigma}_1,\cdots,\hat{\sigma}_{n+1})\in\Gamma_{\mathbf{C}^t}(\bm{\mu})$ such that
		\begin{equation*}
			{\sigma}_i=\hat{\sigma}_i+2m_{i}, \ m_i\in\mathbb{Z}, \ \forall \ i\in I.
		\end{equation*}
	\end{theorem}
	
	\begin{rmk}
		Unlike the classical case, we cannot arrive at a similar conclusion for the affine $\mathbf{B}$ type Toda system with rank $n\geq 3$. The primary obstacle in doing so is the possible emergence of the $\mathbf{D}$ type Toda system from a subsystem of affine $\mathbf{B}$ type. It remains a formidable challenge to investigate the compactness of $\mathbf{D}$ and Toda systems associated with other exceptional Lie algebras. This is because, in contrast to the $\mathbf{A}$ type, the Fuchsian ODE is replaced by a pseudo-differential equation for other types, which poses a hurdle for proving that the local monodromy is a unitary matrix.
	\end{rmk}
	
	Theorems \ref{Affine-Toda-Section-1-Theorem-1} and \ref{Affine-Toda-Section-1-Theorem-2} find practical applications in establishing the compactness results for the associated normalized equations on manifold. Let $\Delta_g$ be the Beltrami operator ($-\Delta_g\geq 0$) on a closed $2$-dimensional Riemann surface $(M,g)$ and denote $\mathbf{u}^k=(u^k_1,\cdots,u^k_{n+1})$ by a sequence of blow-up solutions to the system
	\begin{equation}\label{Affine-Toda-Section-1-Eq-12}
		\begin{cases}
			\Delta_{g} u^k_{i}+\sum\limits_{j\in I}k_{ij}\rho^k_{j}\left(\frac{h^k_{j}e^{u^k_j}}{\int_{M}h^k_{j}e^{u^k_j}\mathrm{d}V_{g}}
			-\dfrac{1}{|M|}\right)
			=4\pi{\sum\limits_{p\in \mathcal{S}}\alpha_{p,i}\left(\delta_{p}-\dfrac{1}{|M|}\right)} \ \text{in} \ M, \ \ \forall \ i\in I,\\
			\sum\limits_{i\in I}u^k_i\equiv C_k, \ \ \text{where} \ \sup\limits_{k\geq 1}|C_k|\leq C \ \text{for some constant} \ C>0,
		\end{cases}
	\end{equation}
	where $(k_{ij})_{(n+1)\times(n+1)}$ is defined as in \eqref{Affine-A_n^1-Cartan-matrix}, $h^k_i$ ($i\in I$) are positive and smooth functions in $M$ satisfying that $\sup\limits_{k\geq 1}\max\limits_{i\in I}\|h^k_i\|_{C^3(M)}\leq C$, $\mathcal{S}$ is a finite set of $M$, $\alpha_{p,i}>-1$ and $\bm{\rho}^k=(\rho^k_1,\cdots,\rho^k_{n+1})$ is a sequence of constant vectors with nonnegative components such that $\lim\limits_{k\rightarrow+\infty}\bm{\rho}^k=(\rho_1,\cdots,\rho_{n+1})$. Since \eqref{Affine-Toda-Section-1-Eq-12} is invariant under adding any constants, we may consider \eqref{Affine-Toda-Section-1-Eq-12} in the space $\mathbf{H}=(\mathring{H}^1(M))^{n+1}$, where
	\begin{equation*}
		\mathring{H}^1(M)=\left\{u\in H^{1}(M) \big| \int_{M}u\mathrm{d}V_{g}=0\right\}.
	\end{equation*}
	Without loss of generality, one may assume that $C_k=0$ and from which we could get that $\sum_{i\in I}\alpha_{p,i}=0$ for any $p\in \mathcal{S}$.
	
	We decompose $u^k_i(x)=\tilde{u}^k_i(x)-4\pi\sum_{p\in\mathcal{S}}\alpha_{p,i}G(x,p)$, where $G(x,p)$ is the Green function satisfying that
	\begin{equation*}
		\Delta_{g}G(x,p)+\left(\delta_p-\frac{1}{|M|}\right)=0, \ \ \int_MG(x,p)\mathrm{d}V_g=0.
	\end{equation*}
	Then we can rewrite \eqref{Affine-Toda-Section-1-Eq-12} as
	\begin{equation}\label{Affine-Toda-Section-1-Eq-13}
		\begin{cases}
			\Delta_g\tilde{u}_i^k+\sum\limits_{j\in I}k_{ij}\rho^k_{j}\left(\frac{\tilde h^k_{j}e^{\tilde{u}^k_j}}{\int_{M}\tilde{h}^k_{j}e^{\tilde{u}^k_j}\mathrm{d}V_{g}}-\dfrac{1}{|M|}\right)=0 \ \text{in} \ M, \ \ \forall \ i\in I,\\
			\sum\limits_{i\in I}\tilde{u}^k_i\equiv0,
		\end{cases}
	\end{equation}
	where $\tilde{h}_i^k(x)=h^k_i(x)e^{-4\pi\sum_{p\in\mathcal{S}}\alpha_{p,i}G(x,p)}$ for $i\in I$. Let $\mu_{p,i}=1+\alpha_{p,i}$ for $i\in I$ and denote by
	\begin{equation*}
		\Gamma^{\mathbf{A}}_{i}=\left\{2\pi\sum\limits_{p\in R}\sigma_{p,i}+4\pi m_i \ \big| \ (\sigma_{p,1},\cdots,\sigma_{p,n+1})\in\Gamma_{\mathbf{A}}(\mu_{p,1},\cdots,\mu_{p,n+1}), R\subseteq S, m_i\in\mathbb{Z}\right\}.
	\end{equation*}
	Then the compactness result for \eqref{Affine-Toda-Section-1-Eq-13} is established as below
	
	\begin{theorem}\label{Affine-Toda-Section-1-Theorem-3}
		Suppose that $\rho_i\notin\Gamma^{\mathbf{A}}_i$ for each $i\in I$. Then there exists a constant $C>0$ depending on  $\bm{\rho}$ such that for any solution $\tilde{\mathbf{u}}=(\tilde u_1,\cdots,\tilde u_{n+1})\in \mathbf{H}$ of \eqref{Affine-Toda-Section-1-Eq-13},
		\begin{equation*}
			|\tilde{u}_i(x)|\leq C, \ \forall \ x\in {M}, \ i\in I.
		\end{equation*}
	\end{theorem}
	
	Similar conclusion also holds for $\mathbf{C}^{t}$ type affine Toda system. Suppose that $\hat{\mathbf{u}}^k=(\hat{u}_1^k,\cdots,\hat{u}_{n+1}^k)$ is a sequence of blow-up solutions of
	\begin{equation}\label{Affine-Toda-Section-1-Eq-14}
		\begin{cases}
			\Delta_g\hat{u}_i^k+\sum\limits_{j\in I}k_{ij}^{\mathbf{c}^{t}}\rho^k_{j}\left(\frac{\hat h^k_{j}e^{\hat{u}^k_j}}{\int_{M}\hat{h}^k_{j}e^{\hat{u}^k_j}\mathrm{d}V_{g}}-\dfrac{1}{|M|}\right)=0 \ \text{in} \ M, \ \ \forall \ i\in I,\\
			\hat{u}_1^k+2\sum\limits_{i\in I\setminus\{1,n+1\}}\hat{u}_i^k+\hat{u}_{n+1}^k\equiv0,
		\end{cases}
	\end{equation}
	where $\hat{h}_i^k(x)=h^k_i(x)e^{-4\pi\sum_{p\in\mathcal{S}}\alpha_{p,i,c}G(x,p)}$ for $i\in I$ and
	\begin{equation*}
		\alpha_{p,1,c}+2\sum_{i\in I\setminus\{1,n+1\}}\alpha_{p,i,c}+\alpha_{p,n+1,c}=0, \ \forall \ p\in\mathcal{S}.
	\end{equation*}
	We set $\mu_{p,i,c}=1+\alpha_{p,i,c}$ for any $p\in\mathcal{S}$, $i\in I$  and denote by
	\begin{equation*}
		\Gamma^{\mathbf{C}^t}_{i}=\left\{2\pi\sum\limits_{p\in Q}\sigma_{p,i}+4\pi m_i \ \big| \ (\sigma_{p,1},\cdots,\sigma_{p,n+1})\in\Gamma_{\mathbf{C}^t}(\mu_{p,1,c},\cdots,\mu_{p,n+1,c}), Q\subseteq S, m_i\in\mathbb{Z}\right\}.
	\end{equation*}
	Then the following compactness result holds.
	
	\begin{theorem}\label{Affine-Toda-Section-1-Theorem-4}
		Suppose that $\rho_i\notin\Gamma^{\mathbf{C}^t}_i$ for each $i\in I$. Then there exists a constant $C>0$ depending on $\bm{\rho}$ such that for any solution $\hat{\mathbf{u}}=(\hat u_1,\cdots,\hat u_{n+1})\in \mathbf{H}$ of \eqref{Affine-Toda-Section-1-Eq-14},
		\begin{equation*}
			|\hat{u}_i(x)|\leq C, \ \forall \ x\in M, \ i\in I.
		\end{equation*}
	\end{theorem}
	
	\begin{rmk}
		The proof of both Theorem \ref{Affine-Toda-Section-1-Theorem-3} and Theorem \ref{Affine-Toda-Section-1-Theorem-4} is almost identical to that of \cite[Theorem 1.3]{Cui-Wei-Yang-Zhang-2022}. Therefore we omit it in this article to avoid redundancy.
	\end{rmk}
	
	This article is organized as follows. In Section \ref{Affine-Toda-Section-2}, we provide all the necessary preliminaries from analytic side and establish some type of Pohozaev identity for affine $\mathbf{A}$ type Toda system. In Section \ref{Affine-Toda-Section-2+}, we focus on the associated algebraic structure and present all the relevant results from an algebraic viewpoint. In Section \ref{Affine-Toda-Section-3}, we sketch out the proof of Theorem \ref{Affine-Toda-Section-1-Theorem-1}. This part has been explained in full detail in prior works such as \cite{Cui-Wei-Yang-Zhang-2022, Lin-Wei-Yang-Zhang-2018, Lin-Yang-Zhong-2020}. In Section \ref{Affine-Toda-Section-4}, we exhibit the counterpart results for the $\mathbf{C}^{t}$ type affine Toda system.
	
	Notations in this article are standard. For any subset $J\subseteq I$, $|J|$ stands for the cardinality of $J$. Boldface type are customarily employed to express sequences and tuples, such as $\bm{\tau}=\{\tau_k\}$, $\mathbf{s}=\{s_k\}$, $\bm{\mu}=(\mu_1,\cdots,\mu_{n+1})$ and $\bm{\sigma}=(\sigma_1,\cdots,\sigma_{n+1})$. Throughout this article, we do not distinguish sequence and subsequence.
	
	\section{Bubbling analysis and Pohozaev identity}\label{Affine-Toda-Section-2}
	\setcounter{equation}{0}
	In this section, we first analyze the bubbling areas by the standard selection procedure and establish some type of Pohozaev identity for local mass, i.e.,
	\begin{equation}\label{Affine-Toda-Section-2-Eq-1}
		\sum_{i=1}^{n+1}\sigma_i^2-\sum_{i=1}^{n+1}\sigma_i\sigma_{i+1}=2\sum_{i=1}^{n+1}\mu_i\sigma_i. \footnote{Conventionally, we set $\sigma_{\ell}=\sigma_i$ if $\ell\equiv i \ (\mathrm{mod} \ n+1)$. Therefore, $\sigma_{n+2}=\sigma_1$.}
	\end{equation}
	In the following we shall always encounter the system
	\begin{equation}\label{Affine-Toda-Section-2-Eq-37}
		\Delta u_s+\sum\limits_{t\in J}k_{st}^{\prime}e^{u_t}=4\pi\alpha_{s}\delta_{0} \ \text{in} \ \mathbb{R}^2, \ \ \int_{\mathbb{R}^2}e^{u_s(y)}\mathrm{d}y<+\infty, \ \ \forall \ s\in J\subseteq I,
	\end{equation}
	where $(k_{st}^{\prime})_{|J|\times|J|}$ is a submatrix of the generalised Cartan matrix of type $\mathbf{A}$ (the entry $(k_{st}^{\prime})=(k_{st})$) unless otherwise specified.
	
	\begin{definition}[\cite{Lin-Yang-Zhong-2020}]\label{Affine-Toda-Section-2-Definition-1}
		We say that $J\subseteq I$ consists of consecutive indices if $J=\{j,j+1,\cdots,j+l\}$ for some $j\in I$ and $l\in\mathbb{N}\cup\{0\}$.
	\end{definition}
	
	\begin{proposition}\label{Affine-Toda-Section-2-Proposition-2}
		Let $\mathbf{u}^{k}$ be a sequence of solutions of system \eqref{Affine-Toda-system-A_n^1-form} satisfying \eqref{Affine-Toda-3Conditions}. Then there exists a sequence of finite points $\Sigma_{k}:=\{0,x^{k}_{1},\cdots,x^{k}_{m}\}$ (if $0$ is not a singular point, then $0$ can be deleted from $\Sigma_k$) and a sequence of positive numbers $l^{k}_{1},\cdots,l^{k}_{m}$ such that\\
		(1) There exists a constant $C>0$ independent of $k$ such that the Harnack-type inequality holds:
		\begin{equation}\label{Affine-Toda-Section-2-Eq-3}
			\max\limits_{i\in I}\left\{u^k_i(x)+2\log\mathrm{dist}(x,\Sigma_k)\right\}\leq C, \ \forall \ x\in B_{1}(0).
		\end{equation}
		(2) $x^{k}_{j}\rightarrow 0$ and $l^{k}_{j}\rightarrow 0$ as $k\rightarrow +\infty$, $l^{k}_{j}\leq\tau^k_j:=\frac{1}{2}\mathrm{dist}(x^{k}_{j},\Sigma_{k}\setminus\{x^{k}_{j}\})$, $j=1,\cdots,m$. Furthermore, $B(x^{k}_{i},l_{i}^{k})\cap B(x^{k}_{j},l_{j}^{k})=\emptyset$  for $1\leq i,j\leq m$, $i\neq j$.\\
		(3) At each $x_j^k$, $\max\limits_{i\in I}u^k_{i}(x^{k}_{j})=\max\limits_{i\in I}\max\limits_{x\in B(x^k_{j},l^k_j)}u^k_i(x)\rightarrow+\infty$ as $k\rightarrow+\infty$, $j=1,\cdots,m$. Denote by
		\begin{equation*}
			\varepsilon_{j}^{k}:=\exp\left(-\frac{1}{2}\max\limits_{i\in I}u^{k}_{i}(x^{k}_{j})\right), \ j=1,\cdots,m.
		\end{equation*}
		Then $R^k_j:={l^{k}_{j}}/{\varepsilon_{j}^{k}}\rightarrow+\infty$ as $k\rightarrow+\infty$, $j=1,\cdots,m$.\\
		(4) In each $B(x^{k}_{j},l_{j}^{k})$, denote by $v^k_{i}(y):=u^k_i(x^k_j+\varepsilon_{j}^{k}y)+2\log\varepsilon_{j}^{k}$ for $i\in I$. Then one of the following alternatives holds:
		
		\noindent\textbf{(I).} The index set $I$ is decomposed into the form $I=J\cup N$ and $J=J_1\cup\cdots\cup J_{\vartheta}$ for some $\vartheta\in\mathbb{N}$ satisfying that
		\begin{enumerate}[(\text{I}-1).]
			\item $J_1,\cdots,J_{\vartheta}$ and $N$ are pairwise disjoint, $N\neq\emptyset$, $\{1,n+1\}\cap N\neq\emptyset$ and each $J_p$, $p=1,\cdots,\vartheta$, denoted by $J_p=\{i_p,\cdots,i_p+l_p\}$ for some $i_p\in I$ and $l_p\in\mathbb{N}\cup\{0\}$, consists of the maximal consecutive indices;
			\item $v^k_s(y)\rightarrow -\infty$ in $L^{\infty}_{\mathrm{loc}}(\mathbb{R}^2)$ for $s\in N$;
			\item $v^k_s(y)\rightarrow v_s(y)$ in $C_{\mathrm{loc}}^2(\mathbb{R}^2)$ for $s\in J_p$, $p=1,\cdots,\vartheta$, where $v_s(y)$ satisfies the $\mathbf{A}_{l_p+1}$ Toda system \eqref{Affine-Toda-Section-2-Eq-37} with $(k_{st}^{\prime})=(k_{st})$ and $\alpha_{s}=0$, $s,t\in J_p$.
		\end{enumerate}
		\noindent\textbf{(II).} The index set $I$ is decomposed into the form $I=J\cup N$ and $J=J_{0}\cup J_1\cup\cdots\cup J_{\vartheta}$ for some $\vartheta\in\mathbb{N}\cup\{0\}$ satisfying that
		\begin{enumerate}[(\text{II}-1).]
			\item $J_{0},J_1,\cdots,J_{\vartheta}$ and $N$ are pairwise disjoint, $N\neq\emptyset$, $\{1,n+1\}\cap N=\emptyset$, each $J_p$, $p=1,\cdots,\vartheta$, denoted by $J_p=\{i_p,\cdots,i_p+l_p\}$ for some $i_p\in I$ and $l_p\in\mathbb{N}\cup\{0\}$, consists of the maximal consecutive indices and $J_{0}=\{r_2,r_2+1,\cdots,n+1,1,2,\cdots,r_1\}$ with $1\leq r_1<r_2-1\leq n$, which can be rewritten as $J_{0}=\{s_1,s_{2},\cdots,s_{|J_0|}\}$ with
			\begin{equation}\label{Affine-Toda-Section-2-Eq-2}
				\left\{
				\begin{aligned}
					&s_1=r_2\\
					&s_2=r_2+1\\
					&\quad\quad\vdots\\
					&s_{n-r_2+2}=n+1
				\end{aligned}
				\right.
				\quad \text{and} \quad
				\left\{
				\begin{aligned}
					&s_{n-r_2+3}=1\\
					&s_{n-r_2+4}=2\\
					&\qquad\vdots\\
					&s_{|J_0|}=r_1
				\end{aligned}
				\right.
				\ \ ;
			\end{equation}
			\item $v^k_s(y)\rightarrow -\infty$ in $L^{\infty}_{\mathrm{loc}}(\mathbb{R}^2)$ for $s\in N$;
			\item $v^k_s(y)\rightarrow v_s(y)$ in $C_{\mathrm{loc}}^2(\mathbb{R}^2)$ for $s\in J_p$, $p=1,\cdots,\vartheta$, where $v_s(y)$ satisfies the $\mathbf{A}_{l_p+1}$ Toda system \eqref{Affine-Toda-Section-2-Eq-37} with $(k_{st}^{\prime})=(k_{st})$ and $\alpha_{s}=0$, $s,t\in J_p$;
			\item $v^k_{s_i}(y)\rightarrow v_{s_i}(y)$ in $C_{\mathrm{loc}}^2(\mathbb{R}^2)$ for $i=1,2,\cdots,|J_0|$, where $v_{s_i}(y)$ satisfies the $\mathbf{A}_{|J_0|}$ Toda system \eqref{Affine-Toda-Section-2-Eq-37} with $(k_{ij}^{\prime})=(k_{s_i s_j})$ and $\alpha_{s_i}=0$, $i,j\in\{1,\cdots,|J_0|\}$.
		\end{enumerate}
	\end{proposition}
	\begin{proof}
		The selection procedure is quite standard now, one can see the details in \cite{Cui-Wei-Yang-Zhang-2022} or \cite{Lin-Yang-Zhong-2020}.
	\end{proof}
	
	\begin{proposition}\label{Affine-Toda-Section-2-Proposition-3}
		Suppose that $\mathbf{u}^{k}$ is a sequence of blow-up solutions to system \eqref{Affine-Toda-system-A_n^1-form} satisfying \eqref{Affine-Toda-3Conditions}. Denote the set of "bad points" by $\Sigma_{k}=\{0,x^{k}_{1},\cdots,x^{k}_{m}\}$ (if $0$ is not a singular point, then $0$ can be deleted from $\Sigma_k$). Then for any $x_{0}\in B_{1}(0)\setminus\Sigma_{k}$, there exists a constant $C_{0}$ independent of $x_{0}$ and $k$ such that
		\begin{equation*}
			|u^{k}_{i}(x_{1})-u^{k}_{i}(x_{2})|\leq C_{0}, \ \forall \ x_{1},x_{2}\in B(x_{0},\mathrm{dist}(x_{0},\Sigma_{k})/2), \ \forall \ i\in I.
		\end{equation*}
	\end{proposition}
	\begin{proof}
		We refer the readers to \cite[Lemma 2.1]{Lin-Wei-Zhang-2015} and \cite[Lemma 2.2]{Liu-Wang-2021} for the proof.
	\end{proof}
	
	Let $x^k_l\in\Sigma_k$ and $\tau^k_l$ be given as in Proposition \ref{Affine-Toda-Section-2-Proposition-2}, we derive from system \eqref{Affine-Toda-system-A_n^1-form} and Proposition \ref{Affine-Toda-Section-2-Proposition-3} that $u^k_i(x)=\overline{u}^k_{x^k_l,i}(r)+O(1)$, $x\in B(x^k_l,\tau^k_l)$ for $i\in I$, where $O(1)$ is independent of $r$ and $k$, $r=|x-x^k_l|$ and $\overline{u}^k_{x^k_l,i}(r)=\frac{1}{2\pi r}\int_{\partial B(x^k_l,r)}u^k_i\mathrm{d}S$.
	
	\begin{definition}[\cite{Cui-Wei-Yang-Zhang-2022,Lin-Wei-Zhang-2015,Lin-Yang-Zhong-2020,Lin-Zhang-2016}]\label{Affine-Toda-Section-2-Definition-4}
		\noindent (i) We say $u^{k}_{i}$ has fast decay on $\partial B(x_{k},r_{k})$ if
		\begin{equation*}
			u^{k}_{i}(x)+2\log|x-x_{k}|\leq-N_{k}, \ \forall \ x\in \partial B(x_{k},r_{k}),
		\end{equation*}
		for some $N_{k}\rightarrow+\infty$ as $k\rightarrow+\infty$.\\
		\noindent (ii) We say $u^{k}_{i}$ has slow decay on $\partial B(x_{k},r_{k})$ if
		\begin{equation*}
			u^{k}_{i}(x)+2\log|x-x_{k}|\geq -C, \ \forall \ x\in \partial B(x_{k},r_{k}),
		\end{equation*}
		for some constant $C>0$ which is independent of $k$.
	\end{definition}
	
	Before giving the next result, we present the exact form for the inverse of the Cartan matrix of $\mathbf{A}_n$ (given in \eqref{A_n-Cartan-matrix}), see \cite{Hyder-Wei-Yang-2020},
	\begin{equation*}
		(\mathbf{A}_n^{-1})_{ij}:=a^{ij}=\frac{\min\{i,j\}(n+1-\max\{i,j\})}{n+1}, \ 1\leq i,j\leq n.
	\end{equation*}
	By direct computation one can easily check that
	\begin{equation}\label{Affine-Toda-Section-2-Eq-5}
		\left\{
		\begin{aligned}
			&a^{i1}+a^{in}=1, \ a^{1j}+a^{nj}=1,\\
			&a^{ij}=a^{ji}, \ \text{for} \ i\neq j,
		\end{aligned}
		\right.
		\ \ \text{for} \ i,j=1,2,\cdots,n.
	\end{equation}
	In next proposition, we denote by $\sigma_{i}^{k}(r,x_{0}):=\sigma(r,x_{0};u^{k}_{i})$ and $\sigma_{i}^{k}(r):=\sigma(r,0;u^{k}_{i})$ for $i\in I$, where
	\begin{equation}\label{Affine-Toda-Section-2-Eq-39}
		\sigma(r,x_{0};u):=\frac{1}{2\pi}\int_{B_{r}(x_{0})}e^{u(x)}\mathrm{d}x.
	\end{equation}
	Now we are able to establish the Pohozaev identity for local mass of the system \eqref{Affine-Toda-system-A_n^1-form}.
	
	\begin{proposition}\label{Affine-Toda-Section-2-Proposition-5}
		Let $\mathbf{u}^{k}=(u^{k}_{1},\cdots,u^{k}_{n+1})$ be a sequence of solutions of system \eqref{Affine-Toda-system-A_n^1-form} satisfying \eqref{Affine-Toda-3Conditions}. Assume that $\Sigma_{k}^{\prime}\subseteq\Sigma_{k}$ is a subset of $\Sigma_{k}$ with $0\in\Sigma_{k}^{\prime}\subseteq B(x_{k},r_{k})\subseteq B_{1}(0)$, and there holds
		\begin{equation*}
			\mathrm{dist}(\Sigma_{k}^{\prime},\partial B(x_{k},r_{k}))=o(1)\mathrm{dist}(\Sigma_{k}\setminus\Sigma_{k}^{\prime},\partial B(x_{k},r_{k})).
		\end{equation*}
		If $\mathbf{u}^{k}$ has fast decay on $\partial B(x_{k},r_{k})$, then
		\begin{equation}\label{Affine-Toda-Section-2-Eq-6-Pohozaev}
			\sum\limits_{i\in I}\sigma_{i}\Big(\sum\limits_{j\in I}k_{ij}\sigma_{j}-2\mu_{i}\Big)=2\sum\limits_{i\in I}\mu_{i}\sigma_{i},
		\end{equation}
		where $(k_{ij})$ is given in \eqref{Affine-A_n^1-Cartan-matrix}, $\mu_i=1+\alpha_i$ and $\sigma_i=\lim\limits_{k\rightarrow+\infty}\sigma^k_i(r_k,x_k)$ for $i\in I$.
	\end{proposition}
	\begin{proof}
		Multiplying the first $n$ equations of \eqref{Affine-Toda-system-A_n^1-form} by $\mathbf{A}^{-1}_{n}$ and using \eqref{Affine-Toda-Section-2-Eq-5}, we get that
		\begin{equation}\label{Affine-Toda-Section-2-Eq-7}
			\sum\limits_{j=1}^{n}a^{ij}\Delta u^k_j+e^{u^k_i}-e^{u^k_{n+1}}=4\pi\sum\limits_{j=1}^{n}a^{ij}\alpha_{j}\delta_0, \ \text{for} \  i=1,\cdots,n.
		\end{equation}
		Without loss of generality, we assume $x_k=0$. For fixed $\theta\in(0,1)$, let $\Omega=B_{\theta}(0)\setminus B_{\varepsilon}(0)$, where $\varepsilon>0$ is small enough. Then multiplying \eqref{Affine-Toda-Section-2-Eq-7} by $\big(x\cdot\nabla u^{k}_{i}\big)$ and applying the standard Pohozaev identity, we have
		\begin{equation*}
			\begin{aligned}
				&\sum\limits_{j\neq i}a^{ij}\int_{\Omega}\Delta u^k_j\big(x\cdot\nabla u^{k}_{i}\big)\mathrm{d}x+a^{ii}\left(-\frac{1}{2}\int_{\partial\Omega}|\nabla u^{k}_{i}|^2(x\cdot\nu)\mathrm{d}S+\int_{\partial \Omega}\big(x\cdot\nabla u^k_i\big)\big(\nu\cdot\nabla u^k_i\big)\mathrm{d}S\right)\\
				&-2\int_{\Omega}e^{u^k_i}\mathrm{d}x+\int_{\partial\Omega}e^{u^k_i}(x\cdot\nu)\mathrm{d}S-\int_{\Omega}e^{u^k_{n+1}}\big(x\cdot\nabla u^{k}_{i}\big)\mathrm{d}x=0, \ \text{for} \ i=1,\cdots,n,
			\end{aligned}
		\end{equation*}
		where $\nu$ denotes by the unit outer normal vector. Summing the above equations from $i=1$ to $i=n$, we deduce that
		\begin{equation}\label{Affine-Toda-Section-2-Eq-8}
			\begin{aligned}
				\mathcal{P}_1+\mathcal{P}_2+\mathcal{P}_3:=&\sum\limits_{i=1}^{n}\sum\limits_{j\neq i}a^{ij}\int_{\Omega}\Delta u^k_j\big(x\cdot\nabla u^{k}_{i}\big)\mathrm{d}x\\
				&+\sum\limits_{i=1}^{n}a^{ii}\left(-\frac{1}{2}\int_{\partial\Omega}|\nabla u^{k}_{i}|^2(x\cdot\nu)\mathrm{d}S+\int_{\partial \Omega}\big(x\cdot\nabla u^k_i\big)\big(\nu\cdot\nabla u^k_i\big)\mathrm{d}S\right)\\
				&+\sum\limits_{i=1}^{n}\left(-2\int_{\Omega}e^{u^k_i}\mathrm{d}x+\int_{\partial\Omega}e^{u^k_i}(x\cdot\nu)\mathrm{d}S
				-\int_{\Omega}e^{u^k_{n+1}}\big(x\cdot\nabla u^{k}_{i}\big)\mathrm{d}x\right)=0.
			\end{aligned}
		\end{equation}
		We calculate $\mathcal{P}_i$, $i=1,2,3$ term by term. At first, by (iii) of \eqref{Affine-Toda-3Conditions}, one can easily check that, up to a subsequence of $\varepsilon$,
		\begin{equation*}
			\begin{aligned}
				\mathcal{P}_3&=\sum\limits_{i=1}^{n}\left(-2\int_{\Omega}e^{u^k_i}\mathrm{d}x+\int_{\partial\Omega}e^{u^k_i}(x\cdot\nu)\mathrm{d}S
				-\int_{\Omega}e^{u^k_{n+1}}\big(x\cdot\nabla u^{k}_{i}\big)\mathrm{d}x\right)\\
				&=-4\pi\sum\limits_{i\in I}\sigma^k_{i}(\theta)+\sum\limits_{i\in I}\theta\int_{\partial B_{\theta}(0)}e^{u^k_i}\mathrm{d}S+o(1).
			\end{aligned}
		\end{equation*}
		For the second term $\mathcal{P}_2$, direct computation shows that
		\begin{equation*}
			\begin{aligned}
				\mathcal{P}_2=&\sum\limits_{i=1}^{n}a^{ii}\left(-\frac{1}{2}\int_{\partial\Omega}|\nabla u^{k}_{i}|^2(x\cdot\nu)\mathrm{d}S+\int_{\partial \Omega}\big(x\cdot\nabla u^k_i\big)\big(\nu\cdot\nabla u^k_i\big)\mathrm{d}S\right)\\
				=&\sum\limits_{i=1}^{n}a^{ii}\left(-\frac{1}{2}\int_{\partial B_{\theta}(0)}|\nabla u^{k}_{i}|^2(x\cdot\nu)\mathrm{d}S+\int_{\partial B_{\theta}(0)}\big(x\cdot\nabla u^k_i\big)\big(\nu\cdot\nabla u^k_i\big)\mathrm{d}S\right)\\
				&+\sum\limits_{i=1}^{n}a^{ii}\left(-\frac{1}{2}\int_{\partial B_{\varepsilon}(0)}|\nabla u^{k}_{i}|^2(x\cdot\mathbf{n})\mathrm{d}S+\int_{\partial B_{\varepsilon}(0)}\big(x\cdot\nabla u^k_i\big)\big(\mathbf{n}\cdot\nabla u^k_i\big)\mathrm{d}S\right),
			\end{aligned}
		\end{equation*}
		where $\mathbf{n}$ denotes by the unit outer normal vector on $\partial B_{\varepsilon}(0)$. Using integration by parts, we infer that
		\begin{equation*}
			\begin{aligned}
				\mathcal{P}_1=&\sum\limits_{i=1}^{n}\sum\limits_{j\neq i}a^{ij}\int_{\Omega}\Delta u^k_j\big(x\cdot\nabla u^{k}_{i}\big)\mathrm{d}x
				=\sum\limits_{1\leq i\neq j\leq n}a^{ij}\int_{\Omega}\Delta u^k_j\big(x\cdot\nabla u^{k}_{i}\big)\mathrm{d}x\\
				=&\sum\limits_{1\leq i< j\leq n}a^{ij}\int_{\partial B_{\theta}(0)}\left(\big(x\cdot\nabla u^k_i\big)\big(\nu\cdot\nabla u^k_j\big)+\big(x\cdot\nabla u^k_j\big)\big(\nu\cdot\nabla u^k_i\big)-(x\cdot\nu)\big(\nabla u^k_i\cdot\nabla u^k_j\big)\right)\mathrm{d}S\\
				&+\sum\limits_{1\leq i< j\leq n}a^{ij}\int_{\partial B_{\varepsilon}(0)}\left(\big(x\cdot\nabla u^k_i\big)\big(\mathbf{n}\cdot\nabla u^k_j\big)+\big(x\cdot\nabla u^k_j\big)\big(\mathbf{n}\cdot\nabla u^k_i\big)-(x\cdot\mathbf{n})\big(\nabla u^k_i\cdot\nabla u^k_j\big)\right)\mathrm{d}S.
			\end{aligned}
		\end{equation*}
		Applying the same argument of \cite[Lemma 4.1]{Lin-Zhang-2010}, we obtain the following estimate
		\begin{equation}\label{Affine-Toda-Section-2-Eq-9}
			\nabla u^k_i(x)=-2\alpha_ix/|x|^2+o(1) \ \text{near the origin}.
		\end{equation}
		Thus it follows from \eqref{Affine-Toda-Section-2-Eq-9} that
		\begin{equation}\label{Affine-Toda-Section-2-Eq-10}
			\begin{aligned}
				&\sum\limits_{1\leq i< j\leq n}a^{ij}\int_{\partial B_{\varepsilon}(0)}\left(\big(x\cdot\nabla u^k_i\big)\big(\mathbf{n}\cdot\nabla u^k_j\big)+\big(x\cdot\nabla u^k_j\big)\big(\mathbf{n}\cdot\nabla u^k_i\big)-(x\cdot\mathbf{n})\big(\nabla u^k_i\cdot\nabla u^k_j\big)\right)\mathrm{d}S\\
				&\quad=-8\pi\sum\limits_{1\leq i< j\leq n}a^{ij}\alpha_i\alpha_j+o(1)
			\end{aligned}
		\end{equation}
		and
		\begin{equation}\label{Affine-Toda-Section-2-Eq-11}
			\begin{aligned}
				\sum\limits_{i=1}^{n}a^{ii}\left(-\frac{1}{2}\int_{\partial B_{\varepsilon}(0)}|\nabla u^{k}_{i}|^2(x\cdot\mathbf{n})\mathrm{d}S+\int_{\partial B_{\varepsilon}(0)}\big(x\cdot\nabla u^k_i\big)\big(\mathbf{n}\cdot\nabla u^k_i\big)\mathrm{d}S\right)=-4\pi\sum\limits_{i=1}^{n}a^{ii}\alpha_{i}^{2}+o(1).
			\end{aligned}
		\end{equation}
		
		Since the elements of $\mathbf{u}^k$ have fast decay on $\partial B(0,{r_k})$, by the same arguments of the proof of \cite[Proposition 2.4]{Cui-Wei-Yang-Zhang-2022}, we can choose $R^k\rightarrow+\infty$ such that $\mathbf{u}^k$ has fast decay on $\partial B(0,R^kr_k)$ and
		\begin{equation*}
			\sigma^k_i(R^k r_k)=\sigma^k_i(r_k)+o(1), \ \text{for} \ i\in I.
		\end{equation*}
		Letting $\theta=\sqrt{R^k}r_k$, then we have
		\begin{equation*}
			-4\pi\sum\limits_{i\in I}\sigma^k_{i}(\theta)=-4\pi\sum\limits_{i\in I}\sigma^k_{i}(r_k)+o(1) \ \ \text{and} \ \ \sum\limits_{i\in I}\theta\int_{\partial B_{\theta}(0)}e^{u^k_i}\mathrm{d}S=o(1).
		\end{equation*}
		Hence it holds that
		\begin{equation}\label{Affine-Toda-Section-2-Eq-12}
			\mathcal{P}_3=-4\pi\sum\limits_{i\in I}\sigma^k_{i}(r_k)+o(1).
		\end{equation}
		
		Similar to the estimate \eqref{Affine-Toda-Section-2-Eq-9} (by \cite[Lemma 4.1]{Lin-Zhang-2010} and a scaling argument), for $i=1,\cdots,n$, we have
		\begin{equation*}
			\nabla u^k_i(x)=-\frac{x}{|x|^2}\Big(\sum\limits_{t\in I}k_{it}\sigma^k_t(r_k)-2\alpha_i\Big)+\frac{o(1)}{|x|},~x\in\partial B_{\sqrt{R^k}r_k}(0).
		\end{equation*}
		Then  \eqref{Affine-Toda-Section-2-Eq-10} and \eqref{Affine-Toda-Section-2-Eq-11} imply that
		\begin{equation}\label{Affine-Toda-Section-2-Eq-13}
			\mathcal{P}_1=2\pi\sum\limits_{1\leq i< j\leq n}a^{ij}\Big(\sum\limits_{t\in I}k_{it}\sigma^k_t(r_k)-2\alpha_i\Big)
			\Big(\sum\limits_{t\in I}k_{jt}\sigma^k_t(r_k)-2\alpha_{j}\Big)-8\pi\sum\limits_{1\leq i< j\leq n}a^{ij}\alpha_i\alpha_j+o(1)
		\end{equation}
		and
		\begin{equation}\label{Affine-Toda-Section-2-Eq-14}
			\mathcal{P}_2=\pi\sum\limits_{i=1}^{n}a^{ii}\Big(\sum\limits_{t\in I}k_{it}\sigma^k_t(r_k)-2\alpha_i\Big)^2
			-4\pi\sum\limits_{i=1}^{n}a^{ii}\alpha_{i}^{2}+o(1).
		\end{equation}
		Substituting \eqref{Affine-Toda-Section-2-Eq-12}, \eqref{Affine-Toda-Section-2-Eq-13} and \eqref{Affine-Toda-Section-2-Eq-14} into \eqref{Affine-Toda-Section-2-Eq-8}, we conclude that
		\begin{equation*}
			\begin{aligned}
				&2\sum\limits_{1\leq i< j\leq n}a^{ij}\Big(\sum\limits_{t\in I}k_{it}\sigma^k_t(r_k)-2\alpha_i\Big)
				\Big(\sum\limits_{t\in I}k_{jt}\sigma^k_t(r_k)-2\alpha_{j}\Big)-8\sum\limits_{1\leq i< j\leq n}a^{ij}\alpha_i\alpha_j\\
				&+\sum\limits_{i=1}^{n}a^{ii}\Big(\sum\limits_{t\in I}k_{it}\sigma^k_t(r_k)-2\alpha_i\Big)^2
				-4\sum\limits_{i=1}^{n}a^{ii}\alpha_{i}^{2}-4\sum\limits_{i\in I}\sigma^k_{i}(r_k)=o(1).
			\end{aligned}
		\end{equation*}
		As a consequence, by taking $k\rightarrow+\infty$, we have
		\begin{equation}\label{Affine-Toda-Section-6-Eq-1}
			\begin{aligned}
				&\sum\limits_{1\leq i,j\leq n}a^{ij}\Big(\sum\limits_{t\in I}k_{it}\sigma_{t}-2\alpha_i\Big)
				\Big(\sum\limits_{s\in I}k_{js}\sigma_s-2\alpha_{j}\Big)-4\sum\limits_{1\leq i,j\leq n}a^{ij}\alpha_i\alpha_j-4\sum\limits_{i\in I}\sigma_{i}=0.
			\end{aligned}
		\end{equation}
		
		It remains to simplify \eqref{Affine-Toda-Section-6-Eq-1}. We split the left hand side of \eqref{Affine-Toda-Section-6-Eq-1} into four parts, i.e.,
		\begin{equation}\label{Affine-Toda-Section-6-Eq-3}
			\begin{aligned}
				0
				&=\sum\limits_{1\leq i,j\leq n}a^{ij}\sum\limits_{t\in I}k_{it}\sigma_{t}\sum\limits_{s\in I}k_{js}\sigma_s
				-2\sum\limits_{1\leq i,j\leq n}a^{ij}\sum\limits_{t\in I}k_{it}\sigma_{t}\alpha_{j}
				-2\sum\limits_{1\leq i,j\leq n}a^{ij}\sum\limits_{s\in I}k_{js}\sigma_s\alpha_{i}
				-4\sum\limits_{i\in I}\sigma_{i}\\
				&=:\mathcal{Q}_1-\mathcal{Q}_2-\mathcal{Q}_3-\mathcal{Q}_4.
			\end{aligned}
		\end{equation}
		We shall compute $\mathcal{Q}_i$, $i=1,2,3$ term by term. Direct computation shows that
		\begin{equation*}
			\begin{aligned}
				\mathcal{Q}_1
				&=\sum\limits_{i,j=1}^{n}a^{ij}\sum\limits_{t,s=1}^{n}a_{it}a_{js}\sigma_{t}\sigma_s
				+\sum\limits_{i,j=1}^{n}a^{ij}\sum\limits_{t=1}^{n}a_{it}k_{j,{n+1}}\sigma_{t}\sigma_{n+1}\\
				&\quad+\sum\limits_{i,j=1}^{n}a^{ij}\sum\limits_{s=1}^{n}a_{js}k_{i,{n+1}}\sigma_s\sigma_{n+1}
				+\sum\limits_{i,j=1}^{n}a^{ij}k_{i,{n+1}}k_{j,{n+1}}\sigma_{n+1}^2\\
				&=:\mathcal{Q}_{11}+\mathcal{Q}_{12}+\mathcal{Q}_{13}+\mathcal{Q}_{14}.
			\end{aligned}
		\end{equation*}
		By \eqref{Affine-Toda-Section-2-Eq-5}, we see that
		\begin{equation*}
			\begin{aligned}
				\mathcal{Q}_{11}
				&=\sum\limits_{j,t,s=1}^{n}a_{js}\sigma_{t}\sigma_s\delta_{jt}=\sum\limits_{j,s=1}^{n}a_{js}\sigma_{j}\sigma_{s}
				=2\sum\limits_{i=1}^{n}\sigma_i^2-2\sum\limits_{i=1}^{n-1}\sigma_{i}\sigma_{i+1},\\
				\mathcal{Q}_{12}
				&=\sigma_{n+1}\sum\limits_{i,t=1}^{n}a_{it}\sigma_{t}\big(-a^{i1}-a^{in}\big)
				=-\sigma_{n+1}\sum\limits_{i,t=1}^{n}a_{it}\sigma_{t}=-\sigma_{n+1}(\sigma_1+\sigma_n),\\
				\mathcal{Q}_{13}
				&=\sigma_{n+1}\sum\limits_{j,s=1}^{n}a_{js}\sigma_{s}\big(-a^{1j}-a^{nj}\big)
				=-\sigma_{n+1}\sum\limits_{j,s=1}^{n}a_{js}\sigma_{s}=-\sigma_{n+1}(\sigma_1+\sigma_n),\\
				\mathcal{Q}_{14}
				&=\big(a^{11}+a^{1n}+a^{n1}+a^{nn}\big)\sigma_{n+1}^{2}=2\sigma_{n+1}^{2}.
			\end{aligned}
		\end{equation*}
		Hence
		\begin{equation}\label{Affine-Toda-Section-6-Eq-4}
			\begin{aligned}
				\mathcal{Q}_1
				&=2\sum\limits_{i\in I}\sigma_i^2-2\sum\limits_{i=1}^{n-1}\sigma_{i}\sigma_{i+1}-2\sigma_{n}\sigma_{n+1}-2\sigma_{n+1}\sigma_{1}\\
				&=(\sigma_1-\sigma_2)^{2}+(\sigma_2-\sigma_3)^{2}+\cdots+(\sigma_n-\sigma_{n+1})^{2}+(\sigma_{n+1}-\sigma_1)^{2}.
			\end{aligned}
		\end{equation}
		Furthermore, using \eqref{Affine-Toda-Section-2-Eq-5} again we find that
		\begin{equation}\label{Affine-Toda-Section-6-Eq-5}
			\begin{aligned}
				\mathcal{Q}_2
				&=2\sum\limits_{i,j=1}^{n}a^{ij}\sum\limits_{t=1}^{n}a_{it}\sigma_{t}\alpha_{j}+2\sum\limits_{i,j=1}^{n}a^{ij}k_{i,{n+1}}\sigma_{n+1}\alpha_{j}\\
				&=2\sum\limits_{j,t=1}^{n}\sigma_{t}\alpha_{j}\delta_{jt}+2\sigma_{n+1}\sum\limits_{j=1}^{n}\alpha_{j}\big(-a^{1j}-a^{nj}\big)
				=2\sum\limits_{j=1}^{n}\alpha_{j}\sigma_{j}-2\sigma_{n+1}\sum\limits_{j=1}^{n}\alpha_{j}=2\sum\limits_{j\in I}\alpha_{j}\sigma_{j}
			\end{aligned}
		\end{equation}
		and
		\begin{equation}\label{Affine-Toda-Section-6-Eq-6}
			\begin{aligned}
				\mathcal{Q}_3
				&=2\sum\limits_{i,j=1}^{n}a^{ij}\sum\limits_{s=1}^{n}a_{js}\sigma_{s}\alpha_{i}+2\sum\limits_{i,j=1}^{n}a^{ij}k_{j,{n+1}}\sigma_{n+1}\alpha_{i}\\
				&=2\sum\limits_{i,s=1}^{n}\sigma_{s}\alpha_{i}\delta_{is}+2\sigma_{n+1}\sum\limits_{i=1}^{n}\alpha_{i}\big(-a^{i1}-a^{in}\big)=2\sum\limits_{i=1}^{n}\alpha_{i}\sigma_{i}-2\sigma_{n+1}\sum\limits_{i=1}^{n}\alpha_{i}=2\sum\limits_{i\in I}\alpha_{i}\sigma_{i}.
			\end{aligned}
		\end{equation}
		Therefore, substituting \eqref{Affine-Toda-Section-6-Eq-4}, \eqref{Affine-Toda-Section-6-Eq-5} and \eqref{Affine-Toda-Section-6-Eq-6} into \eqref{Affine-Toda-Section-6-Eq-3}, we deduce that
		\begin{equation*}
			(\sigma_1-\sigma_2)^{2}+(\sigma_2-\sigma_3)^{2}+\cdots+(\sigma_n-\sigma_{n+1})^{2}+(\sigma_{n+1}-\sigma_1)^{2}
			=4\sum\limits_{i\in I}\mu_i\sigma_i,
		\end{equation*}
		which is equivalent to \eqref{Affine-Toda-Section-2-Eq-6-Pohozaev}. One can directly check that \eqref{Affine-Toda-Section-2-Eq-1} is equivalent to \eqref{Affine-Toda-Section-2-Eq-6-Pohozaev}. This completes the proof of Proposition \ref{Affine-Toda-Section-2-Proposition-5}.
	\end{proof}
	
	\section{Properties for the set $\Gamma_{\mathbf{A}}(\bm{\mu})$ and group representation}\label{Affine-Toda-Section-2+}
	\setcounter{equation}{0}
	The focus of this section is on examining the set $\Gamma_{\mathbf{A}}(\bm{\mu})$ and demonstrating its equivalence to $\Gamma_N^{\mathbf{A}}(\bm{\mu})$, as introduced in Proposition \ref{Affine-Toda-Section-2-Proposition-9}. The main aim of this section is to derive the group representation for the {\em affine Weyl group} $\mathbf{A}_n^{(1)}$.
	
	\subsection{Properties for the set $\Gamma_{\mathbf{A}}(\bm{\mu})$}
	\begin{proposition}\label{Affine-Toda-Section-2-Proposition-6}
		For each element $\bm{\sigma}\in\Gamma_{\mathbf{A}}(\bm{\mu})$, $\bm{\sigma}$ satisfies the Pohozaev identity \eqref{Affine-Toda-Section-2-Eq-1}.
	\end{proposition}
	\begin{proof}
		Similarly to \cite[Proposition 3.1]{Lin-Yang-Zhong-2020}, one can directly derive it by Vieta's theorem.
	\end{proof}
	
	\begin{definition}[\cite{Cui-Wei-Yang-Zhang-2022}]\label{Affine-Toda-Section-2-Definition-7}
		We say $\widetilde{\mathfrak{R}}=\mathfrak{R}_{i_{s}}\mathfrak{R}_{i_{s-1}}\cdots\mathfrak{R}_{i_{1}}$ ($i_{a}\in I$, $1\leq a\leq s$) is a \textbf{simplest chain} with length $s$ if $\widetilde{\mathfrak{R}}$ can not be reduced to a sub-chain $\mathfrak{R}_{j_{s^{\prime}}}\mathfrak{R}_{j_{s^{\prime}-1}}\cdots\mathfrak{R}_{j_{1}}$ for any $s^{\prime}<s$, $j_{a}\in I$, $1\leq a\leq s^{\prime}$.
	\end{definition}
	
	For any $\bm{\sigma}\in\Gamma_{\mathbf{A}}(\bm{\mu})$, $\bm{\sigma}=\mathfrak{R}_{i_{s}}\mathfrak{R}_{i_{s-1}}\cdots\mathfrak{R}_{i_{1}}{\bf 0}$ for some simplest chain $\widetilde{\mathfrak{R}}$ with length $s$. In this way, we say that $\bm{\sigma}$ is in the $s$-th level.
	
	\begin{proposition}\label{Affine-Toda-Section-2-Proposition-8}
		For any $\bm{\sigma}=(\sigma_1,\cdots,\sigma_{n+1})\in\Gamma_{\mathbf{A}}(\bm{\mu})$, $\sigma_i=2\sum_{j\in I}n_{ij}\mu_j$ for some $n_{ij}\in\mathbb{N}\cup\{0\}$.
	\end{proposition}
	\begin{proof}
		We prove this proposition by induction. By the definitions of $\mathfrak{R}_i$'s ($i\in I$), it is easy to see that the conclusion holds for the elements in the $1$-st level. Suppose that the conclusion holds for all $\bm{\sigma}=(\sigma_1,\cdots,\sigma_{n+1})\in\Gamma_{\mathbf{A}}(\bm{\mu})$ in the $M$-th level ($M\geq 1$), it is enough to prove that it also holds for each $\mathfrak{R}_i\bm{\sigma}$, $i\in I$. Assuming that $\sigma_i=2\sum_{j\in I}n_{ij}\mu_j$, where $n_{ij}\in\mathbb{N}\cup\{0\}$ and $i,j\in I$, the objective is to demonstrate that for any $i_0\in I$, the following inequality holds true:
		\begin{equation}\label{Affine-Toda-Section-2-Eq-38}
			\delta_{{i_0}j}-\sum\limits_{s\in I}k_{i_0,s}n_{s,j}+n_{i_0,j}\geq 0, \ \forall \ j\in I.
		\end{equation}
		Here we set $n_{ij}$ and $k_{ij}$ by $n_{i,j}$ and $k_{i,j}$, respectively, for the sake of clarity.
		
		We substitute the expression of $\bm{\sigma}$ into \eqref{Affine-Toda-Section-2-Eq-1} and get the following equations by comparing the coefficients of $\mu_{j}^2$ for each $j\in I$,
		\begin{equation}\label{Affine-Toda-Section-2-Eq-16}
			\sum\limits_{i\in I}\left(n_{i,{j}}-n_{i+1,{j}}\right)^{2}=2n_{{j},{j}}. \footnote{Here $\sigma_{\ell}=\sigma_i$ if $\ell\equiv i \ (\mathrm{mod} \ n+1)$. Therefore, $\sigma_{n+2}=\sigma_1$ and $\sigma_{0}=\sigma_{n+1}$, $n_{n+2,j}=n_{1,j}$ and $n_{0,j}=n_{n+1,j}$ for any $j\in I$.}
		\end{equation}
		For \eqref{Affine-Toda-Section-2-Eq-16}, we can regard $n_{{i_0},{i_0}}$ as a solution of the quadratic equation
		\begin{equation}\label{Affine-Toda-Section-2-Eq-17}
			2x^2-2\big(1+n_{{i_0}-1,{i_0}}+n_{{i_0}+1,{i_0}}\big)x+n_{{i_0}-1,{i_0}}^2+n_{{i_0}+1,{i_0}}^2+\sum\limits_{i\in I\setminus\{{i_0}-1,{i_0}\}}\big(n_{i,{i_0}}-n_{i+1,{i_0}}\big)^{2}=0.
		\end{equation}
		Then it is easy to see the other solution of \eqref{Affine-Toda-Section-2-Eq-17} satisfies
		\begin{equation}\label{Affine-Toda-Section-2-Eq-18}
			1+n_{{i_0}-1,{i_0}}-n_{{i_0},{i_0}}+n_{{i_0}+1,{i_0}}\geq 0.
		\end{equation}
		Similarly, for \eqref{Affine-Toda-Section-2-Eq-16}, we can also regard $n_{{i_0},j}$, $j\neq i_0$ as a solution of the quadratic equation
		\begin{equation}\label{Affine-Toda-Section-2-Eq-19}
			2x^2-2\big(n_{{i_0}-1,{j}}+n_{{i_0}+1,{j}}\big)x+n_{{i_0}-1,{j}}^2+n_{{i_0}+1,{j}}^2+\sum\limits_{i\in I\setminus\{{i_0}-1,{i_0}\}}\big(n_{i,{j}}-n_{i+1,{j}}\big)^{2}-2n_{jj}=0.
		\end{equation}
		In order to show the other solution is non-negative, we claim that
		\begin{equation}\label{Affine-Toda-Section-2-Eq-20}
			n_{{i_0}-1,{j}}^2+n_{{i_0}+1,{j}}^2+\sum\limits_{i\in I\setminus\{{i_0}-1,{i_0}\}}\big(n_{i,{j}}-n_{i+1,{j}}\big)^{2}-2n_{jj}\geq 0.
		\end{equation}
		We prove \eqref{Affine-Toda-Section-2-Eq-20} by investigating the following 4 cases:
		\begin{equation*}
			\text{(a).} \ j=i_{0}-1; \ \ \text{(b).} \ j=i_{0}+1; \ \ \text{(c).} \ j<i_{0}-1; \ \ \text{(d).} \ j>i_{0}+1.
		\end{equation*}
		Here we just give the proof of cases (a) and (c) since the argument of cases (b) and (d) are quite similar.\\
		\noindent {\em Proof of case (a)}. if $j=i_{0}-1$, then
		\begin{equation*}
			\begin{aligned}
				& n_{{i_0}-1,{j}}^2+n_{{i_0}+1,{j}}^2+\sum\limits_{i\in I\setminus\{{i_0}-1,{i_0}\}}\big(n_{i,{j}}-n_{i+1,{j}}\big)^{2}-2n_{jj}\\
				&\geq n_{{i_0}-1,{i_0}-1}+n_{{i_0}+1,{i_0}-1}+\sum\limits_{i\in I\setminus\{{i_0}-1,{i_0}\}}\big(n_{i+1,{i_0}-1}-n_{i,{i_0}-1}\big)-2n_{{i_0}-1,{i_0}-1}=0.
			\end{aligned}
		\end{equation*}
		\noindent {\em Proof of case (c)}. if $j<i_{0}-1$, then
		\begin{equation*}
			\begin{aligned}
				& n_{{i_0}-1,{j}}^2+n_{{i_0}+1,{j}}^2+\sum\limits_{i\in I\setminus\{{i_0}-1,{i_0}\}}\big(n_{i,{j}}-n_{i+1,{j}}\big)^{2}-2n_{jj}\\
				&\geq\big(n_{j,j}-n_{j-1,{j}}\big)+\big(n_{j,j}-n_{j+1,{j}}\big)
				+\sum\limits_{i=j+1}^{{i_0}-2}\big(n_{i,j}-n_{i+1,{j}}\big)+n_{{i_0}-1,{j}}+n_{{i_0}+1,{j}}\\
				&\quad+\sum\limits_{i\in I\setminus\{j-1,j,\cdots,i_{0}-1,i_{0}\}}\big(n_{i+1,j}-n_{i,{j}}\big)-2n_{jj}=0.
			\end{aligned}
		\end{equation*}
		Hence the claim \eqref{Affine-Toda-Section-2-Eq-20} is proved. Therefore, the other solution of \eqref{Affine-Toda-Section-2-Eq-19} satisfies
		\begin{equation}\label{Affine-Toda-Section-2-Eq-21}
			n_{{i_0}-1,j}-n_{{i_0},j}+n_{{i_0}+1,j}\geq 0.
		\end{equation}
		By \eqref{Affine-Toda-Section-2-Eq-18} and \eqref{Affine-Toda-Section-2-Eq-21} we get \eqref{Affine-Toda-Section-2-Eq-38}. This finishes the proof of Proposition \ref{Affine-Toda-Section-2-Proposition-8}.
	\end{proof}
	
	\begin{proposition}\label{Affine-Toda-Section-2-Proposition-9}
		Let $\Gamma_{N}^{\mathbf{A}}(\bm{\mu})$ be defined as
		\begin{equation*}
			\Gamma_{N}^{\mathbf{A}}(\bm{\mu})=\left\{\bm{\sigma} \ \big| \ \bm{\sigma} \ \text{satisfies the Pohozaev identity \eqref{Affine-Toda-Section-2-Eq-6-Pohozaev}},
			\ \sigma_{i}=2\sum\limits_{j\in I}n_{ij}\mu_{j}, \ n_{ij}\in\mathbb{N}\cup\{0\}, \ i\in I\right\}.
		\end{equation*}
		Then it holds that
		\begin{equation*}
			\Gamma_{N}^{\mathbf{A}}(\bm{\mu})=\Gamma_{\mathbf{A}}(\bm{\mu}).
		\end{equation*}
	\end{proposition}
	\begin{proof}
		By Proposition \ref{Affine-Toda-Section-2-Proposition-6} and Proposition \ref{Affine-Toda-Section-2-Proposition-8}, we find that $\Gamma_{\mathbf{A}}(\bm{\mu})\subseteq\Gamma_{N}^{\mathbf{A}}(\bm{\mu})$. Thus, it remains to show $\Gamma_{N}^{\mathbf{A}}(\bm{\mu})\subseteq\Gamma_{\mathbf{A}}(\bm{\mu})$. For any $\bm{\sigma}=(\sigma_1,\cdots,\sigma_{n+1})\in\Gamma_{N}^{\mathbf{A}}(\bm{\mu})$, by Proposition \ref{Affine-Toda-Section-2-Proposition-6}, each $\mathfrak{R}_i\bm{\sigma}$, $i\in I,$ satisfies the Pohozaev identity \eqref{Affine-Toda-Section-2-Eq-6-Pohozaev}. We claim that
		\begin{equation}\label{Affine-Toda-Section-2-Eq-23}
			\text{if} \ \bm{\sigma}\in\Gamma_{N}^{\mathbf{A}}(\bm{\mu}),\ \text{then} \ \mathfrak{R}_i\bm{\sigma}\in\Gamma_{N}^{\mathbf{A}}(\bm{\mu}),\ \text{for each} \ i\in I.
		\end{equation}
		In fact, for any $\bm{\sigma}\in\Gamma_{N}^{\mathbf{A}}(\bm{\mu})$, we suppose that ${\sigma}_i=2\sum_{j\in I}n_{ij}\mu_{j}$, where $n_{ij}\in\mathbb{N}\cup\{0\}$, $i,j\in I$. By the same arguments of Proposition \ref{Affine-Toda-Section-2-Proposition-8}, we obtain that
		\begin{equation*}
			\delta_{{i}j}-\sum\limits_{s\in I}k_{is}n_{sj}+n_{ij}\geq 0, \ \forall \ i,j\in I,
		\end{equation*}
		Thus \eqref{Affine-Toda-Section-2-Eq-23} is proved.
		
		Then by the same arguments of \cite[Proposition 3.4]{Lin-Yang-Zhong-2020}, we define a partial order $\preceq$ in $\Gamma_{N}^{\mathbf{A}}(\bm{\mu})$, i.e., $\bm{\sigma}_{1}\preceq\bm{\sigma}_{2}$ provided that $(\bm{\sigma}_{1})_{i}\leq(\bm{\sigma}_{2})_{i}$ for each $i\in I$. For any $\bm{\sigma}\in\Gamma_{N}^{\mathbf{A}}(\bm{\mu})$, let
		\begin{equation*}
			\Gamma_{\bm{\sigma}}^{\mathbf{A}}:=\left\{\mathfrak{R}_{i_{1}}\mathfrak{R}_{i_{2}}\cdots\mathfrak{R}_{i_{s}}\bm{\sigma}\mid s\in\mathbb{N}\cup\{0\}, \ i_{a}\in I, \ 1\leq a\leq s\right\}.
		\end{equation*}
		Thus, for any $\bm{\sigma}_{1},\bm{\sigma}_{2}\in\Gamma_{N}^{\mathbf{A}}(\bm{\mu})$, we have that
		\begin{equation}\label{Affine-Toda-Section-2-Eq-24}
			\text{either} \ \Gamma_{\bm{\sigma}_{1}}^{\mathbf{A}}=\Gamma_{\bm{\sigma}_{2}}^{\mathbf{A}} \ \text{or} \ \Gamma_{\bm{\sigma}_{1}}^{\mathbf{A}}\cap\Gamma_{\bm{\sigma}_{2}}^{\mathbf{A}}=\emptyset.
		\end{equation}
		One can easily deduce that $\bm{0}\in\Gamma_{\bm{\sigma}}^{\mathbf{A}}$ for any $\bm{\sigma}\in\Gamma_{N}^{\mathbf{A}}(\bm{\mu})$. By \eqref{Affine-Toda-Section-2-Eq-24}, we conclude that $\Gamma_{\bm{\sigma}}^{\mathbf{A}}=\Gamma_{\bm{0}}^{\mathbf{A}}$ for any $\bm{\sigma}\in\Gamma_{N}^{\mathbf{A}}(\bm{\mu})$. Consequently, $\Gamma_{N}^{\mathbf{A}}(\bm{\mu})=\Gamma_{\mathbf{A}}(\bm{\mu})$. This finishes the proof of Proposition \ref{Affine-Toda-Section-2-Proposition-9}.
	\end{proof}
	
	\subsection{The group representation}
	\begin{lemma}\label{Affine-Toda-Section-2-Lemma-10}
		For any $\bm{\sigma}=(\sigma_1,\cdots,\sigma_{n+1})\in\Gamma_{\mathbf{A}}(\bm{\mu})$, denote by $\mu_{s}^{*}=\mu_{s}-\frac{1}{2}\sum_{t\in I}k_{st}\sigma_{t}$ for $s\in I$. Suppose that $J=\{j,j+1,\cdots,j+l\}\subsetneqq I$ consists of consecutive indices for some $j\in I$ and $l\in\mathbb{N}\cup\{0\}$. Let $\bm{\sigma}^{*}=\left(\sigma^{*}_1,\cdots,\sigma^{*}_{n+1}\right)$ be defined as
		\begin{equation*}
			\left\{
			\begin{aligned}
				\sigma_{s}^{*}&=\sigma_{s}, \ &&\text{for} \ s\in I\setminus J,\\
				\\
				\sigma_{s}^{*}&=\sigma_{s}+2\sum\limits_{t\in J}k^{st}\big(\mu_t^{*}+\mu_{t^{*}}^{*}\big), \ &&\text{for} \ s\in J,
			\end{aligned}
			\right.
		\end{equation*}
		where $t^{*}=2j+|J|-t-1$ and $(k^{st})_{|J|\times|J|}$ stands for the inverse of $(k_{st})_{|J|\times|J|}$, $s,t\in J$. Then $\bm{\sigma}^{*}\in\Gamma_{\mathbf{A}}(\bm{\mu})$.
	\end{lemma}
	
	In order to prove Lemma \ref{Affine-Toda-Section-2-Lemma-10}, we shall identify an element $\mathfrak{R}_{J}\in\mathbf{G}_{\mathbf{A}}$ such that $\bm{\sigma}^{*}=\mathfrak{R}_{J}\bm{\sigma}\in\Gamma_{\mathbf{A}}(\bm{\mu})$. To achieve this, we introduce the concept of a \textbf{\em set chain}.
	
	\begin{definition}\label{Affine-Toda-Section-2-Definition-11}
		Suppose that $J\subsetneqq I$ consists of consecutive indices, denoted by $J=\{j,j+1,\cdots,j+l\}$ for some $j\in I$ and $l\in\mathbb{N}\cup\{0\}$. We define the \textbf{$J$-chain}, denoted by $\mathfrak{R}_{J}$, as follows:
		\begin{enumerate}[(a).]
			\item if $l=0$, $|J|=1$, then $\mathfrak{R}_{J}=\mathfrak{R}_{j}$;
			\item if $l=1$, $|J|=2$, then $\mathfrak{R}_{J}=\mathfrak{R}_{j}\mathfrak{R}_{j+1}\mathfrak{R}_{j}$;
			\item if $l=2$, $|J|=3$, then $\mathfrak{R}_{J}=\left(\mathfrak{R}_{j+1}\mathfrak{R}_{j+2}\mathfrak{R}_{j}\right)^{2}$;
			\item if $l=3$, $|J|=4$, then $\mathfrak{R}_{J}=\left(\mathfrak{R}_{j+1}\mathfrak{R}_{j+2}\mathfrak{R}_{j+3}\mathfrak{R}_{j+1}\mathfrak{R}_{j}\right)^{2}$;
			\item if $l\geq 4$, $|J|=l+1\geq 5$, then
			\begin{equation*}
				\mathfrak{R}_{J}=\mathfrak{R}_{J^{\star}}\left(\mathfrak{R}_{j+1}\mathfrak{R}_{j+2}\cdots\mathfrak{R}_{j+l-1}\mathfrak{R}_{j+l}
				\mathfrak{R}_{j+l-2}\mathfrak{R}_{j+l-3}\cdots\mathfrak{R}_{j+1}\mathfrak{R}_{j}\right)^2,
			\end{equation*}
			where $J^{\star}=\{j+2,\cdots,j+l-2\}$ with $|J^{\star}|=l-3$.
		\end{enumerate}
	\end{definition}
	
	\begin{rmk}
		If $J=\{i_1,i_2,\cdots,i_{l+1}\}\subseteq I$ is not a consecutive-index set for some $l\in\mathbb{N}$, we can also give the corresponding $J$-chain as follows. Defining a bijective map
		\begin{equation*}
			g:J=\{i_1,i_2,\cdots,i_{l+1}\}\rightarrow\{1,2,\cdots,l+1\}=:\widetilde{J}, \ i_{t}\mapsto t, \ t=1,2,\cdots,l+1.
		\end{equation*}
		Assume that the $\widetilde{J}$-chain is given by $\mathfrak{R}_{\widetilde{J}}=\prod_{s\in\Lambda}\mathfrak{R}_{s}$, where $\prod$ stands for the composition of generators $\mathfrak{R}_s$'s and $\Lambda$ is a tuple of length ${(l+1)(l+2)}/{2}$ satisfying that $s\in \widetilde{J}$ for each $s\in\Lambda$. Then, the $J$-chain is defined as $\mathfrak{R}_{J}=\prod_{s\in\Lambda}\mathfrak{R}_{g^{-1}(s)}$.
	\end{rmk}
	
	\begin{rmk}
		In the blow-up process (refer to Section \ref{Affine-Toda-Section-3}), the local mass of any given blow-up sequence undergoes a change at each step. We shall measure the amount of change by giving a representation in terms of the {\bf \em set chain}. To be precise, if $\bm{\sigma}$ stands for the local mass in present step, then the new local mass (after a further blow-up analysis) can be expressed as $\bm{\sigma}^{*}=\mathfrak{R}_{J}\bm{\sigma}$ for some $J\subsetneqq I$.
	\end{rmk}
	
	\begin{rmk}
		\begin{enumerate}[(a)]
			\item By Proposition \ref{Affine-Toda-Section-2-Proposition-6}, for any consecutive-index set $J\subsetneqq I$, $\mathfrak{R}_{J}\bm{\sigma}\in\Gamma_{\mathbf{A}}(\bm{\mu})$ provided that $\bm{\sigma}\in\Gamma_{\mathbf{A}}(\bm{\mu})$.
			\item We give some straightforward results for readers: see Table 1.
			\setlength{\tabcolsep}{-0.3mm}
			\setlength{\intextsep}{5pt plus 2pt minus 2pt}
			\begin{table}[!htbp]
				\centering
				\begin{tabular}{|c|c|c|c|}
					\hline
					\multicolumn{4}{|c|}{$J$-\text{chain for affine} $\mathbf{A}$ \text{type}}\\
					\hline
					$l$ & $|J|$ & $\mathfrak{R}_{J}$ &~ \text{number of elements in} $\mathfrak{R}_{J}~$\\
					\hline
					$0$ &  $1$  & $\mathfrak{R}_{j}$ &  {$1$}\\
					\hline
					$1$ &  $2$  & $\mathfrak{R}_{j}\mathfrak{R}_{j+1}\mathfrak{R}_{j}$ &  {$3$} \\
					\hline
					$2$ &  $3$  & $\left(\mathfrak{R}_{j+1}\mathfrak{R}_{j+2}\mathfrak{R}_{j}\right)^{2}$ & $6$\\
					\hline
					$3$ &  $4$  & $\mathfrak{R}_{J}=\left(\mathfrak{R}_{j+1}\mathfrak{R}_{j+2}\mathfrak{R}_{j+3}\mathfrak{R}_{j+1}\mathfrak{R}_{j}\right)^{2}$ & $10=5\times2+0$\\
					\hline
					$4$ &  $5$  & $\mathfrak{R}_{J}={\mathfrak{R}_{j+2}}\left(\mathfrak{R}_{j+1}\mathfrak{R}_{j+2}\mathfrak{R}_{j+3}\mathfrak{R}_{j+4}
					\mathfrak{R}_{j+2}\mathfrak{R}_{j+1}\mathfrak{R}_{j}\right)^{2}$  & $15=7\times2+ {1}$\\
					\hline
					$5$ & $6$ & ~$~\mathfrak{R}_{J}= {\left(\mathfrak{R}_{j+2}\mathfrak{R}_{j+3}\mathfrak{R}_{j+2}\right)}\left(\mathfrak{R}_{j+1}\mathfrak{R}_{j+2}
					\mathfrak{R}_{j+3}\mathfrak{R}_{j+4}\mathfrak{R}_{j+5}\mathfrak{R}_{j+3}\mathfrak{R}_{j+2}\mathfrak{R}_{j+1}\mathfrak{R}_{j}\right)^{2}~$ & $21=9\times2+ {3}$ \\
					\hline
					$\cdots$ & $\cdots$ & $\cdots$ & $\cdots$
					\\
					\hline
					$l$ & $~l+1~$ & $\mathfrak{R}_{J}=\mathfrak{R}_{J^{\star}}\left(\mathfrak{R}_{j+1}\mathfrak{R}_{j+2}\cdots\mathfrak{R}_{j+l-1}\mathfrak{R}_{j+l}
					\mathfrak{R}_{j+l-2}\mathfrak{R}_{j+l-3}\cdots\mathfrak{R}_{j+1}\mathfrak{R}_{j}\right)^2$  & $\frac{(l+1)(l+2)}{2}$
					\\
					\hline
				\end{tabular}
				\caption{$J$-chain for affine $\mathbf{A}$ type}
			\end{table}
			\vspace{-0.5cm}
		\end{enumerate}
	\end{rmk}
	
	\begin{proof}[Proof of Lemma \ref{Affine-Toda-Section-2-Lemma-10}]
		We claim that
		\begin{equation}\label{Affine-Toda-Section-2-Eq-25}
			\bm{\sigma}^{*}=\left(\sigma^{*}_1,\cdots,\sigma^{*}_{n+1}\right)=\mathfrak{R}_{J}\bm{\sigma}\in\Gamma_{\mathbf{A}}(\bm{\mu}).
		\end{equation}
		We prove claim \eqref{Affine-Toda-Section-2-Eq-25} by induction on index $l$. For the cases of $l\leq 4$, we can show that $\bm{\sigma}^{*}=\mathfrak{R}_{J}\bm{\sigma}$ by direct computation and it proves \eqref{Affine-Toda-Section-2-Eq-25}. So, without loss of generality, we start with $l=4$. Suppose that \eqref{Affine-Toda-Section-2-Eq-25} holds for $l\geq 4$, then it suffices to check that it is also true for $l+1$. We split the proof into several steps.
		
		\noindent\textbf{Step 1.} For $l\geq 5$, $|J|=l+1\geq 6$. We calculate $\mathfrak{R}_{J}\bm{\sigma}$. Inductively, one can easily check that
		\begin{equation*}
			\begin{aligned}
				&\left(\mathfrak{R}_{j+1}\mathfrak{R}_{j+2}\cdots\mathfrak{R}_{j+l-1}\mathfrak{R}_{j+l}
				\mathfrak{R}_{j+l-2}\mathfrak{R}_{j+l-3}\cdots\mathfrak{R}_{j+1}\mathfrak{R}_{j}\bm{\sigma}\right)_{s}\\
				&=
				\left\{
				\begin{aligned}
					&\sigma_{s},              \ && \text{for} \ 1\leq s\leq j-1 \ \text{and} \ j+l+1\leq s\leq n+1,\\
					&2\mu_{s}^{*}+\sigma_{s}, \ && \text{for} \ s=j \ \text{and} \ s=j+l,\\
					&2\sum\limits_{k=j}^{j+l}\mu_{k}^{*}+\sigma_{s}, \ && \text{for} \ j+1\leq s\leq j+l-1.
				\end{aligned}
				\right.
			\end{aligned}
		\end{equation*}
		Direct computation shows that $2\sum_{k=j}^{j+l}\mu_{k}^{*}=2\sum_{k=j}^{j+l}\mu_{k}+\sigma_{j-1}-\sigma_{j}-\sigma_{j+l}+\sigma_{j+l+1}$. Hence we have
		\begin{equation}\label{Affine-Toda-Section-2-Eq-26}
			\begin{aligned}
				&\left(\left(\mathfrak{R}_{j+1}\mathfrak{R}_{j+2}\cdots\mathfrak{R}_{j+l-1}\mathfrak{R}_{j+l}
				\mathfrak{R}_{j+l-2}\mathfrak{R}_{j+l-3}\cdots\mathfrak{R}_{j+1}\mathfrak{R}_{j}\right)^2\bm{\sigma}\right)_{s}\\
				&=\left\{
				\begin{aligned}
					&\sigma_{s}, \ && \text{for} \ s\in I\setminus J,\\
					&2\sum\limits_{k=j}^{j+l}\mu_{k}^{*}+\sigma_{s}=2\sum\limits_{k=j}^{j+l}\mu_{k}+\sigma_{j-1}-\sigma_{j+l}+\sigma_{j+l+1},  \ && \text{for} \ s=j,       \\
					&2\sum\limits_{k=j+1}^{j+l-1}\mu_{k}+2\sum\limits_{k=j}^{j+l}\mu_{k}+\sigma_{j-1}-\sigma_{j+1}-\sigma_{j+l-1}+\sigma_{j+l+1}+\sigma_{s},  \ && \text{for} \ j+1\leq s\leq j+l-1,\\
					&2\sum\limits_{k=j}^{j+l}\mu_{k}^{*}+\sigma_{s}=2\sum\limits_{k=j}^{j+l}\mu_{k}+\sigma_{j-1}-\sigma_{j}+\sigma_{j+l+1},    \ && \text{for} \ s=j+l.
				\end{aligned}
				\right.
			\end{aligned}
		\end{equation}
		
		\noindent\textbf{Step 2.} By what follows from \cite{Lin-Yang-Zhong-2020}, we obtain that
		\begin{equation}\label{Affine-Toda-Section-2-Eq-27}
			\sigma_{s}^{*}=\sigma_{s}+2\sum\limits_{t\in J}k^{st}\big(\mu_t^{*}+\mu_{t^{*}}^{*}\big)=\sigma_{s}+2\sum\limits_{q=j-1}^{s-1}
			\left(\sum\limits_{k=j}^{2j+|J|-2-q}\mu^{*}_{k}-\sum\limits_{k=j}^{q}\mu^{*}_k\right), \ \text{for} \ s\in J.
		\end{equation}
		Then, direct computation shows that
		\begin{small}
			\begin{equation*}
				\begin{aligned}
					\sigma_{s}^{*}&=\sigma_{s}+2\sum\limits_{q=j-1}^{s-1}
					\left(\sum\limits_{k=j}^{2j+|J|-2-q}\mu^{*}_{k}-\sum\limits_{k=j}^{q}\mu^{*}_k\right)
					=\sigma_{s}+\sum\limits_{q=j-1}^{s-1}2\sum\limits_{k=j}^{2j+|J|-2-q}\mu^{*}_{k}-\sum\limits_{q=j-1}^{s-1}2\sum\limits_{k=j}^{q}\mu^{*}_k\\
					&=\sigma_{s}+\sum\limits_{q=j-1}^{s-1}\left(2\sum\limits_{k=j}^{2j+|J|-2-q}\mu_{k}+\sigma_{j-1}-\sigma_{j}-\sigma_{2j+|J|-2-q}+\sigma_{2j+|J|-1-q}\right)\\
					&\quad-\sum\limits_{q=j-1}^{s-1}\left(2\sum\limits_{k=j}^{q}\mu_{k}+\sigma_{j-1}-\sigma_{j}-\sigma_{q}+\sigma_{q+1}\right)\\
					&=\sigma_{s}+2\sum\limits_{q=j-1}^{s-1}\sum\limits_{k=j}^{2j+|J|-2-q}\mu_{k}-\sum\limits_{q=j-1}^{s-1}\big(\sigma_{2j+|J|-2-q}-\sigma_{2j+|J|-1-q}\big)
					-2\sum\limits_{q=j-1}^{s-1}\sum\limits_{k=j}^{q}\mu_{k}+\sum\limits_{q=j-1}^{s-1}\big(\sigma_{q}-\sigma_{q+1}\big)\\
					&=2\sum\limits_{q=j-1}^{s-1}\sum\limits_{k=j}^{2j+|J|-2-q}\mu_{k}-2\sum\limits_{q=j-1}^{s-1}\sum\limits_{k=j}^{q}\mu_{k}
					-\big(\sigma_{2j+|J|-s-1}-\sigma_{j+|J|}\big)+\sigma_{j-1}.
				\end{aligned}
			\end{equation*}
		\end{small}
		Since $|J|=l+1$, we deduce that
		\begin{equation}\label{Affine-Toda-Section-2-Eq-28}
			\sigma_{s}^{*}=2\sum\limits_{q=j-1}^{s-1}\sum\limits_{k=j}^{2j+l-q-1}\mu_{k}-2\sum\limits_{q=j-1}^{s-1}\sum\limits_{k=j}^{q}\mu_{k}
			-\big(\sigma_{2j+l-s}-\sigma_{j+l+1}\big)+\sigma_{j-1}, \ \text{for} \ s\in J.
		\end{equation}
		
		\noindent\textbf{Step 3.} Comparing \eqref{Affine-Toda-Section-2-Eq-26} and \eqref{Affine-Toda-Section-2-Eq-27}, one can easily check that, for $s\in J\setminus J^{\star}$, 
		\begin{equation}\label{Affine-Toda-Section-2-Eq-29}
			\left(\left(\mathfrak{R}_{j+1}\mathfrak{R}_{j+2}\cdots\mathfrak{R}_{j+l-1}\mathfrak{R}_{j+l}
			\mathfrak{R}_{j+l-2}\mathfrak{R}_{j+l-3}\cdots\mathfrak{R}_{j+1}\mathfrak{R}_{j}\right)^2\bm{\sigma}\right)_{s}=\sigma^{*}_{s}.
		\end{equation}
		By the assumption, \eqref{Affine-Toda-Section-2-Eq-25} holds for $J^{\star}$, since $|J^{\star}|=l-3<l$. Thus, for any $\bm{\sigma}\in\Gamma_{\mathbf{A}}(\bm{\mu})$, applying the same arguments of \cite{Lin-Yang-Zhong-2020} and \eqref{Affine-Toda-Section-2-Eq-28} (by transforming $j$ to $j+2$ and $|J|=l+1$ to $|J^{\star}|=l-3$), we obtain that $\left(\mathfrak{R}_{J^{\star}}\bm{\sigma}\right)_{s}=\sigma_{s}$ for $s\in I\setminus J^{\star}$ and
		\begin{equation*}
			\begin{aligned}
				\left(\mathfrak{R}_{J^{\star}}\bm{\sigma}\right)_{s}=\sigma_{s}^{*}
				&=2\sum\limits_{q=j+1}^{s-1}\sum\limits_{k=j+2}^{2j+|J^{\star}|+2-q}\mu_{k}-2\sum\limits_{q=j+1}^{s-1}\sum\limits_{k=j+2}^{q}\mu_{k}
				-\big(\sigma_{2j+|J^{\star}|-s+3}-\sigma_{j+2+|J^{\star}|}\big)+\sigma_{j+1}\\
				&=2\sum\limits_{q=j+1}^{s-1}\sum\limits_{k=j+2}^{2j+l-q-1}\mu_{k}-2\sum\limits_{q=j+1}^{s-1}\sum\limits_{k=j+2}^{q}\mu_{k}
				-\big(\sigma_{2j+l-s}-\sigma_{j+l-1}\big)+\sigma_{j+1}, \ \text{for} \ s\in J^{\star}.
			\end{aligned}
		\end{equation*}
		As a consequence, combining \eqref{Affine-Toda-Section-2-Eq-26}, we deduce that for $s\in J^{\star}$, i.e., $j+2\leq s\leq j+l-2$,
		\begin{equation}\label{Affine-Toda-Section-2-Eq-30}
				\begin{aligned}
					\left(\mathfrak{R}_{J}\bm{\sigma}\right)_{s}
					&=2\sum\limits_{q=j+1}^{s-1}\sum\limits_{k=j+2}^{2j+l-q-1}\mu_{k}-2\sum\limits_{q=j+1}^{s-1}\sum\limits_{k=j+2}^{q}\mu_{k}
					-\big(\sigma_{2j+l-s}-\sigma_{j+l-1}\big)\\
					&\quad+2\sum\limits_{k=j+1}^{j+l-1}\mu_{k}+2\sum\limits_{k=j}^{j+l}\mu_{k}+\sigma_{j-1}-\sigma_{j+1}-\sigma_{j+l-1}+\sigma_{j+l+1}+\sigma_{j+1}\\
					&=2\sum\limits_{q=j+1}^{s-1}\sum\limits_{k=j+2}^{2j+l-q-1}\mu_{k}-2\sum\limits_{q=j+1}^{s-1}\sum\limits_{k=j+2}^{q}\mu_{k}
					+2\sum\limits_{k=j+1}^{j+l-1}\mu_{k}+2\sum\limits_{k=j}^{j+l}\mu_{k}\\
&\quad-\big(\sigma_{2j+l-s}-\sigma_{j+l+1}\big)+\sigma_{j-1}.
				\end{aligned}
		\end{equation}
		Comparing \eqref{Affine-Toda-Section-2-Eq-28} with \eqref{Affine-Toda-Section-2-Eq-30}, we find that
		\begin{equation}\label{Affine-Toda-Section-2-Eq-31}
			\left(\mathfrak{R}_{J}\bm{\sigma}\right)_{s}=\sigma_{s}^{*}, \ \text{for} \ j+2\leq s\leq j+l-2 \ (\text{i.e.,} \ s\in J^{\star}).
		\end{equation}
		Indeed, for $s\in J^{\star}$, it holds that
		\begin{equation*}
			\begin{aligned}
				&\quad2\sum\limits_{q=j-1}^{s-1}\sum\limits_{k=j}^{2j+l-q-1}\mu_{k}-2\sum\limits_{q=j-1}^{s-1}\sum\limits_{k=j}^{q}\mu_{k}\\
				&=2\sum\limits_{q=j+1}^{s-1}\sum\limits_{k=j}^{2j+l-q-1}\mu_{k}+2\sum\limits_{k=j}^{j+l}\mu_{k}+2\sum\limits_{k=j}^{j+l-1}\mu_{k}
				-2\sum\limits_{q=j+1}^{s-1}\sum\limits_{k=j}^{q}\mu_{k}-2\sum\limits_{k=j}^{j-1}\mu_{k}-2\sum\limits_{k=j}^{j}\mu_{k}\\
				&=2\sum\limits_{q=j+1}^{s-1}\left(\sum\limits_{k=j+2}^{2j+l-q-1}\mu_{k}+\mu_{j}+\mu_{j+1}\right)+2\sum\limits_{k=j}^{j+l}\mu_{k}
				+2\sum\limits_{k=j+1}^{j+l-1}\mu_{k}+2\mu_{j}\\
				&\quad-2\sum\limits_{q=j+1}^{s-1}\left(\sum\limits_{k=j+2}^{q}\mu_{k}+\mu_{j}+\mu_{j+1}\right)-2\mu_{j}\\
				&=2\sum\limits_{q=j+1}^{s-1}\sum\limits_{k=j+2}^{2j+l-q-1}\mu_{k}-2\sum\limits_{q=j+1}^{s-1}\sum\limits_{k=j+2}^{q}\mu_{k}
				+2\sum\limits_{k=j+1}^{j+l-1}\mu_{k}+2\sum\limits_{k=j}^{j+l}\mu_{k}.
			\end{aligned}
		\end{equation*}
		Hence we get \eqref{Affine-Toda-Section-2-Eq-25} from \eqref{Affine-Toda-Section-2-Eq-29} and \eqref{Affine-Toda-Section-2-Eq-31}. This completes the proof of Lemma \ref{Affine-Toda-Section-2-Lemma-10}.
	\end{proof}
	
	\begin{rmk}
		We can apply Lemma \ref{Affine-Toda-Section-2-Lemma-10} to address the blow-up case \textbf{(I)}, as seen in Proposition \ref{Affine-Toda-Section-2-Proposition-2}-(4). However, when dealing with the blow-up case \textbf{(II)} in Proposition \ref{Affine-Toda-Section-2-Proposition-2}-(4), we need additional techniques as outlined below.
	\end{rmk}
	
	\begin{definition}\label{Affine-Toda-Section-2-Definition-12}
		\begin{enumerate}[(a)]
			\item A permutation map $f$ on the index set (tuple) $I$ is said to be an \textbf{$r^{+}$-permutation} if there exists an element $r\in I$ with $r\geq 2$ such that
			\begin{equation*}
				f:(1,2,\cdots,r-1,r,r+1,\cdots,n+1)\rightarrow(r,r+1,\cdots,n+1,1,2,\cdots,r-1).
			\end{equation*}
			In other words, we can view $f$ as a bijective map on $I$ satisfying that
			\begin{equation}\label{Affine-Toda-Section-2-Eq-32}
				\left\{
				\begin{aligned}
					&f(1)=r\\
					&f(2)=r+1\\
					&\quad\quad\quad\vdots\\
					&f(n-r+2)=n+1
				\end{aligned}
				\right.
				\quad \text{and} \quad
				\left\{
				\begin{aligned}
					&f(n-r+3)=1\\
					&f(n-r+4)=2\\
					&\quad\quad\quad\vdots\\
					&f(n+1)=r-1
				\end{aligned}
				\right.
				\quad.
			\end{equation}
			\item A permutation map $f$ on the index set (tuple) $I$ is said to be an \textbf{$r^{-}$-permutation} if there exists an element $r\in I$ with $r\leq n$ such that
			\begin{equation*}
				f:(1,2,\cdots,r,r+1,r+2,\cdots,n+1)\rightarrow(r+1,r+2\cdots,n+1,1,2,\cdots,r).
			\end{equation*}
			We can explicitly express $f$ as
			\begin{equation*}
				\left\{
				\begin{aligned}
					&f(1) =r+1\\
					&f(2) =r+2\\
					&\quad\quad\quad\vdots\\
					&f(n-r+1)=n+1
				\end{aligned}
				\right.
				\ \ \text{and} \ \
				\left\{
				\begin{aligned}
					&f(n-r+2)=1\\
					&f(n-r+3)=2\\
					&\quad\quad\quad\vdots\\
					&f(n+1)=r
				\end{aligned}
				\right.
				\quad.
			\end{equation*}
			In other words, an $r^{-}$-permutation is an $(r+1)^{+}$-permutation for $r\in I\setminus\{n+1\}$.
		\end{enumerate}
	\end{definition}
	
	\begin{rmk}
		\begin{enumerate}[(a)]
			\item The $1^{+}$-permutation and $(n+1)^{-}$-permutation are exactly the identity.
			\item Any $r^{+}$-permutation is the $(r-1)^{-}$-permutation for $r\in I\setminus\{1\}$.
		\end{enumerate}
	\end{rmk}
	
	For the $r^{\pm}$-permutation, we provide some properties for later application.
	
	\begin{lemma}\label{Affine-Toda-Section-2-Lemma-13}
		Suppose that $\bm{\sigma}=(\sigma_1,\cdots,\sigma_{n+1})\in\Gamma_{\mathbf{A}}(\bm{\mu})$ and $f$ is an $r^{\pm}$-permutation on $I$. Then
		\begin{enumerate}[(a)]
			\item $\bm{\sigma}_{f}=(\sigma_{f(1)},\sigma_{f(2)},\cdots,\sigma_{f(n+1)})\in\Gamma_{\mathbf{A}}(\bm{\mu})$.
			\item $\bm{\sigma}_{f^{-1}}=(\sigma_{f^{-1}(1)},\sigma_{f^{-1}(2)},\cdots,\sigma_{f^{-1}(n+1)})\in\Gamma_{\mathbf{A}}(\bm{\mu})$.
		\end{enumerate}
	\end{lemma}
	\begin{proof}
		Without loss of generality, we just consider the $r^{+}$-permutation. Suppose that $f$ is an $r^{+}$-permutation on $I$ satisfying \eqref{Affine-Toda-Section-2-Eq-32}.
		
		(a) By Proposition \ref{Affine-Toda-Section-2-Proposition-9}, we shall prove that $\bm{\sigma}_{f}$ satisfies the Pohozaev identity \eqref{Affine-Toda-Section-2-Eq-1}, i.e.,
		\begin{equation}\label{Affine-Toda-Section-2-Eq-33}
			(\sigma_{f(1)}-\sigma_{f(2)})^{2}+(\sigma_{f(2)}-\sigma_{f(3)})^{2}+\cdots+(\sigma_{f(n)}-\sigma_{f(n+1)})^{2}+(\sigma_{f(n+1)}-\sigma_1)^{2}
			=4\sum\limits_{i\in I}\mu_{f(i)}\sigma_{f(i)}.
		\end{equation}
		Notice that \eqref{Affine-Toda-Section-2-Eq-33} is equivalent to
		\begin{equation*}
			\begin{aligned}
				&(\sigma_{r}-\sigma_{r+1})^{2}+(\sigma_{r+1}-\sigma_{r+2})^{2}+\cdots+(\sigma_{n}-\sigma_{n+1})^{2}+(\sigma_{n+1}-\sigma_{1})^{2}\\
				&+(\sigma_{1}-\sigma_{2})^{2}+(\sigma_{2}-\sigma_{3})^{2}+\cdots+(\sigma_{r-2}-\sigma_{r-1})^{2}+(\sigma_{r-1}-\sigma_{r})^{2}=4\sum\limits_{i\in I}\mu_{i}\sigma_{i},
			\end{aligned}
		\end{equation*}
		which is automatically true since $\bm{\sigma}$ satisfies \eqref{Affine-Toda-Section-2-Eq-1}.
		
		(b) By \eqref{Affine-Toda-Section-2-Eq-32}, we find that
		\begin{equation*}
			\left\{
			\begin{aligned}
				&f^{-1}(1)=n-r+3\\
				&f^{-1}(2)=n-r+4\\
				&\quad\quad\quad\vdots\\
				&f^{-1}(r-1) =n+1
			\end{aligned}
			\right.
			\quad\text{and}\quad
			\left\{
			\begin{aligned}
				&f^{-1}(r)=1\\
				&f^{-1}(r+1)=2\\
				&\quad\quad\quad\vdots\\
				&f^{-1}(n+1)=n-r+2
			\end{aligned}
			\right.
			\quad.
		\end{equation*}
		Applying the same arguments of (a), it is not difficult to check that
		\begin{equation*}
			\bm{\sigma}_{f^{-1}}=(\sigma_{n-r+3},\sigma_{n-r+4},\cdots,\sigma_{n+1},\sigma_{1},\sigma_{2},\cdots,\sigma_{n-r+2})\in\Gamma_{\mathbf{A}}(\bm{\mu}).
		\end{equation*}
		One can also treat $f^{-1}$ as the $(n-r+3)^{+}$-permutation, and then apply (a) to get the same conclusion. This completes the proof of Lemma \ref{Affine-Toda-Section-2-Lemma-13}.
	\end{proof}
	
	\begin{theorem}\label{Affine-Toda-Section-2-Theorem-14}
		Suppose that $\bm{\sigma}=(\sigma_1,\cdots,\sigma_{n+1})\in\Gamma_{\mathbf{A}}(\bm{\mu})$ and set $\mu_{i}^{*}=\mu_{i}-\frac{1}{2}\sum_{t\in I}k_{it}\sigma_{t}$ for $i\in I$. The index set $I$ is decomposed into one form of the two alternatives \textbf{(I)} and \textbf{(II)} (see Proposition \ref{Affine-Toda-Section-2-Proposition-2}-(4)).
		\begin{enumerate}[(a)]
			\item Alternative \textbf{(I)} holds: the index set $I$ is decomposed into the form $I=J\cup N$ and $J=J_1\cup\cdots\cup J_{\vartheta}$ for some $\vartheta\in\mathbb{N}$ satisfying (I-1). Set $\sigma_{i}^{*}=\sigma_{i}$ for $i\in N$ and
			\begin{equation*}
				\sigma_{i}^{*}=\sigma_{i}+2\sum\limits_{j\in J_p}k_p^{ij}\big(\mu_j^{*}+\mu_{j_p^{*}}^{*}\big), \ \ \text{for} \ i\in J_p, \ p=1,\cdots,\vartheta,
			\end{equation*}
			where $j_p^{*}=2i_p+|J_p|-j-1$ and $(k_{p}^{ij})_{|J_p|\times|J_p|}$ stands for the inverse of $(k_{ij})_{|J_p|\times|J_p|}$, $i,j\in J_p$. Then
			\begin{equation*}
				\bm{\sigma}^{*}=(\sigma^{*}_1,\cdots,\sigma^{*}_{n+1})
				=\left(\mathfrak{R}_{J_1}\mathfrak{R}_{J_2}\cdots\mathfrak{R}_{J_\vartheta}\right)\bm{\sigma}\in\Gamma_{\mathbf{A}}(\bm{\mu}).
			\end{equation*}
			
			\item Alternative \textbf{(II)} holds: the index set $I$ is decomposed into the form $I=J\cup N$ and $J=J_{0}\cup J_1\cup\cdots\cup J_{\vartheta}$ for some $\vartheta\in\mathbb{N}\cup\{0\}$ satisfying (II-1). Set $\sigma_{i}^{*}=\sigma_{i}$ for $i\in N$ and
			\begin{equation}\label{Affine-Toda-Section-2-Eq-34}
				\left\{
				\begin{aligned}
					&\sigma_{i}^{*}=\sigma_{i}+2\sum\limits_{j\in J_p}k_p^{ij}\big(\mu_j^{*}+\mu_{j_p^{*}}^{*}\big), \ &&\text{for} \ i\in J_p, \ p=1,\cdots,\vartheta,\\
					&\sigma_{s_i}^{*}=\sigma_{s_i}+2\sum\limits_{j\in \{1,2,\cdots,|J_0|\}}k_0^{s_is_j}\big(\mu_{s_j}^{*}+\mu_{s_j^{*}}^{*}\big), \ &&\text{for} \ i\in \{1,2,\cdots,|J_0|\},
				\end{aligned}
				\right.
			\end{equation}
			where $j_p^{*}=2i_p+|J_p|-j-1$ and $(k_{p}^{ij})_{|J_p|\times|J_p|}$ stands for the inverse of $(k_{ij})_{|J_p|\times|J_p|}$, $i,j\in J_p$; $s_j^{*}=s_{1+|J_0|-j}$ and $(k_0^{s_is_j})_{|J_0|\times|J_0|}$ stands for the inverse of $(k_{s_is_j})_{|J_0|\times|J_0|}$, $i,j\in \{1,2,\cdots,|J_0|\}$. Then
			\begin{equation*}
				\bm{\sigma}^{*}=(\sigma^{*}_1,\cdots,\sigma^{*}_{n+1}) =\left(\mathfrak{R}_{J_0}\mathfrak{R}_{J_1}\cdots\mathfrak{R}_{J_\vartheta}\right)\bm{\sigma}\in\Gamma_{\mathbf{A}}(\bm{\mu}).
			\end{equation*}
		\end{enumerate}
	\end{theorem}
	\begin{proof}
		(a) By Lemma \ref{Affine-Toda-Section-2-Lemma-10}, we get that
		\begin{equation*}
			\sigma_{i}^{*}=\sigma_{i}+2\sum\limits_{j\in J_p}k_p^{ij}\big(\mu_j^{*}+\mu_{j_p^{*}}^{*}\big)=\big(\mathfrak{R}_{J_p}\bm{\sigma}\big)_{i}, \ \ \text{for} \ i\in J_{p}, \ p=1,\cdots,\vartheta,
		\end{equation*}
		where $\mathfrak{R}_{J_p}$ is the $J_p$-chain corresponding to the consecutive-index set $J_p$. Since each $J_p$ is maximal, we have
		\begin{equation*}
			\emptyset\neq\{i_{p}+l_{p}+1,i_{p}+l_{p}+2,\cdots,i_{p+1}-1\}\subseteq N, \ \text{for} \ p=1,\cdots,\vartheta-1.
		\end{equation*}
		Hence
		\begin{equation*}
			\bm{\sigma}^{*}=\left(\sigma_{1}^{*},\cdots,\sigma_{n+1}^{*}\right)
			=\left(\mathfrak{R}_{J_1}\mathfrak{R}_{J_2}\cdots\mathfrak{R}_{J_\vartheta}\right)\bm{\sigma}
			\in\Gamma_{\mathbf{A}}(\bm{\mu}).
		\end{equation*}
		
		\noindent (b) We split the arguments into several steps.
		
		\noindent\textbf{Step 1.} Suppose that $f$ is the $r_2^{+}$-permutation on $I$ satisfying \eqref{Affine-Toda-Section-2-Eq-32} with $r=r_2$. Since $\bm{\sigma}\in\Gamma_{\mathbf{A}}(\bm{\mu})$, we deduce from Lemma \ref{Affine-Toda-Section-2-Lemma-13}-(a) that $\bm{\sigma}_{f}=\left(\sigma_{f,1},\cdots,\sigma_{f,n+1}\right)\in\Gamma_{\mathbf{A}}(\bm{\mu})$, where $\sigma_{f,i}=\sigma_{f(i)}$ for $i\in I$. Denote by
		\begin{equation*}
			\overline{\mu}_{i}:=\mu_{f(i)}^{*}=\mu_{f(i)}-\frac{1}{2}\sum_{j\in I}k_{f(i),j}\sigma_{j}=\mu_{f(i)}-\frac{1}{2}\sum_{j\in I}k_{ij}\sigma_{f(j)}=\mu_{f(i)}-\frac{1}{2}\sum_{j\in I}k_{ij}\sigma_{f,j}, \ \text{for} \ i\in I
		\end{equation*}
		and
		\begin{equation}\label{Affine-Toda-Section-2-Eq-35}
			\left\{
			\begin{aligned}
				&\overline{J_p}=f^{-1}(J_p)=\{\overline{i}_{p},\overline{i}_{p}+1,\cdots,\overline{i}_{p}+l_{p}\}, \ \text{where} \ \overline{i}_{p}=i_p+(n-r_2+2), \ p=1,\cdots,\vartheta,\\
				&\overline{J_0}=f^{-1}(J_0)=\{\overline{i}_{0},\overline{i}_{0}+1,\cdots,\overline{i}_{0}+|J_0|-1\}, \ \text{where} \ \overline{i}_0=1,\\
				&\overline{N}=f^{-1}(N)=\{s+(n-r_2+2) \mid s\in N\}.
			\end{aligned}
			\right.
		\end{equation}
		Then it is easy to see that $I=\overline{J_0}\cup \overline{J_1}\cup\cdots\cup\overline{J_{\vartheta}}\cup\overline{N}$ and each $\overline{J_p}$, $p=0,1,\cdots,\vartheta$ consists of the maximal consecutive indices. We define $\left(\bm{\sigma}_{f}\right)^{*}=\left((\sigma_{f,1})^{*},(\sigma_{f,2})^{*},\cdots,(\sigma_{f,n+1})^{*}\right)$ with
		\begin{equation*}
			(\sigma_{f,i})^{*}=
			\left\{
			\begin{aligned}
				&\sigma_{f,i}, \ &&\text{for} \ i\in \overline{N},\\
				&\sigma_{f,i}+2\sum\limits_{j\in \overline{J_p}}k_p^{f(i)f(j)}\big(\overline{\mu}_{j}+\overline{\mu}_{\overline{j}_p^{*}}\big), \ &&\text{for} \ i\in \overline{J_p}, \ p=0,1,\cdots,\vartheta,
			\end{aligned}
			\right.
		\end{equation*}
		where $\overline{j}_p^{*}=2\overline{i}_p+|\overline{J_p}|-j-1$, and $(k_{p}^{f(i)f(j)})_{|\overline{J_p}|\times|\overline{J_p}|}$ is the inverse of $(k_{f(i)f(j)})_{|\overline{J_p}|\times|\overline{J_p}|}$, $i,j\in \overline{J_p}$.
		
		\noindent\textbf{Step 2.} For any $\bm{\sigma}\in\Gamma_{\mathbf{A}}(\bm{\mu})$, we denote by $\tilde{\mu}_i=\mu_{f(i)}$ for $i\in I$ and denote $\widetilde{\mathfrak{R}}_i$ by
		\begin{equation*}
			(\widetilde{\mathfrak{R}}_i\bm{\sigma})_j
			=\begin{cases}
				2\tilde{\mu}_i-\sum\limits_{t\in I}k_{it}\sigma_t+\sigma_i, \ &\text{if} \ j=i,\\	
				\sigma_j, \ &\text{if} \ j\neq i.	
			\end{cases}
		\end{equation*}
		Since each $\overline{J_p}$, $p=0,1,\cdots,\vartheta$ consists of the maximal consecutive indices, applying Lemma \ref{Affine-Toda-Section-2-Lemma-10} with $\mu_i=\tilde{\mu}_i$, we find that
		\begin{equation*}
			\sigma_{f,i}+2\sum\limits_{j\in \overline{J_p}}k_p^{f(i)f(j)}\big(\overline{\mu}_{j}+\overline{\mu}_{\overline{j}_p^{*}}\big)
			=\sigma_{f,i}+2\sum\limits_{j\in \overline{J_p}}k_p^{ij}\big(\overline{\mu}_{j}+\overline{\mu}_{\overline{j}_p^{*}}\big)
			=\big(\widetilde{\mathfrak{R}}_{\overline{J_p}}\bm{\sigma}_{f}\big)_{i}, \ \text{for} \ i\in \overline{J_p}, \ p=0,1,\cdots,\vartheta,
		\end{equation*}
		where $\widetilde{\mathfrak{R}}_{\overline{J_p}}$ is the corresponding $\overline{J_p}$-chain with respect to $\widetilde{\mathfrak{R}}_i$, $i\in\overline{J_p}$ and $(k_{p}^{ij})_{|\overline{J_p}|\times|\overline{J_p}|}$ stands for the inverse of $(k_{ij})_{|\overline{J_p}|\times|\overline{J_p}|}$, $i,j\in \overline{J_p}$, $p=0,1,\cdots,\vartheta$.
		Hence we conclude that
		\begin{equation*}
			\left(\bm{\sigma}_{f}\right)^{*}=\left(\widetilde{\mathfrak{R}}_{\overline{J_0}}\widetilde{\mathfrak{R}}_{\overline{J_1}}\cdots
			\widetilde{\mathfrak{R}}_{\overline{J_{\vartheta}}}\right)\bm{\sigma}_{f}.
		\end{equation*}
		Noticing that for each $t\in \overline{J_p}$, $p=0,1,\cdots,\vartheta$, we see that
		\begin{equation*}
			(\widetilde{\mathfrak{R}}_t\bm{\sigma}_f)_s
			=\begin{cases}
				2\mu_{f(t)}-\sum\limits_{j\in I}k_{tj}\sigma_{f,j}+\sigma_{f,t}=(\mathfrak{R}_{f(t)}\bm{\sigma})_{f(t)}, \ &\text{if} \ s=t,\\	
				\sigma_{f,s}=\sigma_{f(s)}=(\mathfrak{R}_{f(t)}\bm{\sigma})_{f(s)}, \ &\text{if} \ s\neq t,
			\end{cases}
		\end{equation*}
		which implies that $\widetilde{\mathfrak{R}}_{t}\bm{\sigma}_f=(\mathfrak{R}_{f(t)}\bm{\sigma})_{f}\in\Gamma_{\mathbf{A}}(\bm{\mu})$. Here we have used Lemma \ref{Affine-Toda-Section-2-Lemma-13}-(a). Thus, we deduce that
		\begin{equation*}
			\widetilde{\mathfrak{R}}_{\overline{J_p}}\bm{\sigma}_f=\Big(\prod\limits_{t_i\in\Lambda}\widetilde{\mathfrak{R}}_{t_i}\Big)\bm{\sigma}_f
			=\Big(\prod\limits_{t_i\in\Lambda}\mathfrak{R}_{f(t_i)}\bm{\sigma}\Big)_f=(\mathfrak{R}_{J_p}\bm{\sigma})_{f}
			\in\Gamma_{\mathbf{A}}(\bm{\mu}), \ p=0,1,\cdots,\vartheta,
		\end{equation*}
		where $\Lambda$ is a tuple of length ${|J_p|\times(|J_p|+1)}/{2}$ satisfying that $t_i\in \overline{J_p}$ for each element $t_i\in\Lambda$. Therefore, we also get that
		\begin{equation*}
			\left(\bm{\sigma}_{f}\right)^{*}=\left(\widetilde{\mathfrak{R}}_{\overline{J_0}}\widetilde{\mathfrak{R}}_{\overline{J_1}}\cdots
			\widetilde{\mathfrak{R}}_{\overline{J_{\vartheta}}}\right)\bm{\sigma}_{f}
			=\left(\mathfrak{R}_{J_0}\mathfrak{R}_{J_1}\cdots\mathfrak{R}_{J_\vartheta}\bm{\sigma}\right)_{f}\in\Gamma_{\mathbf{A}}(\bm{\mu}).
		\end{equation*}
		
		\noindent\textbf{Step 3.} Let $\bm{\sigma}^{*}=(\sigma^{*}_1,\cdots,\sigma^{*}_{n+1})$ be defined as in \eqref{Affine-Toda-Section-2-Eq-34} and denote by $\left(\bm{\sigma}^{*}\right)_{f}=(\sigma^{*}_{f(1)},\cdots,\sigma^{*}_{f(n+1)})$. First we claim that $\left(\bm{\sigma}^{*}\right)_{f}=\left(\bm{\sigma}_{f}\right)^{*}$. In fact, for any $t\in I$, we consider the following cases:
		\begin{enumerate}[(1).]
			\item if $f(t)\in N$, i.e., $t\in f^{-1}(N)=\overline{N}$, then we set $i:=f(t)$ and obtain that
			\begin{equation*}
				\sigma^{*}_{f(t)}=\sigma^{*}_{i}=\sigma_{i}.
			\end{equation*}
			\item if $f(t)\in J_p$ for some $p=1,\cdots,\vartheta$, i.e., $t\in f^{-1}(J_p)=\overline{J_p}$, then we set $i:=f(t)$ and obtain that
			\begin{equation*}
				\begin{aligned}
					\sigma^{*}_{f(t)}=\sigma^{*}_{i}
					&=\sigma_{i}+2\sum\limits_{j\in J_p}k_p^{ij}\big(\mu_j^{*}+\mu_{j_p^{*}}^{*}\big), \ \ \ j_p^{*}=2i_p+|J_p|-j-1,\\
					&=\sigma_{f(t)}+2\sum\limits_{s\in \overline{J_p}}k_p^{f(t)f(s)}\big(\mu_{f(s)}^{*}+\mu_{f(\overline{s}_p^{*})}^{*}\big), \ \ \ \overline{s}_p^{*}=2\overline{i}_p+|\overline{J_p}|-s-1,\\
					&=\sigma_{f,t}+2\sum\limits_{s\in \overline{J_p}}k_p^{f(t)f(s)}\big(\overline{\mu}_{s}+\overline{\mu}_{\overline{s}_p^{*}}\big), \ \ \ \overline{s}_p^{*}=2\overline{i}_p+|\overline{J_p}|-s-1.
				\end{aligned}
			\end{equation*}
			\item if $f(t)\in J_0$, i.e., $t\in f^{-1}(J_0)=\overline{J_0}$, then we set $s_i:=f(t)$ and obtain that
			\begin{equation*}
				\begin{aligned}
					\sigma^{*}_{f(t)}=\sigma_{s_i}^{*}&=\sigma_{s_i}+2\sum\limits_{j\in \{1,2,\cdots,|J_0|\}}k_0^{s_is_j}\left(\mu_{s_j}^{*}+\mu_{s_j^{*}}^{*}\right), \ \ \ s_j^{*}=s_{1+|J_0|-j },\\
					&=\sigma_{f(t)}+2\sum\limits_{s\in \overline{J_0}}k_0^{f(t)f(s)}\left(\mu_{f(s)}^{*}+\mu_{f(\overline{s}_0^{*})}^{*}\right), \ \ \ \overline{s}_0^{*}=1+|\overline{J_0}|-s,\\
					&=\sigma_{f,t}+2\sum\limits_{s\in \overline{J_0}}k_0^{f(t)f(s)}\left(\overline{\mu}_{s}+\overline{\mu}_{\overline{s}_0^{*}}\right), \ \ \ \overline{s}_0^{*}=1+|\overline{J_0}|-s.
				\end{aligned}
			\end{equation*}
		\end{enumerate}
		By the above discussion, we get that $\left(\bm{\sigma}^{*}\right)_{f}=\left(\bm{\sigma}_{f}\right)^{*}$. As a consequence,
		\begin{equation}\label{Affine-Toda-Section-2-Eq-36}
			\left(\bm{\sigma}^{*}\right)_{f}=\left(\bm{\sigma}_{f}\right)^{*}
			=\left(\mathfrak{R}_{J_0}\mathfrak{R}_{J_1}\cdots\mathfrak{R}_{J_\vartheta}\bm{\sigma}\right)_{f}\in\Gamma_{\mathbf{A}}(\bm{\mu}).
		\end{equation}
		
		\noindent\textbf{Step 4.} By Lemma \ref{Affine-Toda-Section-2-Lemma-13}-(b) and \eqref{Affine-Toda-Section-2-Eq-36}, we have
		\begin{equation*}
			\bm{\sigma}^{*}=\left(\bm{\sigma}^{*}\right)_{f\circ f^{-1}}=\left(\left(\bm{\sigma}^{*}\right)_{f}\right)_{f^{-1}}
			=\left(\mathfrak{R}_{J_0}\mathfrak{R}_{J_1}\cdots\mathfrak{R}_{J_\vartheta}\right)\bm{\sigma}\in\Gamma_{\mathbf{A}}(\bm{\mu}).
		\end{equation*}
		
		Thus, the Theorem \ref{Affine-Toda-Section-2-Theorem-14} is proved.
	\end{proof}
	
	\begin{rmk}[remark of Theorem \ref{Affine-Toda-Section-2-Theorem-14}]
		The proof of the (b)-th conclusion of Theorem \ref{Affine-Toda-Section-2-Theorem-14} implies that $\sigma_{s_i}^{*}=\left(\mathfrak{R}_{J_0}\bm{\sigma}\right)_{s_i}$ for $i=1,\cdots,|J_0|$, where $\mathfrak{R}_{J_0}$ refers to the $J_0$-chain associated with the nonconsecutive-index set $J_0$.
	\end{rmk}
	
	\section{Proof of Theorem \ref{Affine-Toda-Section-1-Theorem-1}}\label{Affine-Toda-Section-3}
	\setcounter{equation}{0}
	In this section, we will present the proof for Theorem \ref{Affine-Toda-Section-1-Theorem-1}, which involves utilizing a selection process to identify the "bad" points and their corresponding bubbling disks. We then compute the local mass contributions of each bubbling disk containing the "bad" points and analyze the transformations that occur at various stages. This approach has been successfully applied to the classical Toda systems, including types $\mathbf{A}_n$, $\mathbf{B}_n$, $\mathbf{C}_n$, $\mathbf{G}_2$ and $\mathbf{B}_2^{(1)}$, as detailed in \cite{Cui-Wei-Yang-Zhang-2022,Lin-Wei-Yang-Zhang-2018,Lin-Wei-Zhang-2015,Lin-Yang-Zhong-2020}. Here, we will outline the proof of the entire process. First, we identify all the bubbling disks as described in section \ref{Affine-Toda-Section-2}. Then, we begin with one bubbling disk and compute all possible contributions of local mass from each disk before it meets the nearest bubbling disk. Next, we collect the adjacent bubbling disks into a "group", a concept introduced in \cite{Lin-Wei-Zhang-2015}, and calculate the types of local mass in the group. The algebraic structure plays an essential role in determining the principles of the transformations that occur in the group.
	
	Before sketching out the proof, we make some preparations. Throughout this section, for any sequence $(\mathbf{x},\mathbf{s})=\{(x^k,s_k)\}$, we always denote by
	\begin{equation*}
		\sigma^k_i(B(x^k,s_k))=\frac{1}{2\pi}\int_{B(x^k,s_k)}e^{u^k_i(x)}\mathrm{d}x, \ \text{for} \ i\in I.
	\end{equation*}
	We define the local mass $\hat{\bm{\sigma}}(B(\mathbf{x},\mathbf{s}))=(\hat{\sigma}_1(B(\mathbf{x},\mathbf{s})),\cdots,\hat{\sigma}_{n+1}(B(\mathbf{x},\mathbf{s})))$ as
	\begin{equation*}
		\hat{\sigma}_i(B(\mathbf{x},\mathbf{s}))=
		\begin{cases}
			\lim\limits_{k\rightarrow+\infty}\sigma^k_i(B(x^k,s_k)), \ &\text{if} \ u^k_i \ \text{has fast decay on} \ \partial B(x^k,s_k),\\
			\lim\limits_{r\rightarrow 0}\lim\limits_{k\rightarrow+\infty}\sigma^k_i(B(x^k,rs_k)), \ &\text{if} \ u^k_i \ \text{has slow decay on} \ \partial B(x^k,s_k).
		\end{cases}
	\end{equation*}
	For simplicity we set $\hat{\sigma}_{i}(\mathbf{s})=\hat{\sigma}_{i}(B(\bm{0},\mathbf{s}))$ if $x^k=0$. We will calculate all possible values of $\hat{\sigma}_{i}(\mathbf{s})$ as $\mathbf{s}$ increases, such that the disk $B(\bm{0},\mathbf{s})$ includes all the "bad" points. The argument is divided into the following five steps:
	\medskip
	
	\noindent\textbf{Step 1.} In this step, we shall compute the energy of the blow-up solutions in each bubbling disk $B(x_t^k,l_t^k)$, $t=1,\cdots,m$, as introduced in Proposition \ref{Affine-Toda-Section-2-Proposition-2}. For any fixed $x^k_t\in\Sigma_k\setminus\{0\}$, $l^k_t\ll\tau^k_t=\frac{1}{2}\mathrm{dist}(x^{k}_{t},\Sigma_{k}\setminus\{x^{k}_{t}\})$. Applying Proposition \ref{Affine-Toda-Section-2-Proposition-2}-(1), the Harnack-type inequality holds
	\begin{equation*}
		u^k_i(x)+2\log|x-x^k_t|\leq C, \ |x-x^k_t|\leq\tau^k_t, \ \forall \ i\in I.
	\end{equation*}
	In $B(x^{k}_{t},l_{t}^{k})$, we perform the scaling on $\mathbf{u}^k$ as
	\begin{equation*}
		v^k_{i}(y):=u^k_i(x^k_t+\varepsilon_{t}^{k}y)+2\log\varepsilon_{t}^{k}, \ \text{for} \ i\in I.
	\end{equation*}
	Then by Proposition \ref{Affine-Toda-Section-2-Proposition-2}-(4), $\mathbf{u}^k$ has fast decay on $\partial B(x^k_t,l^k_t)$ and one of the following alternatives holds:\\
	\noindent\textbf{(I).} the index set $I$ is decomposed into the form $I=J\cup N$ and $J=J_1\cup\cdots\cup J_{\vartheta}$ for some $\vartheta\in\mathbb{N}$ satisfying (I-1)--(I-3).\\
	\noindent\textbf{(II).} the index set $I$ is decomposed into the form $I=J\cup N$ and $J=J_{0}\cup J_1\cup\cdots\cup J_{\vartheta}$ for some $\vartheta\in\mathbb{N}\cup\{0\}$ satisfying (II-1)--(II-4).\\
	In any case, by the classification result \cite[Theorem 1.1]{Lin-Wei-Ye-2012}, we obtain that there exists a sequence $R_k$ with $R_k\rightarrow+\infty$ as $k\rightarrow+\infty$ such that $l_t^k=R_k\varepsilon^k_{t}\ll\tau_t^k$ and
	\begin{equation*}
		\int_{B(0,{R_k})}e^{v_{i}(y)}\mathrm{d}y=\int_{\mathbb{R}^2}e^{v_{i}(y)}\mathrm{d}y+o(1), \ \text{for} \ i\in J.
	\end{equation*}
	One can easily conclude that $\hat{\sigma}_i(B(\mathbf{x}_t,\mathbf{l}_t))=0$ for $i\in N$ and
	\begin{equation*}
		\hat{\sigma}_i(B(\mathbf{x}_t,\mathbf{l}_t))=\frac{1}{2\pi}\int_{\mathbb{R}^2}e^{v_i(y)}\mathrm{d}y, \ \text{for} \ i\in J,
	\end{equation*}
	where $(\mathbf{x}_t,\mathbf{l}_t)$ stands for the sequence $\{(x^k_t,l^k_t)\}$. By Theorem \ref{Affine-Toda-Section-2-Theorem-14}-(a) (or Lemma \ref{Affine-Toda-Section-2-Lemma-10}), we deduce that
	\begin{equation*}
		\hat{\bm{\sigma}}(B(\mathbf{x}_t,\mathbf{l}_t))
		=\left(\hat{\sigma}_1(B(\mathbf{x}_t,\mathbf{l}_t)),\cdots,\hat{\sigma}_{n+1}(B(\mathbf{x}_t,\mathbf{l}_t))\right)\in\Gamma_{\mathbf{A}}(1,\cdots,1).
	\end{equation*}
	
	\noindent\textbf{Step 2.} Next, we analyze the variation in energy as we transition from $B(x^k_t,l^k_t)$ to $B(x^k_t,\tau^k_t)$. If
	\begin{equation*}
		\hat{\sigma}_i(B(\mathbf{x}_t,\bm{\tau}_{t}))=\hat{\sigma}_i(B(\mathbf{x}_t,\mathbf{l}_t)), \ \forall \ i\in I,
	\end{equation*}
	where $(\mathbf{x}_t,\bm{\tau}_{t})$ stands for the sequence $\{(x^k_t,\tau^k_t)\}$, then we conclude that
	\begin{equation*}
		\hat{\bm{\sigma}}(B(\mathbf{x}_t,\bm{\tau}_{t}))
		=\left(\hat{\sigma}_1(B(\mathbf{x}_t,\bm{\tau}_{t})),\cdots,\hat{\sigma}_{n+1}(B(\mathbf{x}_t,\bm{\tau}_{t}))\right)\in\Gamma_{\mathbf{A}}(1,\cdots,1),
	\end{equation*}
	as required. Otherwise, there exists some $i_0\in I$ such that
	\begin{equation*}
		\hat{\sigma}_{i_0}(B(\mathbf{x}_t,\bm{\tau}_{t}))>\hat{\sigma}_{i_0}(B(\mathbf{x}_t,\mathbf{l}_t)).
	\end{equation*}
	By the same arguments of \cite[Lemma 4.3]{Cui-Wei-Yang-Zhang-2022}, we can choose a sequence $\{s_k\}$, $l_t^k\ll s_k\ll\tau^k_t$ such that
	\begin{enumerate}[(i)]
		\item some components of $\mathbf{u}^k$ have slow decay on $\partial B(x^k_t,s_k)$,
		\item $\hat{\sigma}_i(B(\mathbf{x}_t,\mathbf{s}))=\hat{\sigma}_i(B(\mathbf{x}_t,\mathbf{l}_t))$ for $i\in I$, where $(\mathbf{x}_t,\mathbf{s})$ stands for the sequence $\{(x^k_t,s_k)\}$.
	\end{enumerate}
	We scale $\mathbf{u}^k$ by $v^k_i(y)=u^k_i(x_t^k+s_ky)+2\log{s_k}$ for $i\in I$ and denote by
	\begin{equation*}
		J=\big\{i\in I \mid u^k_i \ \text{has slow decay on} \ \partial B(x_t^k,s_k)\big\}\subsetneqq I.
	\end{equation*}
	Then, each $v^k_i$, $i\in I\setminus J$ has fast decay on $\partial B(x_t^k,s_k)$, i.e., $v^k_i(y)\rightarrow-\infty$ in $L_{\mathrm{loc}}^{\infty}(\mathbb{R}^2)$ for $i\in I\setminus J$ and $v^k_i(y)\rightarrow v_i(y)$ in $C_{\mathrm{loc}}^2(\mathbb{R}^2)$ for $i\in J$, where $v_i(y)$ satisfies the system
	\begin{equation*}
		\Delta v_i(y)+\sum\limits_{j\in J}k_{ij}e^{v_j(y)}=4\pi\alpha_i^{*}\delta_0 \ \text{in} \ \mathbb{R}^2, \ \int_{\mathbb{R}^2}e^{v_i(y)}\mathrm{d}y<+\infty, \ \ \forall \ i\in J.
	\end{equation*}
	Here, $\alpha_i^{*}=-\frac{1}{2}\sum_{j\in I}k_{ij}\hat{\sigma}_j(B(\mathbf{x}_t,\mathbf{l}_t))>-1$ for $i\in J$. For the index set $J$, we shall encounter the same two alternatives \textbf{(I)} and \textbf{(II)} as in \textbf{Step 1}. In any case, we deduce that there exists a sequence $N_k^{*}$ with $N_k^{*}\rightarrow+\infty$ as $k\rightarrow+\infty$ such that $l_t^k\ll N_k^{*}s_k\ll\tau_t^k$, $\hat{\sigma}_i(B(\mathbf{x}_t,\mathbf{N^{*}s}))=\hat{\sigma}_i(B(\mathbf{x}_t,\mathbf{l}_t))$ for $i\in N$ and
	\begin{equation*}
		\hat{\sigma}_i(B(\mathbf{x}_t,\mathbf{N^{*}s}))=\hat{\sigma}_i(B(\mathbf{x}_t,\mathbf{l}_t))
		+\frac{1}{2\pi}\int_{\mathbb{R}^2}e^{v_i(y)}\mathrm{d}y, \ \text{for} \ i\in J,
	\end{equation*}
	where $B(\mathbf{x}_t,\mathbf{N^{*}s})$ stands for the sequence $\{x_t^k,N_k^{*}s_k\}$. Thus we get from Theorem \ref{Affine-Toda-Section-2-Theorem-14} that
	\begin{equation*}
		\hat{\sigma}_i(B(\mathbf{x}_t,\mathbf{N^{*}s}))\in 2\mathbb{N}\cup\{0\}, \ \text{for} \ i\in I.
	\end{equation*}
	
	Let $s_{k,1}=N_{k}^{*}s_k$. If $\hat{\sigma}_i(B(\mathbf{x}_t,\bm{\tau}_t))=\hat{\sigma}_i(B(\mathbf{x}_t,\mathbf{s}_1))$ for $i\in I$, where $\mathbf{s}_1$ stands for the sequence $\{s_{k,1}\}$, then we finished the discussion in this step. Otherwise, we repeat the above arguments to find $s_{k,1}\ll s_{k,j}\ll s_{k,j+1}$ such that $\hat{\bm{\sigma}}(B(\mathbf{x}_t,\mathbf{s}_{j+1}))\in\Gamma_{\mathbf{A}}(1,\cdots,1)$, where $\mathbf{s}_{j+1}$ stands for the sequence $\{s_{k,j+1}\}$. Since the energy is finite and the total gain for the local masses at each step has a lower bound, the process will stop after $j_0$ (some finite number) steps. Hence $\hat{\bm{\sigma}}(B(\mathbf{x}_t,\bm{\tau}_t))=\hat{\bm{\sigma}}(B(\mathbf{x}_t,\mathbf{s}_{j_0}))$ with $\hat{\sigma}_i(B(\mathbf{x}_t,\bm{\tau}_t))\in 2\mathbb{N}\cup\{0\}$ for each $i\in I$, where $\mathbf{s}_{j_0}$ stands for the sequence $\{s_{k,j_0}\}$.
	\medskip
	
	\noindent\textbf{Step 3.} Next, we gather the adjacent disks into a "group" and calculate the local mass of this group. As we have did in \cite{Cui-Wei-Yang-Zhang-2022}, $\Sigma_k$ can be decomposed into a disjoint union of $S_j^k$:
	\begin{equation*}
		\Sigma_k=\{0\}\cup S_1^k\cup\cdots \cup S_{m_0}^k, \ m_0\leq m \ \text{and} \ m_0\in\mathbb{N}\cup\{0\}.
	\end{equation*}
	Each $S_j^k$ forms a group comprising of relatively close elements in $\Sigma_{k}$. Denote by $S_j^k=\{x_{j,1}^k,\cdots,x_{j,m_j}^k\}$ for some $1\leq m_j\leq m$. Let $\tau_{S^k_j}^k$ and $\tau_{j,l}^k$ be defined as
	\begin{equation*}
		\tau_{S^k_j}^k=\frac12\mathrm{dist}(x_{j,1}^k,\Sigma_k\setminus S_j^k) \ \ \text{and} \ \ \tau_{j,l}^k=\frac12\mathrm{dist}(x_{j,l}^k,S_j^k\setminus\{x_{j,l}^k\}), \ \text{for} \ l=1,\cdots,m_j.
	\end{equation*}
	Then $\tau_{j,l}^k\ll\tau_{S^k_j}^k$ for $l=1,\cdots,m_j$. By \textbf{Step 2}, the local mass $\hat{\sigma}_{i}(B(\mathbf{x}_{j,l},\bm{\tau}_{j,l}))=2m_{j,l,i}\in 2\mathbb{N}\cup\{0\}$ for $i\in I$, where $(\mathbf{x}_{j,l},\bm{\tau}_{j,l})$ stands for the sequence of pair $\{(x^k_{j,l},\tau^k_{j,l})\}$ for $1\leq l\leq m_j$.
	
	Let $\sigma_i=\hat{\sigma}_i(B(\mathbf{x}_{j,1},\bm{\tau}_{S^k_j}))$, one can show that $\sigma_i\in2\mathbb{N}\cup\{0\}$ for $i\in I$. In fact, by the same arguments of Proposition \cite[Proposition 5.3]{Cui-Wei-Yang-Zhang-2022}, we would come across one of the following cases:
	
	\noindent \textbf{Case (1).} $\mathbf{u}^k$ has fast decay on $\partial B(x^k_{j,1},\tau^{k}_{j,1})$ (hence on $\partial B(x^k_{j,1},\tau^{k}_{j,l})$ for any $1\leq l\leq m_j$). Let
	\begin{equation*}
		l_j^k(S^k_j)=2\max_{1\leq l\leq m_j}\mathrm{dist}(x_{j,1}^k,x_{j,l}^k).
	\end{equation*}
	Then direct computation shows that $\sigma^k_i(B(x^k_{j,1},l_j^k(S^k_j)))=2\sum_{l=1}^{m_j}m_{j,l,i}+o(1)$ for $i\in I$. If
	\begin{equation*}
		\hat{\sigma}_{i}(B(\mathbf{x}_{j,1},\bm{\tau}_{S^k_j}))=\hat{\sigma}_{i}(B(\mathbf{x}_{j,1},\mathbf{l}_j(S^k_j))), \ \forall \ i\in I,
	\end{equation*}
	where $(\mathbf{x}_{j,1},\mathbf{l}_j(S^k_j))$ and $(\mathbf{x}_{j,1},\bm{\tau}_{S^k_j})$ stand for the sequences $\{(x^k_{j,1},l_j^k(S^k_j))\}$ and $\{(x^k_{j,1},\tau^k_{S^k_j})\}$ respectively, then the proof is complete. Otherwise, there exists some $i_0\in I$ such that
	\begin{equation*}
		\hat{\sigma}_{i_0}(B(\mathbf{x}_{j,1},\bm{\tau}_{S^k_j}))>
		\hat{\sigma}_{i_0}(B(\mathbf{x}_{j,1},\mathbf{l}_j(S^k_j))).
	\end{equation*}
	Then we can apply the same arguments of \textbf{Step 2} to prove that
	\begin{equation*}
		\hat{\sigma}_{i}(B(\mathbf{x}_{j,1},\bm{\tau}_{S^k_j}))
		=\hat{\sigma}_{i}(B(\mathbf{x}_{j,1},\mathbf{s}_{\ell}))\in2\mathbb{N}\cup\{0\}, \ \forall \ i\in I,
	\end{equation*}
	for some sequence $\mathbf{s}_{\ell}=\{s^k_\ell\}$ satisfying that $l_j^k(S^k_j)\ll s^k_\ell\ll \tau^k_{S^k_j}$.
	
	\noindent\textbf{Case (2).} Some components of $\mathbf{u}^k$ have slow decay on $\partial B(x^k_{j,1},\tau^{k}_{j,1})$. We denote by
	\begin{equation*}
		J=\big\{i\in I \mid u^k_i \ \text{has slow decay on} \ \partial B(x^k_{j,1},\tau^{k}_{j,1})\big\}.
	\end{equation*}
	Since $l_j^k(S^k_j)\thicksim\tau^k_{j,l}$ for $1\leq l\leq m_j$, $u^k_i$ has slow decay on $\partial B(x^k_{j,1},l_j^k(S^k_j))$ for $i\in J$. Let
	\begin{equation*}
		v^k_i(y)=u^k_i(x^k_{j,1}+l_j^k(S^k_j)y)+2\log{l_j^k(S^k_j)}, \ \text{for} \ i\in I,
	\end{equation*}
	then $v^k_i(y)\rightarrow-\infty$ in $L^{\infty}_{\mathrm{loc}}(\mathbb{R}^2)$ for $i\in I\setminus J$ and $v^k_i(y)\rightarrow v_i(y)$ in $C_{\mathrm{loc}}^2(\mathbb{R}^2)$ for $i\in J$, where $v_{i}(y)$ satisfies the system
	\begin{equation*}
		\Delta v_{i}(y)+\sum\limits_{t\in J}k_{it}e^{v_t(y)}=4\pi\sum\limits_{l=1}^{m_j}n_{j,l,i}\delta_{q_{j,l}} \ \text{in} \ \mathbb{R}^2, \ \int_{\mathbb{R}^2}e^{v_i(y)}\mathrm{d}y<+\infty, \ \ \forall \ i\in J,
	\end{equation*}
	where $n_{j,l,i}=-\sum_{t\in I}k_{it}m_{j,l,t}\in\mathbb{N}\cup\{0\}$ for $i\in J$. In the subsequent step, we will face the same two options as discussed in \textbf{Step 1}. In any case, we have $\sigma^k_i(B(x^k_{j,1},R_kl_j^{k}(S^k_j)))=2\sum_{l=1}^{m_j}m_{j,l,i}+o(1)$ for $i\in N$ and
	\begin{equation*}
		\sigma^k_i(B(x^k_{j,1},R_k l_j^{k}(S^k_j)))=\hat{\sigma}_{i}(B(\mathbf{x}_{j,1},\mathbf{l}_j(S^k_j)))+\frac{1}{2\pi}\int_{\mathbb{R}^2}e^{v_i(y)}\mathrm{d}y
		=2\Big(\sum\limits_{l=1}^{m_j}m_{j,l,i}+\widetilde{m}_{j,i}\Big)+o(1), \ i\in J.
	\end{equation*}
	By the same arguments of \textbf{Step 2}, we could get the conclusion after finitely many steps.
	\medskip
	
	\noindent\textbf{Step 4.} We calculate the local energy of the bubbling disk centered at $0$. For the previously defined groups $S^k_j$, $1\leq j\leq m_0$, we treat them, along with $\{0\}$, as points and combine the neighboring members into a larger group. By selecting a representative $x_j^k\in S_j^k$ for each $1\leq j\leq m_0$, we can assume that the sets $S_1^k,\cdots,S_{j_0}^k$ satisfy $C^{-1}|x_1^k|\leq|x_j^k|\leq C|x_1^k|$ for $1\leq j\leq j_0$ and $|x_j^k|\gg|x_1^k|$ for $j>j_0$. Now we manage to calculate the possible blow-up local mass in the bubbling disk centered at the origin. Let $\tau_k=\frac{1}{2}$ if $m_0=0$ and $\tau_k=\frac{1}{2}|x^k_1|$ if $m_0>0$. Using arguments similar to those in {\bf Step 1} and {\bf Step 2}, along with the application of \cite[Lemma 6.1]{Cui-Wei-Yang-Zhang-2022} and Theorem \ref{Affine-Toda-Section-2-Theorem-14}, we can find a sequence ${s_k}$ with $s_k\ll\tau_k$ such that $\mathbf{u}^k$ decays rapidly on $\partial B(0,s_k)$ and $\hat{\bm{\sigma}}(\bm{\tau})=\hat{\bm{\sigma}}(\mathbf{s})\in\Gamma_{\mathbf{A}}(\bm{\mu})$.
	\medskip
	
	\noindent\textbf{Step 5.} Finally, we proceed with the process of assembling groups $\{0\}$ and $S^k_j$, $1\leq j\leq m_0$. Specifically, we encounter the following intermediate problem: for the local mass $\bm{\sigma}$ in the present step, let $v^k_i(y)=u^k_i(x^k+s_ky)+2\log{s_k}$ for $i\in I$ and denote by
	\begin{equation*}
		J=\big\{i\in I \mid u^k_i \ \text{has slow decay on} \ \partial B(x^k,s_k)\big\}\subsetneqq I.
	\end{equation*}
	Then, $v^k_i(y)\rightarrow-\infty$ for $i\in I\setminus J$ in $L^{\infty}_{\mathrm{loc}}(\mathbb{R}^2)$ and $v^k_i(y)\rightarrow v_i(y)$ in $C_{\mathrm{loc}}^2(\mathbb{R}^2)$ for $i\in J$, where $v_i(y)$ satisfies the system
	\begin{equation*}
		\Delta v_i(y)+\sum\limits_{j\in J}k_{ij}e^{v_j(y)}=4\pi\alpha_i^{*}\delta_0+4\pi\sum\limits_{l=1}^{N}m_{il}\delta_{q_l} \ \text{in} \ \mathbb{R}^2, \ \int_{\mathbb{R}^2}e^{v_i(y)}\mathrm{d}y<+\infty, \ \ \forall \ i\in J,
	\end{equation*}
	where $0\neq q_{l}\in\mathbb{R}^2$, $m_{il}\in\mathbb{N}$ and $\alpha_i^{*}=\alpha_i-\frac{1}{2}\sum_{j\in I}k_{ij}\sigma_j>-1$ for $i\in J$. Denote by $\mu^{*}_{i}=1+\alpha_{i}^{*}$ for $i\in I$. Set
	\begin{equation*}
		\sigma^{*}_{j}=
		\left\{
		\begin{aligned}
			&\sigma_j, \ && \text{if} \ j\in I\setminus J,\\
			&\sigma_j+\frac{1}{2\pi}\int_{\mathbb{R}^2} e^{v_j(y)}\mathrm{d}y, \ && \text{if} \ j\in J.
		\end{aligned}
		\right.
	\end{equation*}
	Regarding the potential value of ${\bm{\sigma}}^*$ we have the following result:
	
	\begin{lemma}\label{Affine-Toda-Section-3-Lemma-1}
		Let the index set $I$ be decomposed into one of the alternatives \textbf{(I)} and \textbf{(II)} (given in \textbf{Step 1}). Suppose that $\hat{\bm{\sigma}}\in\Gamma_{\mathbf{A}}(\bm{\mu})$ satisfying $\sigma_{i}=\hat{\sigma}_i+2n_i$ and $n_i\in\mathbb{Z}$. Then $\sigma^*_i=\hat{\sigma}^*_i+2n^*_i$ with $\hat{\bm{\sigma}}^{*}\in\Gamma_{\mathbf{A}}(\bm{\mu})$ and $n^*_i\in\mathbb{Z}$.
	\end{lemma}
	
	\begin{proof}
		We consider the following two cases based on the decomposition of the index set $I$.
		
		\noindent\textbf{Case 1.} Alternative \textbf{(I)} happens and (I-1) holds. By the same arguments of \cite[Lemma 6.3]{Lin-Yang-Zhong-2020}, letting $\mathfrak{f}_p$ be a bijective map from $J_p\cup\{i_p-1\}$ to itself, then we get that
		\begin{equation*}
			\frac{1}{2\pi}\int_{\mathbb{R}^2}e^{v_i(y)}\mathrm{d}y=2\sum\limits_{j=i_p-1}^{i-1}
			\Big(\sum\limits_{l=i_p}^{\mathfrak{f}_p(j)}\overline{\mu}_{l}-\sum\limits_{l=i_p}^{j}\overline{\mu}_{l}\Big)+2\widetilde{N}_{i}, \ \text{for some} \ \widetilde{N}_i\in\mathbb{Z}, \ i\in J_p,
		\end{equation*}
		where $\overline{\mu}_{i}=\mu_{i}-\frac{1}{2}\sum_{j\in I}k_{ij}\hat{\sigma}_{j}$ for $i\in I$. We define an extension $f$ of ${\mathfrak{f}_p}$'s on $I_0=I\cup\{0\}$ as follows
		\begin{equation*}
			f(i)=\left\{
			\begin{aligned}
				&\mathfrak{f}_p(i), \  && \text{if} \ i\in\cup_{p=1}^{\vartheta}\big(J_p\cup\{i_p-1\}\big),\\
				&i,                 \  && \text{if} \ i\in I_0\setminus\big(\cup_{p=1}^{\vartheta}\big(J_p\cup\{i_p-1\}\big)\big).
			\end{aligned}
			\right.
		\end{equation*}
		As a consequence, $\bm{\sigma}^{*}$ can be rewritten as
		\begin{equation*}
			\sigma_i^{*}=\hat{\sigma}_i+2\sum\limits_{j=0}^{i-1}
			\Big(\sum\limits_{l=1}^{f(j)}\overline{\mu}_{l}-\sum\limits_{l=1}^{j}\overline{\mu}_{l}\Big)+2n_{i}^{*}=\hat{\sigma}_i^{*}+2n_{i}^{*}, \ i\in I,
		\end{equation*}
		for some $n_{i}^{*}\in\mathbb{Z}$, $i\in I$, where
		\begin{equation*}
			\hat{\sigma}_i^{*}=\hat{\sigma}_i+2\sum\limits_{j=0}^{i-1}
			\Big(\sum\limits_{l=1}^{f(j)}\overline{\mu}_{l}-\sum\limits_{l=1}^{j}\overline{\mu}_{l}\Big), \ \text{for} \ i\in I.
		\end{equation*}
		It remains to show that $\hat{\bm{\sigma}}^{*}\in\Gamma_{\mathbf{A}}(\bm{\mu})$. Let $g$ be the simple permutation which only alternates the places of $\ell-1$ and $\ell$ for some $\ell\in I$, one can apply the same argument of \cite[Theorem 3.6]{Lin-Yang-Zhong-2020} to check that $\bm{\sigma}_{f\circ g}=\mathfrak{R}_{\ell}\bm{\sigma}_{f}$. Using the fact that $f$ can be represented by a composition of simple permutations, we conclude that $\hat{\bm{\sigma}}^{*}\in\Gamma_{\mathbf{A}}(\bm{\mu})$ since $\bm{\sigma}\in\Gamma_{\mathbf{A}}(\bm{\mu})$. This completes the proof of Lemma \ref{Affine-Toda-Section-3-Lemma-1}.
		
		\noindent\textbf{Case 2.} Alternative \textbf{(II)} happens and (II-1) holds. Suppose that $f$ is the $r_2^{+}$-permutation on $I$ satisfying \eqref{Affine-Toda-Section-2-Eq-32} with $r=r_2$. Since $\hat{\bm{\sigma}}\in\Gamma_{\mathbf{A}}(\bm{\mu})$, we deduce by Lemma \ref{Affine-Toda-Section-2-Lemma-13}-(a) that $\hat{\bm{\sigma}}_{f}=\left(\hat{\sigma}_{f,1},\cdots,\hat{\sigma}_{f,n+1}\right)\in\Gamma_{\mathbf{A}}(\bm{\mu})$, where $\hat{\sigma}_{f,i}=\hat{\sigma}_{f(i)}$ for $i\in I$. Let $\overline{J_0}$, $\overline{J_p}$ and $\overline{N}$ be defined as in \eqref{Affine-Toda-Section-2-Eq-35},  it is easy to see that $I=\overline{J_0}\cup \overline{J_1}\cup\cdots\cup\overline{J_{\vartheta}}\cup\overline{N}$. Set
		\begin{equation*}
			{{\sigma}}^{*}_{f,i}=
			\left\{
			\begin{aligned}
				& {\sigma}_{f,i}, \ && \text{if} \ i\in \overline{N},\\
				& {\sigma}_{f,i}+\frac{1}{2\pi}\int_{\mathbb{R}^2} e^{v_{f(i)}}\mathrm{d}y, \ && \text{if} \ i\in f^{-1}(J)=:\overline{J}=\overline{J_0}\cup \overline{J_1}\cup\cdots\cup\overline{J_{\vartheta}}.
			\end{aligned}
			\right.
		\end{equation*}
		Then we can argue as in \textbf{Case 1} to obtain that
		\begin{equation*}
			{\sigma}^{*}_{f,i}=\hat{\sigma}^{*}_{f,i}+2n_{f,i}, \ \text{for some} \ n_{f,i}\in\mathbb{Z}, \ i\in I,
		\end{equation*}
		where $\hat{\bm{\sigma}}^{*}_{f}=(\hat{\sigma}^{*}_{f,1},\cdots,\hat{\sigma}^{*}_{f,n+1})\in\Gamma_{\mathbf{A}}(\bm{\mu})$. It is not difficult to check that $\hat{\bm{\sigma}}^{*}_{f}=(\hat{\bm{\sigma}}^{*})_{f}$. As a result of Lemma \ref{Affine-Toda-Section-2-Lemma-13}-(b) we get that $\hat{\bm{\sigma}}^{*}\in\Gamma_{\mathbf{A}}(\bm{\mu})$. This finishes the proof of Lemma \ref{Affine-Toda-Section-3-Lemma-1}.
	\end{proof}
	
	\begin{proof}[Proof of Theorem \ref{Affine-Toda-Section-1-Theorem-1}]
		Analogously to the proof of \cite[Theorem 1.1]{Cui-Wei-Yang-Zhang-2022}, one can obtain the desired result by utilizing Lemma \ref{Affine-Toda-Section-3-Lemma-1} and the discussion from {\bf Step 1} to {\bf Step 5}.
	\end{proof}
	
	\section{The affine $\mathbf{C}^t$ type Toda system}\label{Affine-Toda-Section-4}
	\setcounter{equation}{0}
	In this section, we focus on the affine $\mathbf{C}^{t}$ type Toda system \eqref{Affine-Toda-system-C_n^1-form}. Using a similar approach, we investigate the bubbling areas and derive the Pohozaev identity for local mass, which is given by:
	\begin{equation}\label{Affine-Toda-Section-4-Eq-1}
		(\sigma_1-\sigma_2)^{2}+(\sigma_2-\sigma_3)^{2}+\cdots+(\sigma_n-\sigma_{n+1})^{2}
		=2\left(\mu_1\sigma_1+2\sum\limits_{i\in I\setminus\{1,n+1\}}\mu_i\sigma_i+\mu_{n+1}\sigma_{n+1}\right).
	\end{equation}
	We shall obtain the group representation for the ${\bm\sigma}$ by the {\em affine Weyl group} $\mathbf{G}_{\mathbf{C}^{t}}$.
	
	Let $J=\{i,i+1,\cdots,i+l\}\subsetneqq I$ for some $i\in I$ and $l\in\mathbb{N}\cup\{0\}$. Throughout this section, we say that $\mathbf{u}=(u_i,\cdots,u_{i+l})$ is the solution to the $\mathbf{A}_{l+1}$ (or $\mathbf{C}_{l+1}$) Toda system if it satisfies the system
	\begin{equation}\label{Affine-Toda-Section-4-Eq-AC}
		\left\{
		\begin{aligned}
			&\Delta u_{s}+\sum\limits_{t\in J}k_{st}^{\prime}e^{u_t}=4\pi\alpha_s\delta_0 \ \text{in} \ \mathbb{R}^2,\\
			&\int_{\mathbb{R}^2}e^{u_s(y)}\mathrm{d}y<+\infty, \ \forall \ s\in J.
		\end{aligned}
		\right.
	\end{equation}
	Here the coefficient matrix $(k_{st}^{\prime})$ in \eqref{Affine-Toda-Section-4-Eq-AC} is chosen as the Cartan matrix of type $\mathbf{A}$ (or $\mathbf{C}$) with rank $l+1$ and $\alpha_s>-1$ for $s\in J$. Particularly, we can also say \eqref{Affine-Toda-Section-4-Eq-AC} is a $\mathbf{C}_{l+1}$-type Toda system ($l\geq 1$) if the coefficient matrix $(k_{st}^{\prime})$ is of the form
	\begin{equation*}
		(k_{st}^{\prime})=
		\begin{pmatrix}
			2      &  -2    &    0   &  \cdots  &  0     &    0    &    0        \\
			-1     &  2     &    -1  &  \cdots  &  0     &    0    &    0        \\
			0      &  -1    &    2   &  \ddots  &  0     &    0    &    0        \\
			\vdots & \vdots & \ddots &  \ddots  & \ddots & \vdots  &  \vdots     \\
			0      &   0    &    0   &  \ddots  &    2   &   -1    &    0         \\
			0      &   0    &    0   &  \cdots  &    -1  &   2     &    -1        \\
			0      &   0    &    0   &  \cdots  &    0   &   -1    &    2         \\
		\end{pmatrix}.
	\end{equation*}
	Hereafter, $(k_{st})_{|J|\times |J|}$ represents a segment of the generalized Cartan matrix of the affine $\mathbf{C}^t$ type unless explicitly mentioned otherwise.
	
	\subsection{Bubbling analysis and Pohozaev identity}
	In this subsection, we perform the selection process and establish the Pohozaev identity \eqref{Affine-Toda-Section-4-Eq-1}.
	
	\begin{proposition}\label{Affine-Toda-Section-4-Proposition-1}
		Let $\mathbf{u}^{k}$ be a sequence of solutions of system \eqref{Affine-Toda-system-C_n^1-form} satisfying \eqref{Affine-Toda-3Conditions}. Then there exists a sequence of finite points $\Sigma_{k}:=\{0,x^{k}_{1},\cdots,x^{k}_{m}\}$ (if $0$ is not a singular point, then $0$ can be deleted from $\Sigma_k$) and a sequence of positive numbers $l^{k}_{1},\cdots,l^{k}_{m}$ such that the following properties hold:\\
		(1) There exists a constant $C$ independent of $k$ such that the Harnack-type inequality holds:
		\begin{equation*}
			\max\limits_{i\in I}\Big\{u^k_i(x)+2\log\mathrm{dist}(x,\Sigma_k)\Big\}\leq C, \ \forall \ x\in B_{1}(0).
		\end{equation*}
		(2) $x^{k}_{j}\rightarrow 0$ and $l^{k}_{j}\rightarrow 0$ as $k\rightarrow +\infty$, $l^{k}_{j}\leq\tau^{k}_j:=\frac{1}{2}\mathrm{dist}(x^{k}_{j},\Sigma_{k}\setminus\{x^{k}_{j}\})$, $j=1,\cdots,m$. Furthermore, $B(x^{k}_{i},l_{i}^{k})\cap B(x^{k}_{j},l_{j}^{k})=\emptyset$ for $1\leq i,j\leq m$, $i\neq j$.\\
		(3) $\max\limits_{i\in I}u^k_{i}(x^{k}_{j})=\max\limits_{i\in I}\max\limits_{B(x^k_{j},l^k_j)}u^k_i(x)\rightarrow+\infty$ as $k\rightarrow+\infty$, $j=1,\cdots,m$. Denote by
		\begin{equation*}
			\varepsilon_{j}^{k}:=\exp\left(-\frac{1}{2}\max\limits_{i\in I}u^{k}_{i}(x^{k}_{j})\right), \ j=1,\cdots,m.
		\end{equation*}
		Then $R^k_j:={l^{k}_{j}}/{\varepsilon_{j}^{k}}\rightarrow+\infty$ as $k\rightarrow+\infty$, $j=1,\cdots,m$.\\
		(4) In each $B(x^{k}_{j},l_{j}^{k})$, set $v^k_{i}(y):=u^k_i(x^k_j+\varepsilon_{j}^{k}y)+2\log\varepsilon_{j}^{k}$ for $i\in I$. Then one of the following alternatives holds:\\
		\noindent\textbf{(I).} The index set $I$ is decomposed into the form $I=J\cup N$ and $J=J_0\cup J_1\cup\cdots\cup J_{\vartheta}$ for some $\vartheta\in\mathbb{N}\cup\{0\}$ satisfying that
		\begin{enumerate}[(\text{I}-1).]
			\item $J_0,J_1,\cdots,J_{\vartheta}$ and $N$ are pairwise disjoint, $J_0\neq\emptyset$, $N\neq\emptyset$, $\{1,2\}\cap N=\emptyset$, $\{n,n+1\}\cap N\neq\emptyset$ and each $J_p$, $p=0,1,\cdots,\vartheta$ consists of the maximal consecutive indices, denoted by
			\begin{equation*}
				J_p=\{i_p,\cdots,i_p+l_p\}, \ \text{for some} \ i_p\in I \ \text{and} \ l_p\in\mathbb{N}\cup\{0\}.
			\end{equation*}
			In particular, $i_0=1$. Naturally, $n\geq 2$ is needed;
			\item $v^k_s(y)\rightarrow-\infty$ in $L^{\infty}_{\mathrm{loc}}(\mathbb{R}^2)$ for $s\in N$;
			\item $v^k_s(y)\rightarrow v_s(y)$ in $C_{\mathrm{loc}}^2(\mathbb{R}^2)$ for $s\in J_p$, $p=1,\cdots,\vartheta$, where $v_s(y)$ satisfies the $\mathbf{A}_{l_p+1}$ Toda system \eqref{Affine-Toda-Section-4-Eq-AC} with $\alpha_{s}=0$ and $(k_{st}^{\prime})=(k_{st}^{\mathbf{c}^t})$, $s,t\in J_{p}$;
			\item $v^k_s(y)\rightarrow v_s(y)$ in $C_{\mathrm{loc}}^2(\mathbb{R}^2)$ for $s\in J_0$, where $v_s(y)$ satisfies the $\mathbf{C}_{l_0+1}$-type Toda system \eqref{Affine-Toda-Section-4-Eq-AC} with $\alpha_{s}=0$ and $(k_{st}^{\prime})=(k_{st}^{\mathbf{c}^t})$, $s,t\in J_{0}$.
		\end{enumerate}
		
		\noindent\textbf{(II).} The index set $I$ is decomposed into the form $I=J\cup N$ and $J=J_0\cup J_1\cup\cdots\cup J_{\vartheta}$ for some $\vartheta\in\mathbb{N}\cup\{0\}$ satisfying that
		\begin{enumerate}[(\text{II}-1).]
			\item $J_0,J_1,\cdots,J_{\vartheta}$ and $N$ are pairwise disjoint, $J_0\neq\emptyset$, $N\neq\emptyset$, $\{1,2\}\cap N\neq\emptyset$, $\{n,n+1\}\cap N=\emptyset$ and each $J_p$, $p=0,1,\cdots,\vartheta$ consists of the maximal consecutive indices, denoted by
			\begin{equation*}
				J_p=\{i_p,\cdots,i_p+l_p\}, \ \text{for some} \ i_p\in I \ \text{and} \ l_p\in\mathbb{N}\cup\{0\}.
			\end{equation*}
			In particular, $i_0+l_0=n+1$. Naturally, $n\geq 2$ is needed;
			\item $v^k_s(y)\rightarrow-\infty$ in $L^{\infty}_{\mathrm{loc}}(\mathbb{R}^2)$ for $s\in N$;
			\item $v^k_s(y)\rightarrow v_s(y)$ in $C_{\mathrm{loc}}^2(\mathbb{R}^2)$ for $s\in J_p$, $p=1,\cdots,\vartheta$, where $v_s(y)$ satisfies the $\mathbf{A}_{l_p+1}$ Toda system \eqref{Affine-Toda-Section-4-Eq-AC} with $\alpha_{s}=0$ and $(k_{st}^{\prime})=(k_{st}^{\mathbf{c}^t})$, $s,t\in J_{p}$;
			\item $v^k_s(y)\rightarrow v_s(y)$ in $C_{\mathrm{loc}}^2(\mathbb{R}^2)$ for $s\in J_0$, where $v_s(y)$ satisfies the $\mathbf{C}_{l_0+1}$ Toda system \eqref{Affine-Toda-Section-4-Eq-AC} with $\alpha_{s}=0$ and $(k_{st}^{\prime})=(k_{st}^{\mathbf{c}^t})$, $s,t\in J_{0}$.
		\end{enumerate}
		
		\noindent\textbf{(III).} The index set $I$ is decomposed into the form $I=J\cup N$ and $J=J_0\cup J_1\cup\cdots\cup J_{\vartheta}\cup J_{\vartheta+1}$ for some $\vartheta\in\mathbb{N}\cup\{0\}$ satisfying that
		\begin{enumerate}[(\text{III}-1).]
			\item $J_0,J_1,\cdots,J_{\vartheta+1}$ and $N$ are pairwise disjoint, $J_0\neq\emptyset$, $J_{\vartheta+1}\neq\emptyset$, $N\neq\emptyset$, $\{1,2\}\cap N=\emptyset$, $\{n,n+1\}\cap N=\emptyset$ and each $J_p$, $p=0,1,\cdots,\vartheta+1$ consists of the maximal consecutive indices, denoted by
			\begin{equation*}
				J_p=\{i_p,\cdots,i_p+l_p\}, \ \text{for some} \ i_p\in I \ \text{and} \ l_p\in\mathbb{N}\cup\{0\}.
			\end{equation*}
			In particular, $i_0=1$, $i_{\vartheta+1}+l_{\vartheta+1}=n+1$. Naturally, $n\geq 4$ is needed;
			\item $v^k_s(y)\rightarrow-\infty$ in $L^{\infty}_{\mathrm{loc}}(\mathbb{R}^2)$ for $s\in N$;
			\item $v^k_s(y)\rightarrow v_s(y)$ in $C_{\mathrm{loc}}^2(\mathbb{R}^2)$ for $s\in J_p$, $p=1,\cdots,\vartheta$, where $v_s(y)$ satisfies the $\mathbf{A}_{l_p+1}$ Toda system \eqref{Affine-Toda-Section-4-Eq-AC} with $\alpha_{s}=0$ and $(k_{st}^{\prime})=(k_{st}^{\mathbf{c}^t})$, $s,t\in J_{p}$;
			\item $v^k_s(y)\rightarrow v_s(y)$ in $C_{\mathrm{loc}}^2(\mathbb{R}^2)$ for $s\in J_0$, where $v_s(y)$ satisfies the $\mathbf{C}_{l_0+1}$-type Toda system \eqref{Affine-Toda-Section-4-Eq-AC} with $\alpha_{s}=0$ and $(k_{st}^{\prime})=(k_{st}^{\mathbf{c}^t})$, $s,t\in J_{0}$;
			\item $v^k_s(y)\rightarrow v_s(y)$ in $C_{\mathrm{loc}}^2(\mathbb{R}^2)$ for $s\in J_{\vartheta+1}$, where $v_s(y)$ satisfies the $\mathbf{C}_{l_{\vartheta+1}+1}$ Toda system \eqref{Affine-Toda-Section-4-Eq-AC} with $\alpha_{s}=0$ and $(k_{st}^{\prime})=(k_{st}^{\mathbf{c}^t})$, $s,t\in J_{\vartheta+1}$.
		\end{enumerate}
		
		\noindent\textbf{(IV).} The index set $I$ is decomposed into the form $I=J\cup N$ and $J=J_1\cup\cdots\cup J_{\vartheta}$ for some $\vartheta\in\mathbb{N}$ satisfying that
		\begin{enumerate}[(\text{IV}-1).]
			\item $J_1,\cdots,J_{\vartheta}$ and $N$ are pairwise disjoint, $N\neq\emptyset$, $\{1,2\}\cap N\neq\emptyset$, $\{n,n+1\}\cap N\neq\emptyset$ and each $J_p$, $p=1,\cdots,\vartheta$ consists of the maximal consecutive indices, denoted by
			\begin{equation*}
				J_p=\{i_p,\cdots,i_p+l_p\}, \ \text{for some} \ i_p\in I \ \text{and} \ l_p\in\mathbb{N}\cup\{0\}.
			\end{equation*}
			Naturally, $n\geq 2$ is needed;
			\item $v^k_s(y)\rightarrow-\infty$ in $L^{\infty}_{\mathrm{loc}}(\mathbb{R}^2)$ for $s\in N$;
			\item $v^k_s(y)\rightarrow v_s(y)$ in $C_{\mathrm{loc}}^2(\mathbb{R}^2)$ for $s\in J_p$, $p=1,\cdots,\vartheta$, where $v_s(y)$ satisfies the $\mathbf{A}_{l_p+1}$ Toda system \eqref{Affine-Toda-Section-4-Eq-AC} with $\alpha_{s}=0$ and $(k_{st}^{\prime})=(k_{st}^{\mathbf{c}^t})$, $s,t\in J_{p}$.
		\end{enumerate}
	\end{proposition}
	
	Now we establish the Pohozaev identity for any blow-up sequence of solutions of \eqref{Affine-Toda-system-C_n^1-form}. In Proposition \ref{Affine-Toda-Section-4-Proposition-2}, we denote by $\sigma_{i}^{k}(r,x_0):=\sigma(r,x_0;u^{k}_{i})$ and $\sigma_{i}^{k}(r):=\sigma(r,0;u^{k}_{i})$ for $i\in I$, where $\sigma(r,x_0;u)$ is defined as in \eqref{Affine-Toda-Section-2-Eq-39}.
	
	\begin{proposition}\label{Affine-Toda-Section-4-Proposition-2}
		Let $\mathbf{u}^{k}=(u^{k}_{1},\cdots,u^{k}_{n+1})$ be a sequence of solutions of system \eqref{Affine-Toda-system-C_n^1-form} satisfying \eqref{Affine-Toda-3Conditions}. Assume that $\Sigma_{k}^{\prime}\subseteq\Sigma_{k}$ is a subset of $\Sigma_{k}$ with $0\in\Sigma_{k}^{\prime}\subseteq B(x_{k},r_{k})\subseteq B_{1}(0)$ and there holds
		\begin{equation*}
			\mathrm{dist}(\Sigma_{k}^{\prime},\partial B(x_{k},r_{k}))=o(1)\mathrm{dist}(\Sigma_{k}\setminus\Sigma_{k}^{\prime},\partial B(x_{k},r_{k})).
		\end{equation*}
		If $\mathbf{u}^{k}$ has fast decay on $\partial B(x_{k},r_{k})$, then
		\begin{small}
			\begin{equation}\label{Affine-Toda-Section-4-Eq-2-Pohozaev}
				\sum\limits_{i=1,n+1}\sigma_{i}\Big(\sum\limits_{j\in I}k_{ij}^{\mathbf{c}^t}\sigma_{j}-2\mu_{i}\Big)+2\sum\limits_{i\in I\setminus\{1,n+1\}}\sigma_{i}\Big(\sum\limits_{j\in I}k_{ij}^{\mathbf{c}^t}\sigma_{j}-2\mu_{i}\Big)=2\Big(\sum\limits_{i=1,n+1}\mu_i\sigma_i+2\sum\limits_{i\in I\setminus\{1,n+1\}}\mu_i\sigma_i\Big),
			\end{equation}
		\end{small}
		where $(k_{ij}^{\mathbf{c}^t})$ is defined as in \eqref{Affine-C_n^1-Cartan-matrix}, $\mu_i=1+\alpha_i$ and $\sigma_i=\lim\limits_{k\rightarrow+\infty}\sigma^k_i(r_k,x_k)$ for $i\in I$.
	\end{proposition}
	\begin{proof}
		We can transform system \eqref{Affine-Toda-system-C_n^1-form} by carrying out the following operations:
		\begin{equation}\label{Affine-Toda-Section-4-Eq-3}
			v_i=v_{2n+2-i}=u_i, \ \beta_{i}=\beta_{2n+2-i}=\alpha_{i}, \ \text{for} \ i\in I.
		\end{equation}
		This transformation embeds the affine $\mathbf{C}^{(1),t}_{n}$  Toda system \eqref{Affine-Toda-system-C_n^1-form} into an affine $\mathbf{A}^{(1)}_{2n-1}$ Toda system. Precisely, $\mathbf{v}=(v_1,\cdots,v_{2n})$ satisfies the system
		\begin{equation}\label{Affine-Toda-Section-4-Eq-4}
			\begin{cases}\Delta v_i+\sum\limits_{j=1}^{2n}\hat{k}_{ij}e^{v_j}=4\pi\beta_i\delta_0 \ \text{in} \ B_1(0), \ \forall \ i=1,\cdots,2n,\\
				v_1+v_2+\cdots+v_{2n}\equiv0,
			\end{cases}
		\end{equation}
		where $\beta_i>-1$ for $i=1,\cdots,2n$ satisfying that $\sum_{i=1}^{2n}\beta_i=0$ and $(\hat{k}_{ij})_{2n\times 2n}$ stands for the generalized Cartan matrix $\mathbf{A}^{(1)}_{2n-1}$. Without loss of generality, we assume that $x_k=0$ and define $\mathbf{v}^{k}$ as shown in \eqref{Affine-Toda-Section-4-Eq-3} with $\mathbf{u}=\mathbf{u^{k}}$. Denoting by $\bm{\sigma}_{\mathbf{v}}=(\sigma_{\mathbf{v},1},\cdots,\sigma_{\mathbf{v},2n})$ with $\sigma_{\mathbf{v},i}:=\lim\limits_{k\rightarrow+\infty}\sigma(r_k,0;v^{k}_{i})$ for $i=1,\cdots,2n$. Then by Proposition \ref{Affine-Toda-Section-2-Proposition-5}, $\bm{\sigma}_{\mathbf{v}}$ satisfies the Pohozaev identity \eqref{Affine-Toda-Section-2-Eq-6-Pohozaev}, we rewrite it as follows:
		\begin{equation*}
			\begin{aligned}
				&(\sigma_{\mathbf{v},1}-\sigma_{\mathbf{v},2})^{2}+(\sigma_{\mathbf{v},2}-\sigma_{\mathbf{v},3})^{2}+\cdots
				+(\sigma_{\mathbf{v},2n-1}-\sigma_{\mathbf{v},2n})^{2}+(\sigma_{\mathbf{v},2n}-\sigma_{\mathbf{v},1})^{2}\\
				&=4\left((1+\beta_1)\sigma_{\mathbf{v},1}+\cdots+(1+\beta_{2n})\sigma_{\mathbf{v},2n}\right),
			\end{aligned}
		\end{equation*}
		which can be reduced to \eqref{Affine-Toda-Section-4-Eq-1}. Therefore, \eqref{Affine-Toda-Section-4-Eq-2-Pohozaev} holds and it finishes the proof of Proposition \ref{Affine-Toda-Section-4-Proposition-2}.
	\end{proof}
	
	The following property of the set $\Gamma_{\mathbf{C}^t}(\bm{\mu})$ is trivial.
	
	\begin{proposition}\label{Affine-Toda-Section-4-Proposition-3}
		For each element $\bm{\sigma}\in\Gamma_{\mathbf{C}^t}(\bm{\mu})$, $\bm{\sigma}$ satisfies the Pohozaev identity \eqref{Affine-Toda-Section-4-Eq-1}.
	\end{proposition}
	\begin{proof}
		One can verify this straightforwardly by using Vieta's Theorem.
	\end{proof}
	
	\subsection{The group representation}
	In this subsection, we use some elements in group $\mathbf{G}_{\mathbf{C}^t}$ to represent the possible local mass. Suppose that $J_0=\{i,i+1,\cdots,i+l\}\subsetneqq I$ consists of consecutive indices for some $i\in I$ and $l\in\mathbb{N}$.
	
	\begin{lemma}\label{Affine-Toda-Section-4-Lemma-6}
		Let $\bm{\sigma}=(\sigma_1,\cdots,\sigma_{n+1})\in\Gamma_{\mathbf{C}^t}(\bm{\mu})$ and denote by $\alpha_{s}^{*}=\alpha_{s}-\frac{1}{2}\sum_{t\in I}k_{st}^{\mathbf{c}^{t}}\sigma_{t}$ and $\mu_{s}^{*}=1+\alpha_{s}^{*}$ for $s\in I$, where $(k_{st}^{\mathbf{c}^{t}})$ is defined as in \eqref{Affine-C_n^1-Cartan-matrix}.
		\begin{enumerate}[(a)]
			\item Assume that $i=1$. Let $\mathbf{v}=(v_1,\cdots,v_{l+1})$ be the solution of the $\mathbf{C}_{l+1}$-type Toda system \eqref{Affine-Toda-Section-4-Eq-AC} with $\alpha_s=\alpha_{s}^{*}$ and $(k_{st}^{\prime})=(k_{st}^{\mathbf{c}^{t}})$, $s,t\in J_0$.
			\item Assume that $i+l=n+1$. Let $\mathbf{v}=(v_i,\cdots,v_{n+1})$ be the solution of the $\mathbf{C}_{l+1}$ Toda system \eqref{Affine-Toda-Section-4-Eq-AC} with $\alpha_s=\alpha_{s}^{*}$ and $(k_{st}^{\prime})=(k_{st}^{\mathbf{c}^{t}})$, $s,t\in J_0$.
		\end{enumerate}
		In any case, we denote by $\sigma_{s}^{*}=\sigma_{s}$ for $s\in I\setminus J_0$ and
		\begin{equation*}
			\sigma_{s}^{*}=\sigma_{s}+\frac{1}{2\pi}\int_{\mathbb{R}^2}e^{v_s(y)}\mathrm{d}y, \ \ \text{for} \ s\in J_0.
		\end{equation*}
		Then $\bm{\sigma}^{*}=(\sigma_1^{*},\cdots,\sigma_{n+1}^{*})\in\Gamma_{\mathbf{C}^t}(\bm{\mu})$.
	\end{lemma}
	
	In order to establish Lemma \ref{Affine-Toda-Section-4-Lemma-6}, it is necessary to search for an element $\mathfrak{R}_{J_0}\in\mathbf{G}_{\mathbf{C}^t}$ that satisfies $\bm{\sigma}^{*}=\mathfrak{R}_{J_0}\bm{\sigma}\in\Gamma_{\mathbf{C}^t}(\bm{\mu})$. To facilitate this, we introduce the corresponding concept of a \textbf{chain of set}.
	
	\begin{definition}\label{Affine-Toda-Section-4-Definition-7}
		Suppose that $J\subsetneqq I$ consists of consecutive indices, denoted by $J=\{j,j+1,\cdots,j+l\}$ for some $j\in I$ and $l\in\mathbb{N}\cup\{0\}$. We define the \textbf{$J$-chain}, denoted by $\mathfrak{R}_{J}$, as follows:
		\begin{enumerate}[(a).]
			\item if $l=0$, $|J|=1$, then $\mathfrak{R}_{J}=\mathfrak{R}_{j}$;
			\item if $l\geq 1$, $|J|\geq 2$, then
			\begin{enumerate}[(\text{b}-1).]
				\item if $j=1$ (naturally, $j+l<n+1$), then $\mathfrak{R}_{J}=(\mathfrak{R}_{j+l}\cdots\mathfrak{R}_{j+1}\mathfrak{R}_{j})^{l+1}$;
				\item if $j+l=n+1$ (naturally, $j>1$), then $\mathfrak{R}_{J}=\left(\mathfrak{R}_{j}\mathfrak{R}_{j+1}\cdots\mathfrak{R}_{j+l}\right)^{l+1}$;
				\item if $1<j<j+l<n+1$, then $\mathfrak{R}_{J}$ is defined as in Definition \ref{Affine-Toda-Section-2-Definition-11}.
			\end{enumerate}
		\end{enumerate}
	\end{definition}
	
	\begin{rmk}[remark of Definition \ref{Affine-Toda-Section-4-Definition-7}]
		Subsequent to each blow-up process, if $\bm{\sigma}$ denotes by the local mass at the previous stage, then the updated local mass $\bm{\sigma}^{*}$ can be represented as $\bm{\sigma}^{*}=\mathfrak{R}_{J}\bm{\sigma}$ for some $J\subsetneqq I$.
	\end{rmk}
	
	\begin{proof}[Proof of Lemma \ref{Affine-Toda-Section-4-Lemma-6}]
		We shall check that $\bm{\sigma}^{*}=\mathfrak{R}_{J_0}\bm{\sigma}\in\Gamma_{\mathbf{C}^t}(\bm{\mu})$, where $\mathfrak{R}_{J_0}$ is the $J_0$-chain defined as in Definition \ref{Affine-Toda-Section-4-Definition-7}. The discussion is split into two parts, each corresponds to the conclusions (a) and (b)  respectively:
		\medskip
		
		\noindent{\em Proof of (a)}. The proof is split into several steps.
		
		\noindent\textbf{Step 1.} We calculate the mass of each $v_s$, $s\in J_0$. Denote by $\tilde{\mathbf{v}}=(\tilde{v}_1,\tilde{v}_2,\cdots,\tilde{v}_{l+1})$, where $\tilde{v}_i=v_{l+2-i}$ and set $\tilde{\alpha}_i^{*}=\alpha_{l+2-i}^{*}$ for $i\in J_0$. Then $\tilde{\mathbf{v}}$ satisfies the $\mathbf{C}_{l+1}$ Toda system \eqref{Affine-Toda-Section-4-Eq-AC} with $\alpha_{i}=\tilde{\alpha}_i^{*}$ and $(k_{ij}^{\prime})=(c_{ij})$, $i,j\in J_0$, where $(c_{ij})_{|J_0|\times|J_0|}$ stands for the $\mathbf{C}$ type Cartan matrix of rank $l+1$. As we know, any $\mathbf{C}_{l+1}$ Toda system is embedded into an $\mathbf{A}_{2l+1}$ Toda system, i.e., letting
		\begin{equation*}
			u_i=u_{2l+2-i}=\tilde{v}_i, \ \beta_{i}=\beta_{2l+2-i}=\tilde{\alpha}_i^{*}, \ i=1,\cdots,l+1,
		\end{equation*}
		then $\mathbf{u}=(u_1,u_2,\cdots,u_{2l+1})$ satisfies the $\mathbf{A}_{2l+1}$ Toda system
		\begin{equation*}
			\Delta u_i+\sum\limits_{j=1}^{2l+1}a_{ij}e^{u_j}=4\pi\beta_i\delta_0  \ \text{in} \ \mathbb{R}^2, \ \
			\int_{\mathbb{R}^2}e^{u_i(y)}\mathrm{d}y<+\infty, \ \ \forall \ i=1,2,\cdots,2l+1,
		\end{equation*}
		where $(a_{ij})_{(2l+1)\times(2l+1)}$ stands for the $\mathbf{A}$ type Cartan matrix of rank $2l+1$. As the \textbf{Step 2} in the proof of Lemma \ref{Affine-Toda-Section-2-Lemma-10}, we extract that for $i\in J_0=\{1,\cdots,l+1\}$,
		\begin{equation}\label{Affine-Toda-Section-4-Eq-17}
			\begin{aligned}
				\sigma_{\tilde{v}_i}&=\sigma_{u_i}=\frac{1}{2\pi}\int_{\mathbb{R}^2}e^{u_i(y)}\mathrm{d}y
				=2\sum\limits_{q=0}^{i-1}
				\Big(\sum\limits_{t=1}^{2l+1-q}(1+\beta_t)-\sum\limits_{t=1}^{q}(1+\beta_t)\Big)\\
				&=2\sum\limits_{q=0}^{i-1}\sum\limits_{t=1}^{l+1}(1+\beta_t)+2\sum\limits_{q=0}^{i-1}\sum\limits_{t=l+2}^{2l+1-q}(1+\beta_t)
				-2\sum\limits_{q=0}^{i-1}\sum\limits_{t=1}^{q}(1+\beta_t)\\
				&=2\sum\limits_{q=0}^{i-1}\sum\limits_{t=1}^{l+1}(1+\tilde{\alpha}_t^{*})
				+2\sum\limits_{q=0}^{i-1}\sum\limits_{t=l+2}^{2l+1-q}(1+\tilde{\alpha}_{2l+2-t}^{*})
				-2\sum\limits_{q=0}^{i-1}\sum\limits_{t=1}^{q}(1+\tilde{\alpha}_t^{*})\\
				&=2\sum\limits_{q=0}^{i-1}\sum\limits_{t=1}^{l+1}(1+{\alpha}_{l+2-t}^{*})
				+2\sum\limits_{q=0}^{i-1}\sum\limits_{t=l+2}^{2l+1-q}(1+{\alpha}_{t-l}^{*})
				-2\sum\limits_{q=0}^{i-1}\sum\limits_{t=1}^{q}(1+{\alpha}_{l+2-t}^{*})\\
				&=:I+II-III.
			\end{aligned}
		\end{equation}
		We calculate $I$, $II$ and $III$ term by term. Direct computation shows that
		\begin{equation*}
			\begin{aligned}
				I&=\sum\limits_{q=0}^{i-1}2\sum\limits_{t=1}^{l+1}(1+{\alpha}_{l+2-t}^{*})
				=2\sum\limits_{q=0}^{i-1}\sum\limits_{t=1}^{l+1}\mu_{l+2-t}
				-\sum\limits_{q=0}^{i-1}\sum\limits_{t=1}^{l+1}\sum\limits_{j\in I}k_{l+2-t,j}^{\mathbf{c}^{t}}\sigma_{j}\\
				&=2\sum\limits_{q=0}^{i-1}\sum\limits_{t=1}^{l+1}\mu_{l+2-t}-i(\sigma_{1}-\sigma_{2}+\sigma_{l+1}-\sigma_{l+2}),
			\end{aligned}
		\end{equation*}
		\begin{equation*}
			\begin{aligned}
				II&=2\sum\limits_{q=0}^{i-1}\sum\limits_{t=l+2}^{2l+1-q}(1+{\alpha}_{t-l}^{*})
				=2\sum\limits_{q=0}^{i-1}\sum\limits_{t=l+2}^{2l+1-q}\mu_{t-l}-\sum\limits_{q=0}^{i-1}\sum\limits_{t=l+2}^{2l+1-q}\sum\limits_{j\in I}k_{t-l,j}^{\mathbf{c}^{t}}\sigma_j\\
				&=2\sum\limits_{q=0}^{i-1}\sum\limits_{t=l+2}^{2l+1-q}\mu_{t-l}-\sum\limits_{q=1}^{i-1}\sum\limits_{t=l+2}^{2l+1-q}\sum\limits_{j\in I}k_{t-l,j}^{\mathbf{c}^{t}}\sigma_j-\sum\limits_{t=l+2}^{2l+1}\sum\limits_{j\in I}k_{t-l,j}^{\mathbf{c}^{t}}\sigma_j\\
				&=2\sum\limits_{q=0}^{i-1}\sum\limits_{t=l+2}^{2l+1-q}\mu_{t-l}-(i-1)(-\sigma_1+\sigma_2)+(\sigma_{l+1}-\sigma_{l+2-i})
				-(-\sigma_1+\sigma_2+\sigma_{l+1}-\sigma_{l+2})
			\end{aligned}
		\end{equation*}
		and
		\begin{equation*}
			\begin{aligned}
				III&=2\sum\limits_{q=0}^{i-1}\sum\limits_{t=1}^{q}(1+{\alpha}_{l+2-t}^{*})=2\sum\limits_{q=0}^{i-1}\sum\limits_{t=1}^{q}\mu_{l+2-t}
				-\sum\limits_{q=0}^{i-1}\sum\limits_{t=1}^{q}\sum\limits_{j\in I}k_{l+2-t,j}^{\mathbf{c}^{t}}\sigma_{j}\\
				&=2\sum\limits_{q=0}^{i-1}\sum\limits_{t=1}^{q}\mu_{l+2-t}-\sum\limits_{q=1}^{i-1}(-\sigma_{l+1-q}+\sigma_{l+2-q}+\sigma_{l+1}-\sigma_{l+2})\\
				&=2\sum\limits_{q=0}^{i-1}\sum\limits_{t=1}^{q}\mu_{l+2-t}-(\sigma_{l+1}-\sigma_{l+2-i})-(i-1)(\sigma_{l+1}-\sigma_{l+2}).
			\end{aligned}
		\end{equation*}
		Substituting the above results of $I$, $II$ and $III$ into \eqref{Affine-Toda-Section-4-Eq-17}, we conclude that
		\begin{equation*}
			\begin{aligned}
				\sigma_{\tilde{v}_i}=2\sum\limits_{q=0}^{i-1}\sum\limits_{t=1}^{l+1}\mu_{l+2-t}+2\sum\limits_{q=0}^{i-1}\sum\limits_{t=l+2}^{2l+1-q}\mu_{t-l}
				-2\sum\limits_{q=0}^{i-1}\sum\limits_{t=1}^{q}\mu_{l+2-t}
				+2\sigma_{l+2}-2\sigma_{l+2-i}.
			\end{aligned}
		\end{equation*}
		Since $\sigma_{v_{l+2-i}}=\sigma_{\tilde{v}_i}$ for $i\in J_0$, we obtain that for $s\in J_0$,
		\begin{equation*}
			\begin{aligned}
				\sigma_{v_{s}}=\frac{1}{2\pi}\int_{\mathbb{R}^2}e^{v_{s}(y)}\mathrm{d}y
				&=2\sum\limits_{q=0}^{l+1-s}\sum\limits_{t=1}^{l+1}\mu_{l+2-t}+2\sum\limits_{q=0}^{l+1-s}\sum\limits_{t=l+2}^{2l+1-q}\mu_{t-l}
				-2\sum\limits_{q=0}^{l+1-s}\sum\limits_{t=1}^{q}\mu_{l+2-t}+2\sigma_{l+2}-2\sigma_{s}.
			\end{aligned}
		\end{equation*}
		Hence it holds that for $s\in J_0$,
		\begin{equation*}
			\sigma_{s}^{*}=\sigma_{s}+\sigma_{v_s}=2(l+2-s)\sum\limits_{t=1}^{l+1}\mu_{t}+2\sum\limits_{q=0}^{l+1-s}\sum\limits_{t=l+2}^{2l+1-q}\mu_{t-l}
			-2\sum\limits_{q=0}^{l+1-s}\sum\limits_{t=1}^{q}\mu_{l+2-t}-\sigma_{s}+2\sigma_{l+2}.
		\end{equation*}
		
		\noindent\textbf{Step 2.} We calculate $\mathfrak{R}_{J_0}\bm{\sigma}=(\mathfrak{R}_{l+1}\mathfrak{R}_{l}\cdots\mathfrak{R}_{1})^{l+1}\bm{\sigma}$. By direct computation we get that
		\begin{equation*}
			\begin{aligned}
				(\mathfrak{R}_{l+1}\mathfrak{R}_{l}\cdots\mathfrak{R}_{1}\bm{\sigma})_{s}=\left\{
				\begin{aligned}
					&2\sum\limits_{t=1}^{s}\mu_{t}-\sigma_1+\sigma_2+\sigma_{s+1}, \ && \text{for} \ 1\leq s\leq l+1,\\
					&\sigma_{s}, \ && \text{for} \ l+2\leq s\leq n+1.
				\end{aligned}
				\right.
			\end{aligned}
		\end{equation*}
		Inductively, one can deduce that for any $2\leq r\leq l+1$,
		\begin{equation*}
			\begin{aligned}
				&\left((\mathfrak{R}_{l+1}\mathfrak{R}_{l}\cdots\mathfrak{R}_{1})^r\bm{\sigma}\right)_{s}=\\
				&\left\{
				\begin{aligned}
					&2\sum\limits_{q=s}^{s+r-1}\sum\limits_{t=1}^{q}\mu_{t}+2\sum\limits_{w=1}^{r-1}\sum\limits_{q=1}^{r-w}
					\Big(\sum\limits_{t=1}^{q+1}\mu_t-\sum\limits_{t=1}^{q}\mu_t\Big)
					-\sigma_1+\sigma_{r+1}+\sigma_{s+r}, \ && \text{for} \ 1\leq s\leq l-(r-2),\\
					&2\sum\limits_{q=s}^{l+1}\sum\limits_{t=1}^{q}\mu_{t}+
					2\sum\limits_{w=1}^{r-i}\sum\limits_{q=1}^{r-w}
					\Big(\sum\limits_{t=1}^{q+1}\mu_t-\sum\limits_{t=1}^{q}\mu_t\Big)
					-\sigma_{i+1}+\sigma_{r+1}+\sigma_{s+r-i}, \ && \text{for} \ s=l+2-r+i, \ 1\leq i\leq r-1, \\
					&\sigma_{s}, \ && \text{for} \ l+2\leq s\leq n+1.
				\end{aligned}
				\right.
			\end{aligned}
		\end{equation*}
		
		\noindent\textbf{Step 3.} Finally, we compare the results in the above two steps. For $r=l+1$ and $s=1$, we get that
		\begin{equation*}
			\begin{aligned}
				\left((\mathfrak{R}_{l+1}\mathfrak{R}_{l}\cdots\mathfrak{R}_{1})^{l+1}\bm{\sigma}\right)_{1}&=
				2\sum\limits_{q=1}^{l+1}\sum\limits_{t=1}^{q}\mu_{t}+2\sum\limits_{w=1}^{l}\sum\limits_{q=1}^{l+1-w}\mu_{q+1}
				-\sigma_1+2\sigma_{l+2}\\
				&=2(l+1)\sum\limits_{t=1}^{l+1}\mu_{t}-2\sum\limits_{q=1}^{l+1}\sum\limits_{t=q+1}^{l+1}\mu_{t}+2\sum\limits_{w=1}^{l}\sum\limits_{q=1}^{l+1-w}\mu_{q+1}
				-\sigma_1+2\sigma_{l+2}\\
				&=2(l+1)\sum\limits_{t=1}^{l+1}\mu_{t}-2\sum\limits_{q=0}^{l}\sum\limits_{t=1}^{q}\mu_{l+2-t}+2\sum\limits_{q=0}^{l}\sum\limits_{t=l+2}^{2l+1-q}\mu_{t-l}
				-\sigma_{1}+2\sigma_{l+2}\\
				&=\sigma_{1}^{*}.
			\end{aligned}
		\end{equation*}
		For $r=l+1$ and $s=i+1$ ($1\leq i\leq l$), we have $l+2-r+i=i+1$. Then it holds that
		\begin{equation*}
			\begin{aligned}
				\left((\mathfrak{R}_{l+1}\mathfrak{R}_{l}\cdots\mathfrak{R}_{1})^{l+1}\bm{\sigma}\right)_{i+1}&=
				2\sum\limits_{q=s}^{l+1}\sum\limits_{t=1}^{q}\mu_{t}+
				2\sum\limits_{w=1}^{l+1-i}\sum\limits_{q=1}^{l+1-w}\mu_{q+1}
				-\sigma_{i+1}+2\sigma_{l+2}\\
				&=\left(2(l+1-i)\sum\limits_{t=1}^{l+1}\mu_{t}-2\sum\limits_{q=0}^{l-i}\sum\limits_{t=1}^{q}\mu_{l+2-t}\right)
				+2\sum\limits_{q=0}^{l-i}\sum\limits_{t=l+2}^{2l+1-q}\mu_{t-l}
				-\sigma_{i+1}+2\sigma_{l+2}\\
				&=\sigma_{i+1}^{*}.
			\end{aligned}
		\end{equation*}
		This implies that $(\mathfrak{R}_{l+1}\mathfrak{R}_{l}\cdots\mathfrak{R}_{1})^{l+1}\bm{\sigma}=\bm{\sigma}^{*}$.
		\medskip
		
		\noindent{\em Proof of (b)}. As the proof of (a), we also split the proof into several steps.
		
		\noindent\textbf{Step 1.} We calculate the mass of each $v_s$, $s\in J_0$. Denote by $\tilde{\mathbf{v}}=(\tilde{v}_i,\tilde{v}_{i+1},\cdots,\tilde{v}_{n+1})$, where $\tilde{v}_s=v_{s+i-1}$ and set $\tilde{\alpha}_{s}^{*}=\alpha_{s+i-1}^{*}$ for $s=1,\cdots,l+1$. Then $\tilde{\mathbf{v}}$ satisfies the $\mathbf{C}_{l+1}$ Toda system \eqref{Affine-Toda-Section-4-Eq-AC} with $\alpha_{s}=\tilde{\alpha}_s^{*}$ and $(k_{st}^{\prime})=(c_{st})$, $s,t=1,\cdots,l+1$, where $(c_{st})_{(l+1)\times(l+1)}$ stands for the $\mathbf{C}$ type Cartan matrix of rank $l+1$. Let $u_s=u_{2l+2-s}=\tilde{v}_s$ and $\beta_{s}=\beta_{2l+2-s}=\tilde{\alpha}_s^{*}$ for $s=1,\cdots,l+1$, then $\mathbf{u}=(u_1,u_2,\cdots,u_{2l+1})$ satisfies the $\mathbf{A}_{2l+1}$ Toda system
		\begin{equation*}
			\Delta u_s+\sum\limits_{t=1}^{2l+1}a_{st}e^{u_t}=4\pi\beta_s\delta_0 \ \text{in} \ \mathbb{R}^2, \ \ \int_{\mathbb{R}^2}e^{u_s(y)}\mathrm{d}y<+\infty, \ \ \forall \ s=1,2,\cdots,2l+1,
		\end{equation*}
		where $(a_{st})_{(2l+1)\times(2l+1)}$ stands for the $\mathbf{A}$ type Cartan matrix of rank $2l+1$. For $s=1,\cdots,l+1$, we have
		\begin{equation}\label{Affine-Toda-Section-4-Eq-18}
			\begin{aligned}
				\sigma_{\tilde{v}_s}&=\sigma_{u_s}=\frac{1}{2\pi}\int_{\mathbb{R}^2}e^{u_s(y)}\mathrm{d}y
				=2\sum\limits_{q=0}^{s-1}
				\Big(\sum\limits_{t=1}^{2l+1-q}(1+\beta_t)-\sum\limits_{t=1}^{q}(1+\beta_t)\Big)\\
				&=2\sum\limits_{q=0}^{s-1}\sum\limits_{t=1}^{l+1}(1+\beta_t)+2\sum\limits_{q=0}^{s-1}\sum\limits_{t=l+2}^{2l+1-q}(1+\beta_t)
				-2\sum\limits_{q=0}^{s-1}\sum\limits_{t=1}^{q}(1+\beta_t)\\
				&=2\sum\limits_{q=0}^{s-1}\sum\limits_{t=1}^{l+1}(1+\tilde{\alpha}_t^{*})
				+2\sum\limits_{q=0}^{s-1}\sum\limits_{t=l+2}^{2l+1-q}(1+\tilde{\alpha}_{2l+2-t}^{*})
				-2\sum\limits_{q=0}^{s-1}\sum\limits_{t=1}^{q}(1+\tilde{\alpha}_t^{*})\\
				&=2\sum\limits_{q=0}^{s-1}\sum\limits_{t=1}^{l+1}(1+{\alpha}_{t+i-1}^{*})
				+2\sum\limits_{q=0}^{s-1}\sum\limits_{t=l+2}^{2l+1-q}(1+{\alpha}_{2l+i+1-t}^{*})
				-2\sum\limits_{q=0}^{s-1}\sum\limits_{t=1}^{q}(1+{\alpha}_{t+i-1}^{*})\\
				&=:I+II-III.
			\end{aligned}
		\end{equation}
		In a similar way as the discussion in the first part, we calculate $I$, $II$ and $III$ term by term directly
		\begin{equation*}
			\left\{
			\begin{aligned}
				&I=2\sum\limits_{q=0}^{s-1}\sum\limits_{t=1}^{l+1}\mu_{t+i-1}-s(-\sigma_{i-1}+\sigma_{i}-\sigma_{n}+\sigma_{n+1}),\\
				&II=2\sum\limits_{q=0}^{s-1}\sum\limits_{t=l+2}^{2l+1-q}\mu_{2l+i+1-t}-s(\sigma_n-\sigma_{n+1})-(-\sigma_{i-1}+\sigma_{s+i-1}),\\
				&III=2\sum\limits_{q=0}^{s-1}\sum\limits_{t=1}^{q}\mu_{t+i-1}-s(-\sigma_{i-1}+\sigma_{i})-(\sigma_{i-1}-\sigma_{s+i-1}).
			\end{aligned}
			\right.
		\end{equation*}
		Substituting the above results of $I$, $II$ and $III$ into \eqref{Affine-Toda-Section-4-Eq-18}, we conclude that
		\begin{equation*}
			\begin{aligned}
				\sigma_{\tilde{v}_s}=
				2\sum\limits_{q=0}^{s-1}\sum\limits_{t=1}^{l+1}\mu_{t+i-1}
				+2\sum\limits_{q=0}^{s-1}\sum\limits_{t=l+2}^{2l+1-q}\mu_{2l+i+1-t}
				-2\sum\limits_{q=0}^{s-1}\sum\limits_{t=1}^{q}\mu_{t+i-1}
				+2\sigma_{i-1}-2\sigma_{s+i-1}.
			\end{aligned}
		\end{equation*}
		Since $\sigma_{v_{s+i-1}}=\sigma_{\tilde{v}_s}$ for $s=1,\cdots,l+1$, we deduce that for $s\in J_0$,
		\begin{equation*}
			\begin{aligned}
				\sigma_{v_{s}}=
				2\sum\limits_{q=0}^{s-i}\sum\limits_{t=1}^{l+1}\mu_{t+i-1}
				+2\sum\limits_{q=0}^{s-i}\sum\limits_{t=l+2}^{2l+1-q}\mu_{2l+i+1-t}
				-2\sum\limits_{q=0}^{s-i}\sum\limits_{t=1}^{q}\mu_{t+i-1}
				+2\sigma_{i-1}-2\sigma_{s}.
			\end{aligned}
		\end{equation*}
		Hence it holds that for $s\in J_0$,
		\begin{equation*}
			\sigma_{s}^{*}=\sigma_{s}+\sigma_{v_s}=2\sum\limits_{q=0}^{s-i}\sum\limits_{t=1}^{l+1}\mu_{t+i-1}
			+2\sum\limits_{q=0}^{s-i}\sum\limits_{t=l+2}^{2l+1-q}\mu_{2l+i+1-t}
			-2\sum\limits_{q=0}^{s-i}\sum\limits_{t=1}^{q}\mu_{t+i-1}-\sigma_{s}+2\sigma_{i-1}.
		\end{equation*}
		
		\noindent\textbf{Step 2.} We calculate $\mathfrak{R}_{J_0}\bm{\sigma}=(\mathfrak{R}_{i}\cdots\mathfrak{R}_{n}\mathfrak{R}_{n+1})^{l+1}\bm{\sigma}$. By direct computation we get that
		\begin{equation*}
			\begin{aligned}
				(\mathfrak{R}_{i}\cdots\mathfrak{R}_{n}\mathfrak{R}_{n+1}\bm{\sigma})_{s}=\left\{
				\begin{aligned}
					&\sigma_{s}, \ && \text{for} \ 1\leq s\leq i-1,\\
					&2\sum\limits_{t=s}^{n+1}\mu_{t}+\sigma_{s-1}+\sigma_n-\sigma_{n+1}, \ && \text{for} \ i\leq s\leq n+1,
				\end{aligned}
				\right.
			\end{aligned}
		\end{equation*}
		and
		\begin{equation*}
			\begin{aligned}
				&\left((\mathfrak{R}_{i}\cdots\mathfrak{R}_{n}\mathfrak{R}_{n+1})^2\bm{\sigma}\right)_{s}\\
				&=\left\{
				\begin{aligned}
					&\sigma_{s}, \ && \text{for} \ 1\leq s\leq i-1,\\
					&2\sum\limits_{t=s}^{n+1}\mu_{t}+2\mu_{n}+\sigma_{s-1}+\sigma_{n-1}-\sigma_{n}, \ && \text{for} \ s=i,\\
					&2\sum\limits_{q=s-1}^{s}\sum\limits_{t=q}^{n+1}\mu_{t}+2\mu_{n}
					+\sigma_{s-2}+\sigma_{n-1}-\sigma_{n+1}, \ && \text{for} \ i+1\leq s\leq n+1.
				\end{aligned}
				\right.
			\end{aligned}
		\end{equation*}
		We decompose $s=i+j$ if $i\leq s\leq n+1$, where $0\leq j\leq l$. Inductively, one can deduce that, for any $3\leq r\leq l+1$,
		
		\noindent (1) if $0\leq s-i\leq r-3$, then
		\begin{equation*}
			\begin{aligned}
				&\left((\mathfrak{R}_{i}\cdots\mathfrak{R}_{n}\mathfrak{R}_{n+1})^r\bm{\sigma}\right)_{s}\\
				&=2\sum\limits_{q=s-j}^{s}\sum\limits_{t=q}^{n+1}\mu_{t}+2\sum\limits_{w=n+2-r}^{n+2-r+j}\sum\limits_{q=w}^{n}\sum\limits_{t=q}^{n+1}\mu_{t}
				-2\sum\limits_{w=n+3-r}^{n+3-r+j}\sum\limits_{q=w}^{n+1}\sum\limits_{t=q}^{n+1}\mu_{t}+\sigma_{s-j-1}+\sigma_{n+1-r}-\sigma_{n+2-r+j};
			\end{aligned}
		\end{equation*}
		\noindent (2) if $j=r-2$, i.e., $s=i+r-2$, then
		\begin{equation*}
			\begin{aligned}
				&\left((\mathfrak{R}_{i}\cdots\mathfrak{R}_{n}\mathfrak{R}_{n+1})^r\bm{\sigma}\right)_{s}\\
				&=2\sum\limits_{q=s-r+2}^{s}\sum\limits_{t=q}^{n+1}\mu_{t}+2\sum\limits_{w=n+2-r}^{n-1}\sum\limits_{q=w}^{n}\sum\limits_{t=q}^{n+1}\mu_{t}
				-2\sum\limits_{w=n+3-r}^{n}\sum\limits_{q=w}^{n+1}\sum\limits_{t=q}^{n+1}\mu_{t}+2\mu_{n}+\sigma_{s-r+1}+\sigma_{n+1-r}-\sigma_{n+2-r+j};
			\end{aligned}
		\end{equation*}
		\noindent (3) if $r-1\leq j\leq l$, then
		\begin{equation*}
			\begin{aligned}
				&\left((\mathfrak{R}_{i}\cdots\mathfrak{R}_{n}\mathfrak{R}_{n+1})^r\bm{\sigma}\right)_{s}\\
				&=2\sum\limits_{q=s-r+1}^{s}\sum\limits_{t=q}^{n+1}\mu_{t}+2\sum\limits_{w=n+2-r}^{n-1}\sum\limits_{q=w}^{n}\sum\limits_{t=q}^{n+1}\mu_{t}
				-2\sum\limits_{w=n+3-r}^{n}\sum\limits_{q=w}^{n+1}\sum\limits_{t=q}^{n+1}\mu_{t}+2\mu_{n}+\sigma_{s-r}+\sigma_{n+1-r}-\sigma_{n+1};
			\end{aligned}
		\end{equation*}
		\noindent (4) if $1\leq s\leq i-1$, then $\left((\mathfrak{R}_{i}\cdots\mathfrak{R}_{n}\mathfrak{R}_{n+1})^r\bm{\sigma}\right)_{s}=\sigma_{s}$.
		\medskip
		
		\noindent\textbf{Step 3.} Finally, we compare the results in the above two steps. For $r=l+1$, we decompose $s=i+j$ if $i\leq s\leq n+1$, where $0\leq j\leq l$. On one hand, by \textbf{Step 1} we have
		\begin{equation*}
			\begin{aligned}
				\sigma_{s}^{*}&=2\sum\limits_{q=0}^{s-i}\sum\limits_{t=i}^{n+1}\mu_{t}
				+2\sum\limits_{q=0}^{s-i}\sum\limits_{t=q+i}^{n}\mu_{t}
				-2\sum\limits_{q=0}^{s-i}\sum\limits_{t=i}^{q+i-1}\mu_{t}-\sigma_{s}+2\sigma_{i-1}\\
				&=4\sum\limits_{q=0}^{s-i}\sum\limits_{t=q+i}^{n+1}\mu_{t}-2(j+1)\mu_{n+1}-\sigma_{s}+2\sigma_{i-1}\\
				&=4\sum\limits_{q=i}^{s}\sum\limits_{t=q}^{n+1}\mu_{t}-2(j+1)\mu_{n+1}-\sigma_{s}+2\sigma_{i-1}.
			\end{aligned}
		\end{equation*}
		On the other hand, from the results in {\bf Step 2} we have:
		
		\noindent (1) if $0\leq j\leq r-3=l-2$,
		\begin{equation*}
			\begin{aligned}
				&\left((\mathfrak{R}_{i}\cdots\mathfrak{R}_{n}\mathfrak{R}_{n+1})^r\bm{\sigma}\right)_{s}\\
				&=2\sum\limits_{q=i}^{s}\sum\limits_{t=q}^{n+1}\mu_{t}+2\sum\limits_{w=i}^{s}\sum\limits_{q=w}^{n}\sum\limits_{t=q}^{n+1}\mu_{t}
				-2\sum\limits_{w=i+1}^{s+1}\sum\limits_{q=w}^{n+1}\sum\limits_{t=q}^{n+1}\mu_{t}-\sigma_{s}+2\sigma_{i-1}\\
				&=2\sum\limits_{q=i}^{s}\sum\limits_{t=q}^{n+1}\mu_{t}+2\left(\sum\limits_{w=i+1}^{s}\sum\limits_{q=w}^{n}\sum\limits_{t=q}^{n+1}\mu_{t}+
				\sum\limits_{q=i}^{n}\sum\limits_{t=q}^{n+1}\mu_{t}\right)
				-2\left(\sum\limits_{w=i+1}^{s}\sum\limits_{q=w}^{n}\sum\limits_{t=q}^{n+1}\mu_{t}+\sum\limits_{q=s+1}^{n+1}\sum\limits_{t=q}^{n+1}\mu_{t}+j\mu_{n+1}\right)\\
				& \ \ \ \ -\sigma_{s}+2\sigma_{i-1}\\
				&=2\sum\limits_{q=i}^{s}\sum\limits_{t=q}^{n+1}\mu_{t}+2\sum\limits_{q=i}^{n}\sum\limits_{t=q}^{n+1}\mu_{t}-2\sum\limits_{q=s+1}^{n+1}\sum\limits_{t=q}^{n+1}\mu_{t}
				-2j\mu_{n+1}-\sigma_{s}+2\sigma_{i-1}\\
				&=4\sum\limits_{q=i}^{s}\sum\limits_{t=q}^{n+1}\mu_{t}-2(j+1)\mu_{n+1}-\sigma_{s}+2\sigma_{i-1};
			\end{aligned}
		\end{equation*}
		
		\noindent (2) if $j=r-2=l-1$, i.e., $s=i+r-2=n$,
		\begin{equation*}
			\begin{aligned}
				\left((\mathfrak{R}_{i}\cdots\mathfrak{R}_{n}\mathfrak{R}_{n+1})^{l+1}\bm{\sigma}\right)_{n}
				&=2\sum\limits_{q=i}^{n}\sum\limits_{t=q}^{n+1}\mu_{t}+2\sum\limits_{w=i}^{n-1}\sum\limits_{q=w}^{n}\sum\limits_{t=q}^{n+1}\mu_{t}
				-2\sum\limits_{w=i+1}^{n}\sum\limits_{q=w}^{n+1}\sum\limits_{t=q}^{n+1}\mu_{t}+2\mu_{n}-\sigma_{n}+2\sigma_{i-1}\\
				&=4\sum\limits_{q=i}^{n}\sum\limits_{t=q}^{n+1}\mu_{t}-2(j+1)\mu_{n+1}-\sigma_{n}+2\sigma_{i-1};
			\end{aligned}
		\end{equation*}
		
		\noindent (3) if $l=r-1\leq j\leq l$, i.e., $s=n+1$,
		\begin{equation*}
			\begin{aligned}
				\left((\mathfrak{R}_{i}\cdots\mathfrak{R}_{n}\mathfrak{R}_{n+1})^{l+1}\bm{\sigma}\right)_{n+1}
				&=2\sum\limits_{q=i}^{n+1}\sum\limits_{t=q}^{n+1}\mu_{t}+2\sum\limits_{w=i}^{n-1}\sum\limits_{q=w}^{n}\sum\limits_{t=q}^{n+1}\mu_{t}
				-2\sum\limits_{w=i+1}^{n}\sum\limits_{q=w}^{n+1}\sum\limits_{t=q}^{n+1}\mu_{t}+2\mu_{n}-\sigma_{n+1}+2\sigma_{i-1}\\
				&=4\sum\limits_{q=i}^{n+1}\sum\limits_{t=q}^{n+1}\mu_{t}-2(j+1)\mu_{n+1}-\sigma_{n+1}+2\sigma_{i-1}.
			\end{aligned}
		\end{equation*}
		Therefore, we get $\left((\mathfrak{R}_{i}\cdots\mathfrak{R}_{n}\mathfrak{R}_{n+1})^{l+1}\bm{\sigma}\right)_{s}=\sigma_{s}^{*}$ for $s\in I$. This finishes the proof of Lemma \ref{Affine-Toda-Section-4-Lemma-6}.
	\end{proof}
	
	As a consequence of Lemma \ref{Affine-Toda-Section-4-Lemma-6}, we derive the following theorem which will be repeatedly used in the blow-up procedure.
	
	\begin{theorem}\label{Affine-Toda-Section-4-Theorem-10}
		Suppose that $\bm{\sigma}=(\sigma_1,\cdots,\sigma_{n+1})\in\Gamma_{\mathbf{C}^t}(\bm{\mu})$. Denote by $\alpha_{i}^{*}=\alpha_{i}-\frac{1}{2}\sum_{t\in I}k_{it}^{\mathbf{c}^t}\sigma_{t}$ and $\mu_{i}^{*}=1+\alpha_{i}^{*}$ for $i\in I$. One of the four possibilities stated in Proposition \ref{Affine-Toda-Section-4-Proposition-1}-(4) is fulfilled by the set of indices $I$:
		\begin{enumerate}[(a)]
			\item Assume that \textbf{(I)} holds and satisfies (I-1). Denote by $\mathbf{v}=(v_1,\cdots,v_{n+1})$, where each $v_s$, $s\in J_0$ satisfies the $\mathbf{C}_{l_0+1}$-type Toda system \eqref{Affine-Toda-Section-4-Eq-AC} with $\alpha_{s}=\alpha_{s}^{*}$ and $(k_{st}^{\prime})=(k_{st}^{\mathbf{c}^t})$, $s,t\in J_{0}$; each $v_s$, $s\in J_p$, $p=1,\cdots,\vartheta$ satisfies the $\mathbf{A}_{l_p+1}$ Toda system \eqref{Affine-Toda-Section-4-Eq-AC} with $\alpha_{s}=\alpha_{s}^{*}$ and $(k_{st}^{\prime})=(k_{st}^{\mathbf{c}^t})$, $s,t\in J_{p}$; the left elements can be treated as $-\infty$.
			
			\item Assume that \textbf{(II)} holds and satisfies (II-1). Denote by $\mathbf{v}=(v_1,\cdots,v_{n+1})$, where each $v_s$, $s\in J_0$ satisfies the $\mathbf{C}_{l_0+1}$ Toda system \eqref{Affine-Toda-Section-4-Eq-AC} with $\alpha_{s}=\alpha_{s}^{*}$ and $(k_{st}^{\prime})=(k_{st}^{\mathbf{c}^t})$, $s,t\in J_{0}$; each $v_s$, $s\in J_p$, $p=1,\cdots,\vartheta$ satisfies the $\mathbf{A}_{l_p+1}$ Toda system \eqref{Affine-Toda-Section-4-Eq-AC} with $\alpha_{s}=\alpha_{s}^{*}$ and $(k_{st}^{\prime})=(k_{st}^{\mathbf{c}^t})$, $s,t\in J_{p}$; the left elements can be treated as $-\infty$.
			
			\item Assume that \textbf{(III)} holds and satisfies (III-1). Denote by $\mathbf{v}=(v_1,\cdots,v_{n+1})$, where each $v_s$, $s\in J_0$ satisfies the $\mathbf{C}_{l_0+1}$-type Toda system \eqref{Affine-Toda-Section-4-Eq-AC} with $\alpha_{s}=\alpha_{s}^{*}$ and $(k_{st}^{\prime})=(k_{st}^{\mathbf{c}^t})$, $s,t\in J_{0}$; each $v_s$, $s\in J_{\vartheta+1}$ satisfies the $\mathbf{C}_{l_{\vartheta+1}+1}$ Toda system \eqref{Affine-Toda-Section-4-Eq-AC} with $\alpha_{s}=\alpha_{s}^{*}$ and $(k_{st}^{\prime})=(k_{st}^{\mathbf{c}^t})$, $s,t\in J_{\vartheta+1}$; each $v_s$, $s\in J_p$, $p=1,\cdots,\vartheta$ satisfies the $\mathbf{A}_{l_p+1}$ Toda system \eqref{Affine-Toda-Section-4-Eq-AC} with $\alpha_{s}=\alpha_{s}^{*}$ and $(k_{st}^{\prime})=(k_{st}^{\mathbf{c}^t})$, $s,t\in J_{p}$; the left elements can be treated as $-\infty$.
			
			\item Assume that \textbf{(IV)} holds and satisfies (IV-1). Denote by $\mathbf{v}=(v_1,\cdots,v_{n+1})$, where each $v_s$, $s\in J_p$, $p=1,\cdots,\vartheta$ satisfies the $\mathbf{A}_{l_p+1}$ Toda system \eqref{Affine-Toda-Section-4-Eq-AC} with $\alpha_{s}=\alpha_{s}^{*}$ and $(k_{st}^{\prime})=(k_{st}^{\mathbf{c}^t})$, $s,t\in J_{p}$; the left elements can be treated as $-\infty$.
		\end{enumerate}
		In any case, we set $\sigma_{s}^{*}=\sigma_{s}$ for $s\in I\setminus J$ and $\sigma_{s}^{*}=\sigma_{s}+\frac{1}{2\pi}\int_{\mathbb{R}^2}e^{v_s(y)}\mathrm{d}y$ for $s\in J$. Then it holds that
		\begin{equation*}
			\begin{aligned}
				&\text{(a):} \ \ \bm{\sigma}^{*}=(\sigma^{*}_1,\cdots,\sigma^{*}_{n+1})
				=\left(\mathfrak{R}_{J_0}\mathfrak{R}_{J_1}\cdots\mathfrak{R}_{J_\vartheta}\right)\bm{\sigma}\in\Gamma_{\mathbf{C}^{t}}(\bm{\mu}),\\
				&\text{(b):} \ \ \bm{\sigma}^{*}=(\sigma^{*}_1,\cdots,\sigma^{*}_{n+1})
				=\left(\mathfrak{R}_{J_0}\mathfrak{R}_{J_1}\cdots\mathfrak{R}_{J_\vartheta}\right)\bm{\sigma}\in\Gamma_{\mathbf{C}^{t}}(\bm{\mu}),\\
				&\text{(c):} \ \ \bm{\sigma}^{*}=(\sigma^{*}_1,\cdots,\sigma^{*}_{n+1})
				=\left(\mathfrak{R}_{J_0}\mathfrak{R}_{J_1}\cdots\mathfrak{R}_{J_\vartheta}\mathfrak{R}_{J_{\vartheta+1}}\right)\bm{\sigma}
				\in\Gamma_{\mathbf{C}^{t}}(\bm{\mu}),\\
				&\text{(d):} \ \ \bm{\sigma}^{*}=(\sigma^{*}_1,\cdots,\sigma^{*}_{n+1})
				=\left(\mathfrak{R}_{J_1}\cdots\mathfrak{R}_{J_\vartheta}\right)\bm{\sigma}\in\Gamma_{\mathbf{C}^{t}}(\bm{\mu}).
			\end{aligned}
		\end{equation*}
	\end{theorem}
	\begin{proof}
		By Lemma \ref{Affine-Toda-Section-2-Lemma-10} and Lemma \ref{Affine-Toda-Section-4-Lemma-6}, one can directly obtain these conclusions.
	\end{proof}
	
	\subsection{Blow-up analysis for affine $\mathbf{C}^t$ type Toda system}
	In this subsection, we prove Theorem \ref{Affine-Toda-Section-1-Theorem-2} in a similar way as Theorem \ref{Affine-Toda-Section-1-Theorem-1} in Section \ref{Affine-Toda-Section-3}. Thus, we just outline main steps. Prior to that, we would like to present some algebraic computations and review certain findings in \cite{Lin-Yang-Zhong-2020}.
	
	For any integer $l\geq1$, let $\mathbb{S}_{\mathbf{C}_{l+1}}$ be a subgroup of the permutation group for $\{0,1,\cdots,2l+1\}$ such that each $f\in\mathbb{S}_{\mathbf{C}_{l+1}}$ satisfies
	\begin{equation}\label{Affine-Toda-Section-4-Eq-14}
		f(j)+f(2l+1-j)=2l+1, \ 0\leq j\leq 2l+1.
	\end{equation}
	By \cite[Proposition 8.5]{Lin-Yang-Zhong-2020}, all simple permutations satisfying \eqref{Affine-Toda-Section-4-Eq-14} (i.e., in $\mathbb{S}_{\mathbf{C}_{l+1}}$) are given as follows
	\begin{equation}\label{5.simple}
		f_i(j)=
		\left\{
		\begin{aligned}
			&j+1, \ && \text{if} \ j=i,2l-i,\\
			&j-1, \ && \text{if} \ j=i+1,2l+1-i,\\
			&j, \ && \text{if} \ j\neq i,i+1,2l-i,2l+1-i,
		\end{aligned}
		\right.
		\ \ \ \text{for} \ 0\leq i\leq l.
	\end{equation}
	In other words, each $f\in\mathbb{S}_{\mathbf{C}_{l+1}}$ can be expressed as a composition of these simple permutations. Let $J_0=\{i_0,i_0+1,\cdots,i_0+l_0\}\subsetneqq I$ for some $i_0\in I$ and $l_0\in\mathbb{N}$. In the sequel, we consider the following singular Toda system
	\begin{equation}\label{Affine-Toda-Section-4-Eq-15}
		\Delta u_i+\sum\limits_{j\in J_0}k_{ij}^{\prime}e^{u_j}=4\pi\sum\limits_{\ell=1}^{N}\alpha_{\ell,i}\delta_{p_{\ell}} \ \text{in} \ \mathbb{R}^2, \ \int_{\mathbb{R}^2}e^{u_i(y)}\mathrm{d}y <+\infty, \ \ \forall \ i\in J_0,
	\end{equation}
	where $(k_{ij}^{\prime})_{|J_0|\times|J_0|}$ is the  $\mathbf{C}$ type Cartan matrix, $p_1,\cdots,p_{N}$ are distinct points in $\mathbb{R}^2$ and $\alpha_{\ell,i}>-1$, $1\leq\ell\leq N$, $i\in J_0$. In particular, $\alpha_{\ell,i}$ are positive integers except $\alpha_{1,i}$, $i\in J_0$. Let $\mathbf{u}=(u_{i_0},\cdots,u_{i_0+l_0})$ be the solution of \eqref{Affine-Toda-Section-4-Eq-15}, we define the mass of each $u_i$, denoted by $\sigma_{u_i}$, as
	\begin{equation}\label{Affine-Toda-Section-4-Eq-16}
		\sigma_{u_i}=\frac{1}{2\pi}\int_{\mathbb{R}^2}e^{u_i(y)}\mathrm{d}y, \ \ \text{for} \ i\in J_0.
	\end{equation}
	Then, applying the group $\mathbb{S}_{\mathbf{C}_{l_0+1}}$, we are able to characterize the quantized result for some $\mathbf{C}_{l_0+1}$ Toda system as below.
	
	\begin{lemma}$($\cite[Theorem 8.1]{Lin-Yang-Zhong-2020}$)$\label{Affine-Toda-Section-4-Lemma-11}
		Let $\mathbf{u}=(u_{i_0},\cdots,u_{i_0+l_0})$ be the solution of \eqref{Affine-Toda-Section-4-Eq-15} and $\sigma_{u_i}$ be defined as in \eqref{Affine-Toda-Section-4-Eq-16}. If $i_0=1$ and $(k_{ij}^{\prime})_{|J_0|\times|J_0|}$ stands for the Cartan matrix associated with simple Lie algebra $\mathbf{C}_{l_0+1}$, then there exists a permutation map $f\in\mathbb{S}_{\mathbf{C}_{l_0+1}}$ such that
		\begin{equation*}
			\sigma_{u_i}=2\sum\limits_{j=0}^{i-1}\Big(\sum\limits_{r=1}^{f(j)}\bar\alpha_{1,r}-\sum\limits_{r=1}^{j}\bar\alpha_{1,r}\Big)+2\bar{N}_i, \ \text{for some} \ \bar{N}_i\in\mathbb{Z}, \ 1\leq i\leq l_0+1,
		\end{equation*}
		where $\bar\alpha_{1,r}=\bar\alpha_{1,2l_0+2-r}=\alpha_{1,r}$, $1\leq r\leq l_0+1$.
	\end{lemma}
	
	As a consequence of Lemma \ref{Affine-Toda-Section-4-Lemma-11}, we deduce the following lemma for later use.
	
	\begin{lemma}\label{Affine-Toda-Section-4-Lemma-12}
		Let $\mathbf{u}=(u_{i_0},\cdots,u_{i_0+l_0})$ be the solution of \eqref{Affine-Toda-Section-4-Eq-15} and $\sigma_{u_i}$ be defined as in \eqref{Affine-Toda-Section-4-Eq-16}. Suppose that $(k_{ij}^{\prime})=(k_{ij}^{\mathbf{c}^{t}})$, $i,j\in J_0$.
		\begin{enumerate}[(a)]
			\item If $i_0=1$, then there exists a permutation map $f\in\mathbb{S}_{\mathbf{C}_{l_0+1}}$ such that
			\begin{equation*}
				\sigma_{u_{i}}=2\sum\limits_{j=0}^{l_0+1-i}\Big(\sum\limits_{r=1}^{f(j)}\hat\alpha_{1,r,1}-\sum\limits_{r=1}^{j}\alpha_{1,l_0+2-r}\Big)
				+2{N}_{i,1}, \ \ \text{for some} \ {N}_{i,1}\in\mathbb{Z}, \ 1\leq i\leq l_0+1,
			\end{equation*}
			where $\hat\alpha_{1,r,1}=\alpha_{1,l_0+2-r}$ for $1\leq r\leq l_0+1$ and $\hat\alpha_{1,r,1}=\alpha_{1,r-l_0}$ for $l_0+2\leq r\leq 2l_0+1$.
			
			\item If $i_0+l_0=n+1$, then there exists a permutation map $f\in\mathbb{S}_{\mathbf{C}_{l_0+1}}$ such that
			\begin{equation*}
				\sigma_{u_i}=2\sum\limits_{j=0}^{i-i_0}\Big(\sum\limits_{r=1}^{f(j)}\hat\alpha_{1,r,2}-\sum\limits_{r=1}^{j}\alpha_{1,r+i_0-1}\Big)+2{N}_{i,2}, \ \ \text{for some} \ {N}_{i,2}\in\mathbb{Z}, \ i_0\leq i\leq n+1,
			\end{equation*}
			where $\hat\alpha_{1,r,2}=\alpha_{1,r+i_0-1}$ for $1\leq r\leq l_0+1$ and $\hat\alpha_{1,r,2}=\alpha_{1,2l_0+1+i_0-r}$ for $l_0+2\leq r\leq 2l_0+1$.
		\end{enumerate}
	\end{lemma}
	\begin{proof}
		(a) Denote by $\tilde{\mathbf{u}}=(\tilde{u}_1,\cdots,\tilde{u}_{l_0+1})$, where $\tilde{u}_i=u_{l_0+2-i}$ and set $\tilde{\alpha}_{\ell,i}=\alpha_{\ell,l_0+2-i}$ for $1\leq \ell\leq N$ and $i=1,\cdots,l_0+1$. Then $\tilde{\mathbf{u}}$ satisfies the $\mathbf{C}_{l_0+1}$ Toda system \eqref{Affine-Toda-Section-4-Eq-15} with $(k_{ij}^{\prime})=(c_{ij})$ and $\alpha_{\ell,i}=\tilde{\alpha}_{\ell,i}$ for $1\leq \ell\leq N$ and $1\leq i,j\leq l_0+1$, where $(c_{ij})_{(l_0+1)\times(l_0+1)}$ stands for the $\mathbf{C}$ type Cartan matrix of rank $l_0+1$. By Lemma \ref{Affine-Toda-Section-4-Lemma-11}, there exists a permutation map $f\in\mathbb{S}_{\mathbf{C}_{l_0+1}}$ such that
		\begin{equation*}
			\sigma_{\tilde{u}_i}=2\sum\limits_{j=0}^{i-1}\Big(\sum\limits_{r=1}^{f(j)}\beta_{1,r}-\sum\limits_{r=1}^{j}\beta_{1,r}\Big)+2\hat{N}_i, \ \text{for some} \ \hat{N}_i\in\mathbb{Z}, \ 1\leq i\leq l_0+1,
		\end{equation*}
		where $\beta_{\ell,r}=\beta_{\ell,2l_0+2-r}=\tilde\alpha_{\ell,r}$, $1\leq r\leq l_0+1$. Hence we conclude that
		\begin{equation*}
			\sigma_{u_{i}}=2\sum\limits_{j=0}^{l_0+1-i}\Big(\sum\limits_{r=1}^{f(j)}\hat\alpha_{1,r,1}-\sum\limits_{r=1}^{j}\hat\alpha_{1,r,1}\Big)
			+2{N}_{i,1}=2\sum\limits_{j=0}^{l_0+1-i}\Big(\sum\limits_{r=1}^{f(j)}\hat\alpha_{1,r,1}-\sum\limits_{r=1}^{j}\alpha_{1,l_0+2-r}
			\Big)+2{N}_{i,1},
		\end{equation*}
		for some $N_{i,1}\in\mathbb{Z}$, $1\leq i\leq l_0+1$, where $\hat\alpha_{1,r,1}=\alpha_{1,l_0+2-r}$ for $1\leq r\leq l_0+1$ and $\hat\alpha_{1,r,1}=\alpha_{1,r-l_0}$ for $l_0+2\leq r\leq 2l_0+1$.
		
		(b) Denote by $\tilde{\mathbf{u}}=(\tilde{u}_1,\cdots,\tilde{u}_{l_0+1})$, where $\tilde{u}_i=u_{i+i_0-1}$ and set $\tilde{\alpha}_{\ell,i}=\alpha_{\ell,i+i_0-1}$ for $1\leq \ell\leq N$ and $i=1,\cdots,l_0+1$. Then $\tilde{\mathbf{u}}$ satisfies the $\mathbf{C}_{l_0+1}$ Toda system \eqref{Affine-Toda-Section-4-Eq-15} with $(k_{ij}^{\prime})=(c_{ij})$ and $\alpha_{\ell,i}=\tilde{\alpha}_{\ell,i}$ for $1\leq \ell\leq N$ and $1\leq i,j\leq l_0+1$, where $(c_{ij})_{(l_0+1)\times(l_0+1)}$ stands for the $\mathbf{C}$ type Cartan matrix of rank $l_0+1$. By Lemma \ref{Affine-Toda-Section-4-Lemma-11}, there exists a permutation map $f\in\mathbb{S}_{\mathbf{C}_{l_0+1}}$ such that
		\begin{equation*}
			\sigma_{\tilde{u}_i}=2\sum\limits_{j=0}^{i-1}\Big(\sum\limits_{r=1}^{f(j)}\hat\alpha_{1,r,2}-\sum\limits_{r=1}^{j}\hat\alpha_{1,r,2}\Big)+2\tilde{N}_i, \ \ \text{for some} \ \tilde{N}_i\in\mathbb{Z}, \ 1\leq i\leq l_0+1,
		\end{equation*}
		where $\hat\alpha_{1,r,2}=\alpha_{1,r+i_0-1}$ for $1\leq r\leq l_0+1$ and $\hat\alpha_{1,r,2}=\alpha_{1,2l_0+1+i_0-r}$ for $l_0+2\leq r\leq 2l_0+1$. Hence
		\begin{equation*}
			\sigma_{u_i}=2\sum\limits_{j=0}^{i-i_0}\Big(\sum\limits_{r=1}^{f(j)}\hat\alpha_{1,r,2}-\sum\limits_{r=1}^{j}\hat\alpha_{1,r}\Big)+2N_{i,2}
			=2\sum\limits_{j=0}^{i-i_0}\Big(\sum\limits_{r=1}^{f(j)}\hat\alpha_{1,r,2}-\sum\limits_{r=1}^{j}\alpha_{1,r+i_0-1}\Big)+2N_{i,2},
		\end{equation*}
		for some $N_{i,2}\in\mathbb{Z}$, $i_0\leq i\leq n+1$. This completes the proof of Lemma \ref{Affine-Toda-Section-4-Lemma-12}.
	\end{proof}
	
	\begin{lemma}\label{Affine-Toda-Section-4-Lemma-13}
		Let $\bm{\sigma}\in\Gamma_{\mathbf{C}^t}(\bm{\mu})$ and $\overline{\mu}_{i}:=\mu_{i}-\frac{1}{2}\sum_{t\in I}k_{it}^{\mathbf{c}^{t}}\sigma_{t}$ for $i\in I$.
		\begin{enumerate}[(a)]
			\item  Assume that $i_0=1$. For any permutation map $f\in\mathbb{S}_{\mathbf{C}_{l_0+1}}$, we define $\bm{\sigma}_{f}=(\sigma_{f,1},\cdots,\sigma_{f,n+1})$ as
			\begin{equation*}
				\sigma_{f,i}=\left\{
				\begin{aligned}
					&\sigma_{i}+2\sum\limits_{j=0}^{l_0+1-i}\Big(\sum\limits_{r=1}^{f(j)}\hat{\overline{\mu}}_{r}-\sum\limits_{r=1}^{j}\hat{\overline{\mu}}_{r}\Big), \ && \text{for} \ 1\leq i\leq l_0+1,\\
					&\sigma_{i}, \ && \text{for} \ l_0+2\leq i\leq n+1,
				\end{aligned}
				\right.
			\end{equation*}
			where $\hat{\overline{\mu}}_{r}=\overline{\mu}_{l_0+2-r}$ for $1\leq r\leq l_0+1$ and $\hat{\overline{\mu}}_{r}=\overline{\mu}_{r-l_0}$ for $l_0+2\leq r\leq 2l_0+1$. Then $\bm{\sigma}_{f}\in\Gamma_{\mathbf{C}^t}(\bm{\mu})$.
			\item Assume that $i_0+l_0=n+1$. For any permutation map $f\in\mathbb{S}_{\mathbf{C}_{l_0+1}}$, we define $\bm{\sigma}_{f}=(\sigma_{f,1},\cdots,\sigma_{f,n+1})$ as
			\begin{equation*}
				\sigma_{f,i}=\left\{
				\begin{aligned}
					&\sigma_{i}, \ && \text{for} \ 1\leq i\leq i_0-1,\\
					&\sigma_{i}+2\sum\limits_{j=0}^{i-i_0}\Big(\sum\limits_{r=1}^{f(j)}\hat{\overline{\mu}}_{r}-\sum\limits_{r=1}^{j}\hat{\overline{\mu}}_{r}\Big), \ && \text{for} \ i_0\leq i\leq n+1,
				\end{aligned}
				\right.
			\end{equation*}
			where $\hat{\overline{\mu}}_{r}=\overline{\mu}_{r+i_0-1}$ for $1\leq r\leq l_0+1$ and $\hat{\overline{\mu}}_{r}=\overline{\mu}_{2l_0+1+i_0-r}$ for $l_0+2\leq r\leq 2l_0+1$. Then $\bm{\sigma}_{f}\in\Gamma_{\mathbf{C}^t}(\bm{\mu})$.
		\end{enumerate}
	\end{lemma}
	\begin{proof}
		We prove it similarly to \cite[Lemma 3.6 and Lemma 8.6]{Lin-Yang-Zhong-2020}. Let $f_i$'s be the simple permutation maps defined in \eqref{5.simple}, one can easily check that
		\begin{equation*}
			\begin{aligned}
				&\text{(a).} \ &&\bm{\sigma}_{f\circ f_i}=\mathfrak{R}_{l_0+1-i}\bm{\sigma}_{f}, \ &&\text{for} \ 0\leq i\leq l_0,\\
				&\text{(b).} \ &&\bm{\sigma}_{f\circ f_i}=\mathfrak{R}_{i+i_0}\bm{\sigma}_{f}, \ &&\text{for} \ 0\leq i\leq l_0.
			\end{aligned}
		\end{equation*}
		Together with the fact that $f$ can be viewed as a composition of some simple permutations, we finish the proof of Lemma \ref{Affine-Toda-Section-4-Lemma-13}.
	\end{proof}
	
	Now we sketch out the blow-up procedure for affine $\mathbf{C}^t$ type Toda system. Let $\mathbf{u}^k$ be a sequence of blow-up solutions of \eqref{Affine-Toda-system-C_n^1-form} satisfying \eqref{Affine-Toda-3Conditions}. The notations used in this part are consistent with those in Section \ref{Affine-Toda-Section-3}. We first calculate the local mass of $\mathbf{u}^k$ in the bubbling areas $B(x^k_t,\tau^k_t)$, i.e.,
	\begin{equation}\label{Affine-Toda-Section-4-Eq-9}
		\hat{\bm{\sigma}}(B(\mathbf{x}_t,\bm{\tau}_t))\in\Gamma_{\mathbf{C}^t}(1,\cdots,1),
	\end{equation}
	where $(\mathbf{x}_t,\bm{\tau}_{t})$ stands for the sequence $\{(x^k_t,\tau^k_t)\}$. During the procedure, we start from disk $B(x^k_t,l^k_t)$ and choose some sequence $s_k\in(l^k_t,\tau^k_t)$ satisfying $l^k_t\ll s_k\ll\tau^k_t$ such that
	\begin{enumerate}[(1).]
		\item some components of $\mathbf{u}^k$ have slow decay on $\partial B(x^k_t,s_k)$,
		\item $\hat{\sigma}_i(B(\mathbf{x}_t,\mathbf{s}))=\hat{\sigma}_i(B(\mathbf{x}_t,\mathbf{l}_t))$ for $i\in I$, where $(\mathbf{x}_t,\mathbf{s})$ stands for the sequence $\{(x^k_t,s_k)\}$.
	\end{enumerate}
	We scale $\mathbf{u}^k$ by $v^k_i(y)=u^k_i(x_t^k+s_ky)+2\log{s_k}$ for $i\in I$ and denote by
	\begin{equation*}
		J=\big\{i\in I \mid u^k_i \ \text{has slow decay on} \ \partial B(x_t^k,s_k)\big\}\subsetneqq I.
	\end{equation*}
	Naturally, $v^k_i(y)\rightarrow-\infty$ in $L_{\mathrm{loc}}^{\infty}(\mathbb{R}^2)$ for $i\in I\setminus J$, and $v^k_i(y)\rightarrow v_i(y)$ in $C_{\mathrm{loc}}^2(\mathbb{R}^2)$ for $i\in J$, where $v_i(y)$ satisfies the system
	\begin{equation*}
		\Delta v_i(y)+\sum\limits_{j\in J}k_{ij}^{\mathbf{c}^{t}}e^{v_j(y)}=4\pi\alpha_i^{*}\delta_0 \ \text{in} \ \mathbb{R}^2, \ \int_{\mathbb{R}^2}e^{v_i(y)}\mathrm{d}y<+\infty, \ \ \forall \ i\in J.
	\end{equation*}
	Here, $\alpha_i^{*}=-\frac{1}{2}\sum_{j\in I}k_{ij}^{\mathbf{c}^{t}}\hat{\sigma}_j(B(\mathbf{x}_t,\mathbf{l}_t))>-1$ for $i\in J$. Next we would encounter the same four alternatives \textbf{(I)}, \textbf{(II)}, \textbf{(III)} and \textbf{(IV)} established in Proposition \ref{Affine-Toda-Section-4-Proposition-1}-(4). Applying Theorem \ref{Affine-Toda-Section-4-Theorem-10}, we get the desired conclusion \eqref{Affine-Toda-Section-4-Eq-9}.
	
	Then, we calculate the local mass of groups containing adjacent elements. As an illustration, for the group $S_j^k=\{x_{j,1}^k,\cdots,x_{j,m_j}^k\}\subseteq\Sigma_k$ with some $1\leq m_j\leq m$, letting $\tau_{S^k_j}^k$ be defined as $\tau_{S^k_j}^k=\frac12\mathrm{dist}(x_{j,1}^k,\Sigma_k\setminus S_j^k)$. We deduce that $\hat{\bm{\sigma}}(B(\mathbf{x}_{j,1},\bm{\tau}_{S_j^k}))\in\Gamma_{\mathbf{C}^t}(1,\cdots,1)$ by counting the contribution from each bubbling disk containing $x_{j,l}^k$, $l=1,\cdots,m_j$. In the process, we also encounter the same four alternatives \textbf{(I)}, \textbf{(II)}, \textbf{(III)} and \textbf{(IV)} as in Proposition \ref{Affine-Toda-Section-4-Proposition-1}-(4). Theorem \ref{Affine-Toda-Section-4-Theorem-10} is used for dealing with the corresponding problems.
	
	Furthermore, we collect $\{0\}$ and "closer" groups $S^k_j$, $1\leq j\leq m_0$. The following scenarios would arise: for the local mass $\bm{\sigma}$ in the last step, let $v^k_i(y)=u^k_i(x^k+s_ky)+2\log{s_k}$ for $i\in I$ and denote by $J=\left\{i\in I \mid u^k_i \ \text{has slow decay on} \ \partial B(x^k,s_k)\right\}\subsetneqq I$. Then, $v^k_i(y)\rightarrow-\infty$ for $i\in I\setminus J$ in $L^{\infty}_{\mathrm{loc}}(\mathbb{R}^2)$ and $v^k_i(y)\rightarrow v_i(y)$ in $C_{\mathrm{loc}}^2(\mathbb{R}^2)$ for $i\in J$, where $v_i(y)$ satisfies the system
	\begin{equation*}
		\Delta v_i(y)+\sum\limits_{j\in J}k_{ij}^{\mathbf{c}^t}e^{v_j(y)}=4\pi\alpha_i^{*}\delta_0+4\pi\sum\limits_{t=1}^{N}m_{it}\delta_{p_t} \ \text{in} \ \mathbb{R}^2, \ \int_{\mathbb{R}^2}e^{v_i(y)}\mathrm{d}y<+\infty, \ \ \forall \ i\in J,
	\end{equation*}
	where $0\neq p_{t}\in\mathbb{R}^2$, $m_{it}\in\mathbb{N}$ and $\alpha_i^{*}=\alpha_i-\frac{1}{2}\sum_{j\in I}k_{ij}^{\mathbf{c}^t}\sigma_j>-1$ for $i\in J$. Denote by $\mu^{*}_{i}=1+\alpha_{i}^{*}$ for $i\in I$. Set
	\begin{equation*}
		\sigma^{*}_{j}=
		\left\{
		\begin{array}{ll}
			\sigma_j, \ &\text{if} \ j\in I\setminus J,\\
			\sigma_j+\frac{1}{2\pi}\int_{\mathbb{R}^2} e^{v_j(y)}\mathrm{d}y, \ &\text{if} \ j\in J.
		\end{array}\right.
	\end{equation*}
	To complete the proof of Theorem \ref{Affine-Toda-Section-1-Theorem-2}, we need to demonstrate a similar result for $\bm{\sigma}^*$ as stated in Lemma \ref{Affine-Toda-Section-3-Lemma-1}, namely,
	
	\begin{lemma}\label{Affine-Toda-Section-4-Lemma-14}
		Let the index set $I$ be decomposed into one form of the four alternatives \textbf{(I)}, \textbf{(II)}, \textbf{(III)} and \textbf{(IV)} as in Proposition \ref{Affine-Toda-Section-4-Proposition-1}-(4). Suppose that $\hat{\bm{\sigma}}\in\Gamma_{\mathbf{C}^t}(\bm{\mu})$ satisfying $\sigma_{i}=\hat{\sigma}_i+2n_i$ and $n_i\in\mathbb{Z}$. Then $\sigma^*_i=\hat{\sigma}^*_i+2n^*_i$ with $\hat{\bm{\sigma}}^{*}\in\Gamma_{\mathbf{C}^t}(\bm{\mu})$ and $n^*_i\in\mathbb{Z}$.
	\end{lemma}
	\begin{proof}
		Applying Lemma \ref{Affine-Toda-Section-4-Lemma-12} and Lemma \ref{Affine-Toda-Section-4-Lemma-13}, one can prove it via a similar argument as in Lemma \ref{Affine-Toda-Section-3-Lemma-1}.
	\end{proof}
	
	Finally, we give the proof of Theorem \ref{Affine-Toda-Section-1-Theorem-2}.
	\begin{proof}[Proof of Theorem \ref{Affine-Toda-Section-1-Theorem-2}]
		Similar to the proof of \cite[Theorem 1.1]{Cui-Wei-Yang-Zhang-2022} and Theorem \ref{Affine-Toda-Section-1-Theorem-1}, one can get the desired conclusion by Lemma \ref{Affine-Toda-Section-4-Lemma-14}.
	\end{proof}

\end{document}